\documentclass[10pt]{article}

\usepackage[margin=1in]{geometry}  
\usepackage{graphicx}              
\usepackage{amsmath}             
\usepackage{amsfonts}             
\usepackage{amsthm}               
\usepackage{amsbsy}
\usepackage{amssymb}
\usepackage{amsthm}

\usepackage{fancyhdr}

\newtheorem{thm}{Theorem}[subsection]
\newtheorem{lem}[thm]{Lemma}
\newtheorem{prop}[thm]{Proposition}
\newtheorem{cor}[thm]{Corollary}

\newtheorem{thm2}{Theorem}[section]
\newtheorem{lem2}[thm2]{Lemma}

\newtheorem{cor2}[thm2]{Corollary}

\theoremstyle{definition}
\newtheorem{defn}[thm]{Definition}
\newtheorem{nota}[thm]{Notation}
\newtheorem{rem}[thm]{Remarks}
\newtheorem{cons}[thm]{Construction}
\newtheorem{disc}[thm]{Discussion}

\newtheorem{rem2}[thm2]{Remarks}

\begin{document}

\nocite{*}

\title{On a Notion of Exactness for Reduced Free Products of C$^*$-Algebras}

\author{Paul Skoufranis\thanks{This research was supported in part by NSF grant DMS-090076 and by NSERC PGS.
 }}

\maketitle

\begin{abstract}
We will study some modifications to the notion of an exact C$^*$-algebra by replacing the minimal tensor product with the reduced free product.  First we will demonstrate how the reduced free product of a short exact sequence of C$^*$-algebras with another C$^*$-algebra may be taken.  It will then be demonstrated that this operation preserves exact sequences.  We will also establish that adjoining arbitrary $k$-tuples of operators in a free way behaves well with respect to taking ultrapowers. 
\end{abstract}

\section{Introduction}

The notion of an exact C$^*$-algebra has played a fundamental role in the theory of C$^*$-algebras and has been well-studied by Kirchberg, Wassermann, and others (see \cite{Ki} and \cite{Wa}).  Exact C$^*$-algebras are generally well-behaved and many of the common and interesting examples of C$^*$-algebras are exact.  In addition, the property that a C$^*$-algebra is exact is preserved under many common operations such as taking subalgebras, taking direct sums, taking minimal tensor products, and taking reduced free products (for example, see \cite{BO} and the references therein).
\par
Over the years many equivalent definitions of an exact C$^*$-algebra have been developed and the most common are listed in the following theorem.  In this theorem (and for the rest of this paper) $\mathcal{B}(\mathcal{H})$ will denote the space of bounded linear maps on a Hilbert space $\mathcal{H}$, $\mathcal{M}_n(\mathbb{C})$ will denote the $n \times n$ matrices with complex entries, $\mathcal{A} \odot \mathcal{B}$ will denote the algebraic tensor product of two algebras $\mathcal{A}$ and $\mathcal{B}$, and $\mathfrak{A} \otimes_{\min} \mathfrak{B}$ will denote the minimal tensor product of two C$^*$-algebras $\mathfrak{A}$ and $\mathfrak{B}$.
\begin{thm2}[Due to Kirchberg, Wassermann, and others; see \cite{BO} for the proof of the first four equivalences]
\label{exactnessTheorem}
Let $\mathfrak{B}$ be a C$^*$-algebra.  Then the following are equivalent:
\begin{enumerate}
  \item There exists a Hilbert space $\mathcal{H}$, a faithful representation $\sigma : \mathfrak{B} \to \mathcal{B}(\mathcal{H})$, and nets $(\varphi_\lambda : \mathfrak{B} \to \mathcal{M}_{n_\lambda}(\mathbb{C}))_\Lambda$ and $(\psi_\lambda : \mathcal{M}_{n_\lambda}(\mathbb{C}) \to \mathcal{B}(\mathcal{H}))_\Lambda$ of contractive, completely positive maps such that
  \[
  \lim_\Lambda \left\|\sigma(B) - \psi_\lambda(\varphi_\lambda(B))\right\| = 0
  \]
  for all $B \in \mathfrak{B}$.
  \item For every Hilbert space $\mathcal{H}$ and faithful representation $\sigma : \mathfrak{B} \to \mathcal{B}(\mathcal{H})$ there exists nets $(\varphi_\lambda : \mathfrak{B} \to \mathcal{M}_{n_\lambda}(\mathbb{C}))_\Lambda$ and $(\psi_\lambda : \mathcal{M}_{n_\lambda}(\mathbb{C}) \to \mathcal{B}(\mathcal{H}))_\Lambda$ of contractive, completely positive maps such that
  \[
  \lim_\Lambda \left\|\sigma(B) - \psi_\lambda(\varphi_\lambda(B))\right\| = 0
  \]
  for all $B \in \mathfrak{B}$.
	\item For every exact sequence of C$^*$-algebras $0\rightarrow \mathfrak{J}\stackrel{i}{\rightarrow}\mathfrak{A} \stackrel{q}{\rightarrow}(\mathfrak{A}/\mathfrak{J})\rightarrow0$ the sequence
\[
0\rightarrow \mathfrak{J}\otimes_{\min} \mathfrak{B}\stackrel{i \otimes Id_\mathfrak{B}}{\rightarrow}\mathfrak{A} \otimes_{\min} \mathfrak{B} \stackrel{q\otimes Id_\mathfrak{B}}{\rightarrow}(\mathfrak{A}/\mathfrak{J})\otimes_{\min} \mathfrak{B}\rightarrow0
\]
is exact.
\item For any sequence $(\mathfrak{A}_n)_{n\geq1}$ of unital C$^*$-algebras the $^*$-homomorphism
\[
\left( \frac{\prod_{n\geq 1} \mathfrak{A}_n}{\bigoplus_{n\geq1} \mathfrak{A}_n}\right) \odot \mathfrak{B} \to \frac{\left(\prod_{n\geq 1} \mathfrak{A}_n\right) \otimes_{\min} \mathfrak{B} }{\left(\bigoplus_{n\geq1}  \mathfrak{A}_n\right)\otimes_{\min}\mathfrak{B}}
\]
defined by
\[
\left((A_n)_{n\geq 1} + \bigoplus_{n\geq1} \mathfrak{A}_n \right) \otimes B \mapsto (A_n )_{n\geq 1} \otimes B + \left(\bigoplus_{n\geq1}  \mathfrak{A}_n\right)\otimes_{\min}\mathfrak{B}
\]
is continuous with respect to the minimal tensor norm on $\left(\frac{\prod_{n\geq 1}\mathfrak{A}_n}{\bigoplus_{n\geq1}\mathfrak{A}_n}\right) \odot \mathfrak{B}$.
\item If $\mathfrak{A}_n$ and $\mathfrak{A}$ are unital C$^*$-algebras, $k \in \mathbb{N}$, $A_1, \ldots, A_k \in \mathfrak{A}$, and $\{A_{i,n}\}^k_{i=1} \subseteq \mathfrak{A}_n$ are such that $\left\|p(A_1, \ldots, A_k)\right\|_\mathfrak{A} = \limsup_{n\to\infty} \left\|p(A_{1,n}, \ldots, A_{k,n})\right\|_{\mathfrak{A}_n}$ for every polynomials $p$ in $k$ non-commutating variables and their complex conjugates, then for all $B_1, \ldots, B_k \in \mathfrak{B}$
\[
\left\|\sum^k_{i=1} A_i \otimes B_i\right\|_{\mathfrak{A} \otimes_{\min}\mathfrak{B}} = \limsup_{n\to\infty} \left\|\sum^k_{i=1} A_{i,n} \otimes B_i\right\|_{\mathfrak{A}_n \otimes_{\min} \mathfrak{B}}.
\]
\end{enumerate}
If any of the above conditions hold then $\mathfrak{B}$ is said to be an exact C$^*$-algebra.
\end{thm2}
\par
As the proof that the fifth statement of Theorem \ref{exactnessTheorem} is equivalent to the others is not standard, we present the proof here.
\begin{proof}[Proof that the fifth statement of Theorem \ref{exactnessTheorem} is equivalent to the fourth statement.]
Let $(\mathfrak{A}_n)_{n\geq1}$ be a sequence of unital C$^*$-algebras and let
\[
T = \sum^k_{i=1} A_i \otimes B_i \in \left( \frac{\prod_{n\geq 1} \mathfrak{A}_n}{\bigoplus_{n\geq 1} \mathfrak{A}_n}\right) \odot \mathfrak{B}
\]
be arbitrary.  For all $i \in \{1,\ldots, k\}$ there exists $A_{i,n} \in \mathfrak{A}_n$ such that 
\[
\left\|p(A_1, \ldots, A_k)\right\|_\mathfrak{A} = \limsup_{n\to\infty} \left\|p(A_{1,n}, \ldots, A_{k,n})\right\|_{\mathfrak{A}_n}
\]
for every polynomials $p$ in $k$ non-commutating variables and their complex conjugates (that is, choose a lifting of each $A_i$).  If $\mathfrak{B}$ satisfies the fifth statement of Theorem \ref{exactnessTheorem} then
\begin{eqnarray} 
\left\|T\right\|_{\left( \frac{\prod_{n\geq 1} \mathfrak{A}_n}{\bigoplus_{n\geq 1} \mathfrak{A}_n }\right) \otimes_{\min} \mathfrak{B}}\!\!\!\!&=&\!\!\!\! \limsup_{n\to\infty} \left\|\sum^k_{i=1} A_{i,n} \otimes B_i\right\|_{\mathfrak{A}_n \otimes_{\min} \mathfrak{B}} \nonumber\\
&=&\!\!\!\! \left\| \left(\sum^k_{i=1} A_{i,n}\otimes B_i\right)_{n\geq 1} + \bigoplus_{n\geq 1} \left(\mathfrak{A}_n \otimes_{\min} \mathfrak{B}\right)\right\|_{\frac{\prod_{n\geq 1} (\mathfrak{A}_n \otimes_{\min} \mathfrak{B})}{\bigoplus_{n\geq 1} (\mathfrak{A}_n \otimes_{\min} \mathfrak{B})}}  \nonumber\\
&=&\!\!\!\! \left\| \sum^k_{i=1}\left( A_{i,n}\right)_{n\geq 1} \otimes B_i + \left(\bigoplus_{n\geq 1} \mathfrak{A}_n\right) \otimes_{\min} \mathfrak{B}\right\|_{\frac{\left(\prod_{n\geq 1}\mathfrak{A}_n\right) \otimes_{\min} \mathfrak{B}}{\left(\bigoplus_{n\geq 1} \mathfrak{A}_n \right)\otimes_{\min} \mathfrak{B}}}  \nonumber
\end{eqnarray}
where the last equality follows from Lemma \ref{tensorImbed}.  Thus the fifth statement of Theorem \ref{exactnessTheorem} implies the fourth statement.
\par
For the other direction, suppose $\mathfrak{B}$ satisfies the fourth statement in Theorem \ref{exactnessTheorem}.  Let $\mathfrak{A}_n$ and $\mathfrak{A}$ be unital C$^*$-algebras, let $k \in \mathbb{N}$, let $A_1, \ldots, A_k \in \mathfrak{A}$, and let $\{A_{i,n}\}^k_{i=1} \subseteq \mathfrak{A}_n$ be such that 
\[
\left\|p(A_1, \ldots, A_k)\right\|_\mathfrak{A} = \limsup_{n\to\infty}\left\|p(A_{1,n}, \ldots, A_{k,n})\right\|_{\mathfrak{A}_n}
\]
for every non-commutative polynomials $p$ in $k$-variables and their complex conjugates.  We may assume that $\mathfrak{A} =$ $^*$-$\overline{Alg(A_1, \ldots, A_k)}$ by properties of the minimal tensor product.
\par
 Fix $B_1, \ldots, B_k \in \mathfrak{B}$.  The fourth equivalence in Theorem \ref{exactnessTheorem} implies that the canonical inclusion
\[
\left(\frac{\prod_{n\geq 1} \mathfrak{A}_n}{\bigoplus_{n\geq 1} \mathfrak{A}_n}\right) \odot \mathfrak{B} \to \frac{\left(\prod_{n\geq 1} \mathfrak{A}_n\right) \otimes_{\min} \mathfrak{B}}{\left(\bigoplus_{n\geq 1} \mathfrak{A}_n\right) \otimes_{\min} \mathfrak{B}}
\]
is continuous with respect to the minimal tensor product and extends to an injective inclusion on the minimal tensor product.  By the assumptions on $\mathfrak{A}$, $\mathfrak{A} \subseteq (\prod_{n\geq 1} \mathfrak{A}_n)/(\bigoplus_{n\geq 1} \mathfrak{A}_n)$ via the identification of $A_i$ with $(A_{i,n})_{n\geq 1} + \bigoplus_{n\geq 1} \mathfrak{A}_n$.  Thus
\begin{eqnarray} 
\left\|\sum^k_{i=1} A_i \otimes B_i\right\|_{\mathfrak{A} \otimes_{\min}\mathfrak{B}}\!\!\!\!&=&\!\!\!\!  \left\|\sum^k_{i=1} \left((A_{i,n})_{n\geq 1} + \left(\bigoplus_{n\geq 1} \mathfrak{A}_n\right)\right) \otimes B_i\right\|_{\left(\frac{\prod_{n\geq 1} \mathfrak{A}_n}{\bigoplus_{n\geq 1} \mathfrak{A}_n}\right) \otimes_{\min}\mathfrak{B}}\nonumber\\
&=&\!\!\!\! \left\|\sum^k_{i=1} (A_{i,n})_{n\geq 1}\otimes B_i + \left(\bigoplus_{n\geq 1} \mathfrak{A}_n\right) \otimes_{\min} \mathfrak{B}  \right\|_{\frac{\left(\prod_{n\geq 1} \mathfrak{A}_n\right) \otimes_{\min} \mathfrak{B}}{\left(\bigoplus_{n\geq 1} \mathfrak{A}_n\right) \otimes_{\min} \mathfrak{B}}} \nonumber\\
&=&\!\!\!\! \left\|\sum^k_{i=1} (A_{i,n})_{n\geq 1} \otimes B_i + \bigoplus_{n\geq 1} \left(\mathfrak{A}_n \otimes_{\min} \mathfrak{B} \right) \right\|_{\frac{\prod_{n\geq 1} \left( \mathfrak{A}_n\otimes_{\min} \mathfrak{B}\right) }{\bigoplus_{n\geq 1} \left(\mathfrak{A}_n \otimes_{\min} \mathfrak{B}\right)}} \nonumber
\end{eqnarray}
(where the last equality follows from Lemma \ref{tensorImbed}) so 
\[
\left\|\sum^k_{i=1} A_i \otimes B_i\right\|_{\mathfrak{A} \otimes_{\min}\mathfrak{B}} = \limsup_{n\to\infty} \left\|\sum^k_{i=1} A_{i,n} \otimes B_i\right\|_{\mathfrak{A}_n \otimes_{\min} \mathfrak{B}}
\]
as desired.  
\end{proof}
\begin{lem2}
\label{tensorImbed}
For any C$^*$-algebra $\mathfrak{B}$ and any sequence of unital C$^*$-algebras $(\mathfrak{A}_n)_{n\geq1}$ there exists an injective $^*$-homomorphism
\[
\Phi : \frac{\left(\prod_{n\geq 1} \mathfrak{A}_n\right) \otimes_{\min}\mathfrak{B}}{\left(\bigoplus_{n\geq 1} \mathfrak{A}_n\right) \otimes_{\min}\mathfrak{B}} \to \frac{\prod_{n\geq 1} (\mathfrak{A}_n\otimes_{\min} \mathfrak{B})}{\bigoplus_{n\geq 1} (\mathfrak{A}_n\otimes_{\min} \mathfrak{B})}
\]
defined by 
\[
\Phi\left( (A_n)_{n\geq 1} \otimes B + \left(\bigoplus_{n\geq 1} \mathfrak{A}_n\right) \otimes_{\min}\mathfrak{B} \right) = (A_n \otimes B)_{n\geq 1} + \bigoplus_{n\geq 1} (\mathfrak{A}_n\otimes_{\min} \mathfrak{B})
\]
for all $(A_n)_{n\geq1} \in \prod_{n\geq1}\mathfrak{A}_n$ and $B \in \mathfrak{B}$.
\end{lem2}
\begin{proof}
Consider the map $\pi_0 : \left(\prod_{n\geq1} \mathfrak{A}_n \right) \odot \mathfrak{B} \to \prod_{n\geq 1} (\mathfrak{A}_n \otimes_{\min} \mathfrak{B})$ defined by
\[
\pi_0((A_n)_{n\geq 1} \otimes B) = (A_n \otimes B)_{n\geq 1}.
\]
It is easy to verify that $\pi_0$ is well-defined, continuous, and isometric with respect to the minimal tensor products and thus induces a injective $^*$-homomorphism 
\[
\pi : \left(\prod_{n\geq1} \mathfrak{A}_n\right) \otimes_{\min} \mathfrak{B} \to \prod_{n\geq 1} (\mathfrak{A}_n \otimes_{\min} \mathfrak{B}).
\]
Notice 
\[
\pi\left( \left(\bigoplus_{n\geq1} \mathfrak{A}_n\right) \otimes_{\min} \mathfrak{B} \right) \subseteq \bigoplus_{n\geq 1} (\mathfrak{A}_n \otimes_{\min} \mathfrak{B})
\]
(as this is clearly true of elementary tensors and thus the closure of the span of elementary tensors).  Therefore the $^*$-homomorphism 
\[
\Phi : \frac{\left(\prod_{n\geq 1} \mathfrak{A}_n\right) \otimes_{\min}\mathfrak{B}}{\left(\bigoplus_{n\geq 1} \mathfrak{A}_n\right) \otimes_{\min}\mathfrak{B}} \to \frac{\prod_{n\geq 1} (\mathfrak{A}_n\otimes_{\min} \mathfrak{B})}{\bigoplus_{n\geq 1} (\mathfrak{A}_n\otimes_{\min} \mathfrak{B})}
\]
as described in the statement of the lemma exists.
\par
To see $\Phi$ is injective, suppose $T \in \left(\prod_{n\geq 1} \mathfrak{A}_n \right) \otimes_{\min}\mathfrak{B}$ and $\pi(T) \in \bigoplus_{n\geq 1} (\mathfrak{A}_n \otimes_{\min} \mathfrak{B})$.  Let $(B_\lambda)_\Lambda$ be a C$^*$-bounded approximate identity for $\mathfrak{B}$.  For each $n \in \mathbb{N}$ and $\lambda \in \Lambda$ let
\[
E_{n, \lambda} := (I_{\mathfrak{A}_1}, I_{\mathfrak{A}_2}, \cdots, I_{\mathfrak{A}_n}, 0, 0, \cdots) \otimes B_\lambda \in \left(\bigoplus_{n\geq 1} \mathfrak{A}_n\right) \otimes_{\min}\mathfrak{B}.
\]
Define a partial ordering on $\mathbb{N} \times \Lambda$ by $(n, \lambda) \leq (m, \lambda')$ if and only if $n \leq m$ and $\lambda \leq \lambda'$.  It is easy to verify that $(E_{n,\lambda})_{\mathbb{N} \times \Lambda}$ is a C$^*$-bounded approximate identity for $\left(\bigoplus_{n\geq 1} \mathfrak{A}_n\right) \otimes_{\min}\mathfrak{B}$ and $(\pi(E_{n,\lambda}))_{\mathbb{N} \times \Lambda}$ is a C$^*$-bounded approximate identity for $\bigoplus_{n\geq 1} (\mathfrak{A}_n \otimes_{\min}\mathfrak{B})$.  Whence 
\[
\lim_{\mathbb{N} \times \Lambda} \left\|\pi(TE_{n,\lambda} - T)\right\| = \lim_{\mathbb{N} \times \Lambda}\left\|\pi(T)\pi(E_{n,\lambda}) - \pi(T)\right\| = 0.
\]
Since $\pi$ is isometric, $\lim_{\mathbb{N} \times \Lambda}\left\|TE_{n,\lambda} - T\right\| = 0$ so
\[
T = \lim_{\mathbb{N} \times \Lambda} TE_{n,\lambda} \in \left(\bigoplus_{n\geq 1} \mathfrak{A}_n \right) \otimes_{\min}\mathfrak{B}.
\]
Thus $ker(\pi) = \left( \bigoplus_{n\geq1} \mathfrak{A}_n\right) \otimes_{\min} \mathfrak{B}$ so $\Phi$ is injective. 
\end{proof}
In this paper we will analyze how the third and fifth equivalences in Theorem \ref{exactnessTheorem} can be adapted to the context of reduced free products.  In Section 2 we will modify the third equivalence in Theorem \ref{exactnessTheorem} by replacing the minimal tensor product with the reduced free product.  First we will demonstrate a way to take the reduced free product of a short exact sequence of C$^*$-algebras against a fixed C$^*$-algebra. Our main result is that every C$^*$-algebra is `freely exact'; that is, taking the reduced free product of a short exact sequence of C$^*$-algebras against a fixed C$^*$-algebra preserves exactness.  This will be accomplished by embedding these short sequences into a short exact sequence involving Toeplitz-Pimsner algebras and restricting back to our original sequence.
\par
In Section 3 of this paper we will analyze how the fifth equivalence of Theorem \ref{exactnessTheorem} can be adapted to the context of reduced free products.  It will be demonstrated that the conclusion of fifth equivalence of Theorem \ref{exactnessTheorem} holds when the minimal tensor product is replaced with the reduced free product for any C$^*$-algebra. This will be accomplished by first proving the result for the C$^*$-algebra generated by a finite number of free creation operators (previously proven in the appendix of \cite{Ma} due to Shlyakhtenko), then for exact C$^*$-algebras, and finally for arbitrary C$^*$-algebras.
\par
In Section 4 we will show if the nuclear embeddings in the second equivalence of Theorem \ref{exactnessTheorem} are required to be state-preserving, then nothing new is gained.  This will be accomplished by showing that if a unital, completely positive maps on a C$^*$-subalgebra $\mathfrak{A}$ of $\mathfrak{B}$ preserves a state then it can be extended in a state-preserving way to a unital, completely positive map on $\mathfrak{B}$ and by using arguments similar to those found in \cite{Oz}.

\section{Short Sequences of Reduced Free Products}
\label{sect:SSoRFP}

\subsection{Notation and a Construction}

The purpose of this section is to replace the tensor products with reduced free products in the third equivalence in Theorem \ref{exactnessTheorem} and examine the result. We begin by describing a reduced free product analog of taking the tensor product of an exact sequence with a fixed C$^*$-algebra.  Most typical results for the reduced free product of C$^*$-algebras requires the states used in the construction to have faithful GNS representations and thus hinders the consideration of quotient maps.  The solutions is to go straight to the construction of the reduced free product of two C$^*$-algebras.
\begin{nota}
\label{freeprod}
For $i\in\{1,2\}$ let $\mathfrak{A}_i$ be unital C$^*$-algebras, let $\pi_i : \mathfrak{A}_i \to \mathcal{B}(\mathcal{H}_i)$ be faithful, unital representations, and let $\xi_i \in \mathcal{H}_i$ be unit vectors.  We define the free product $(\mathcal{H}_1, \xi_1) \ast (\mathcal{H}_2, \xi_2)$ of the Hilbert spaces $(\mathcal{H}_1, \xi_1)$ and $(\mathcal{H}_2, \xi_2)$ in the standard way:  if $\mathcal{H}^0_i = \mathcal{H}_i \ominus \mathbb{C} \xi_i$ then
\[
(\mathcal{H}_1, \xi_1) \ast (\mathcal{H}_2, \xi_2) := \mathbb{C} \xi_0 \oplus \left( \bigoplus_{n\geq 1} \left( \bigoplus_{\begin{array}{cc}
\{i_k\}^n_{k=1} \subseteq \{1,2\},  \\
 i_k \neq i_{k+1} \mbox{ for }k \in \{1, \ldots, n-1\}
\end{array}
} \mathcal{H}^0_{i_1} \otimes \mathcal{H}^0_{i_2} \otimes \cdots \otimes \mathcal{H}^0_{i_n}  \right)\right).
\]
The vector $\xi_0$ is called the distinguished unit vector (and may be denoted $\xi_1 \ast \xi_2$).  
\par
There is a canonical action of each $\mathfrak{A}_i$ on $(\mathcal{H}_1, \xi_1) \ast (\mathcal{H}_2, \xi_2)$.  To define this action let
\[
\mathcal{H}(i) := \mathbb{C} \xi_0 \oplus \left( \bigoplus_{n\geq 1} \left( \bigoplus_{\begin{array}{cc}
\{i_k\}^n_{k=1} \subseteq \{1,2\}, i_1 \neq i,  \\
 i_k \neq i_{k+1} \mbox{ for }k\in\{1, \ldots, n-1\}
\end{array}} \mathcal{H}^0_{i_1} \otimes \mathcal{H}^0_{i_2} \otimes \cdots \otimes \mathcal{H}^0_{i_n}  \right)\right)
\]
for $i \in \{1,2\}$.  Then there exists a canonical isomorphism $U_i : \mathcal{H}_i \otimes \mathcal{H}(i) \to (\mathcal{H}_1, \xi_1) \ast (\mathcal{H}_2, \xi_2)$ defined by
\[
U_i : \left\{\begin{array}{ll}
\mathbb{C} \xi_i \otimes \mathbb{C} \xi_0 \\
\mathcal{H}_i^0 \otimes \mathbb{C} \xi_0 \\
\mathbb{C} \xi_i \otimes \mathcal{H}^0_{i_1} \otimes \mathcal{H}^0_{i_2} \otimes \cdots \otimes \mathcal{H}^0_{i_n}\\
\mathcal{H}_i^0 \otimes \mathcal{H}^0_{i_1} \otimes \mathcal{H}^0_{i_2} \otimes \cdots \otimes \mathcal{H}^0_{i_n}
\end{array} 
\stackrel{\simeq}{\rightarrow}
\begin{array}{ll}
 \mathbb{C} \xi_0 \\
\mathcal{H}_i^0  \\
\mathcal{H}^0_{i_1} \otimes \mathcal{H}^0_{i_2} \otimes \cdots \otimes \mathcal{H}^0_{i_n} \\
\mathcal{H}_i^0 \otimes \mathcal{H}^0_{i_1} \otimes \mathcal{H}^0_{i_2} \otimes \cdots \otimes \mathcal{H}^0_{i_n}
\end{array} 
\right.
\]
where $U_i$ is the canonical isomorphism in each of the four parts listed.  We define the action of $\mathfrak{A}_i$ on $(\mathcal{H}_1, \xi_1) \ast (\mathcal{H}_2, \xi_2)$ by $A\zeta := U(\pi_i(A) \otimes Id) U^*\zeta$ for all $A \in \mathfrak{A}_i$ and for all $\zeta \in (\mathcal{H}_1, \xi_1) \ast (\mathcal{H}_2, \xi_2)$.  In particular, the action of an element $A \in \mathfrak{A}_i$ on $(\mathcal{H}_1, \xi_1) \ast (\mathcal{H}_2, \xi_2)$ is given by
\[
A(\xi_0) = \langle \pi_i(A)\xi_i, \xi_i\rangle_{\mathcal{H}_i} \xi_0 \oplus  (\pi_i(A)\xi_i - \langle \pi_i(A)\xi_i, \xi_i\rangle_{\mathcal{H}_i} \xi_i) \in \mathbb{C}\xi_0 \oplus \mathcal{H}_i^0,
\]
for all $\zeta_1 \otimes \zeta_2 \otimes \cdots \otimes \zeta_n \in \mathcal{H}^0_{i_1} \otimes \mathcal{H}^0_{i_2} \otimes \cdots \otimes \mathcal{H}^0_{i_n}$ where $i_1 = i$
\[
A(\zeta_1 \otimes \zeta_2 \otimes \cdots \otimes \zeta_n) = (\langle \pi_i(A)\zeta_1, \xi_i\rangle_{\mathcal{H}_i}\zeta_2 \otimes \cdots \otimes \zeta_n) \oplus ((\pi_i(A)\zeta_1 - \langle \pi_i(A)\zeta_1, \xi_i\rangle_{\mathcal{H}_i} \xi_i) \otimes \zeta_2 \otimes \cdots \otimes \zeta_n)
\]
which is an element of $(\mathcal{H}^0_{i_2} \otimes \cdots \otimes \mathcal{H}^0_{i_n}) \oplus (\mathcal{H}^0_{i_1} \otimes \mathcal{H}^0_{i_2} \otimes \cdots \otimes \mathcal{H}^0_{i_n})$, and for all $\zeta_1 \otimes \zeta_2 \otimes \cdots \otimes \zeta_n \in \mathcal{H}^0_{i_1} \otimes \mathcal{H}^0_{i_2} \otimes \cdots \otimes \mathcal{H}^0_{i_n}$ where $i_1 \neq i$
\[
A(\zeta_1 \otimes \zeta_2 \otimes \cdots \otimes \zeta_n) = (\langle \pi_i(A)\xi_i, \xi_i\rangle_{\mathcal{H}_i}\zeta_1 \otimes \cdots \otimes \zeta_n) \oplus ((\pi_i(A)\xi_i - \langle \pi_i(A)\xi_i, \xi_i\rangle_{\mathcal{H}_i} \xi_i) \otimes \zeta_1 \otimes \cdots \otimes \zeta_n)
\]
which is an element of $(\mathcal{H}^0_{i_1} \otimes \cdots \otimes \mathcal{H}^0_{i_n}) \oplus (\mathcal{H}^0_{i} \otimes \mathcal{H}^0_{i_1} \otimes \cdots \otimes \mathcal{H}^0_{i_n})$.  We denote the C$^*$-subalgebra of $\mathcal{B}((\mathcal{H}_1, \xi_1) \ast(\mathcal{H}_2, \xi_2))$ generated by $\mathfrak{A}_1$ and $\mathfrak{A}_2$ by $(\mathfrak{A}_1, \pi_1, \xi_1) \ast (\mathfrak{A}_2, \pi_2, \xi_2)$.
\end{nota}
As previously mentioned, when dealing with reduced free products of C$^*$-algebras, it is typical that the representations $\pi_i$ are faithful GNS representations with unit cyclic vectors $\xi_i$.  We will not make this restriction.
\begin{cons}
\label{primecons}  Let $\mathfrak{A}_1$ and $\mathfrak{A}_2$ be unital C$^*$-algebras, let $\mathfrak{J}$ be an ideal of $\mathfrak{A}_1$, let $\pi_{1,0} : \mathfrak{A}_1/\mathfrak{J} \to \mathcal{B}(\mathcal{H}_{1,0})$, $\pi_{1,1} : \mathfrak{A}_1 \to \mathcal{B}(\mathcal{H}_{1,1})$, and $\pi_2 : \mathfrak{A}_2 \to \mathcal{B}(\mathcal{H}_2)$ be unital representations such that $\pi_{1,0}$ and $\pi_2$ are faithful and, if $\mathcal{H}_1 := \mathcal{H}_{1,0} \oplus \mathcal{H}_{1,1}$ and $q : \mathfrak{A}_1 \to \mathfrak{A}_1/\mathfrak{J}$ is the canonical quotient map, $\pi_1 := (\pi_{1,0} \circ q) \oplus \pi_{1,1} : \mathfrak{A}_1 \to \mathcal{B}(\mathcal{H}_1)$ is faithful, and let $\xi_1 \in \mathcal{H}_{1,0}$ and $\xi_2 \in \mathcal{H}_2$ be unit vectors.  Consider the reduced free products $(\mathfrak{A}_1/\mathfrak{J}, \pi_{1,0}, \xi_1) \ast (\mathfrak{A}_2, \pi_2, \xi_2)$ and $(\mathfrak{A}_1, \pi_1, \xi_1) \ast (\mathfrak{A}_2, \pi_2, \xi_2)$.  Let $\langle \mathfrak{J}\rangle_{\mathfrak{A}_1 \ast \mathfrak{A}_2}$ denote the closed ideal of $(\mathfrak{A}_1, \pi_1, \xi_1) \ast (\mathfrak{A}_2, \pi_2, \xi_2)$ generated by $\mathfrak{J}$.  
\par
Notice $(\mathfrak{A}_1, \pi_1, \xi_1) \ast (\mathfrak{A}_2, \pi_2, \xi_2)$ acts on $(\mathcal{H}_1, \xi_1) \ast (\mathcal{H}_2, \xi_2)$ and $(\mathfrak{A}_1/\mathfrak{J}, \pi_{1,0}, \xi_1) \ast (\mathfrak{A}_2, \pi_2, \xi_2)$ acts on $(\mathcal{H}_{1,0}, \xi_1) \ast (\mathcal{H}_2, \xi_2)$.  By the construction of the free product of Hilbert spaces, $(\mathcal{H}_{1,0}, \xi_1) \ast (\mathcal{H}_2, \xi_2)$ can be viewed canonically as a Hilbert subspace of $(\mathcal{H}_1, \xi_1) \ast (\mathcal{H}_2, \xi_2)$.  Moreover, by considering the action of $(\mathfrak{A}_1, \pi_1, \xi_1) \ast (\mathfrak{A}_2, \pi_2, \xi_2)$ on $(\mathcal{H}_{1,0}, \xi_1) \ast (\mathcal{H}_2, \xi_2) \subseteq (\mathcal{H}_1, \xi_1) \ast (\mathcal{H}_2, \xi_2)$, it is easily seen that $(\mathcal{H}_{1,0}, \xi_1) \ast (\mathcal{H}_2, \xi_2)$ is an invariant subspace of $(\mathfrak{A}_1, \pi_1, \xi_1) \ast (\mathfrak{A}_2, \pi_2, \xi_2)$ and the compression of $(\mathfrak{A}_1, \pi_1, \xi_1) \ast (\mathfrak{A}_2, \pi_2, \xi_2)$ to this subspace is $(\mathfrak{A}_1/\mathfrak{J}, \pi_{1,0}, \xi_1) \ast (\mathfrak{A}_2, \pi_2, \xi_2)$.  Thus there is a well-defined $^*$-homomorphism 
\[
\pi : (\mathfrak{A}_1, \pi_1, \xi_1) \ast (\mathfrak{A}_2, \pi_2, \xi_2) \to (\mathfrak{A}_1/\mathfrak{J}, \pi_{1,0}, \xi_1) \ast (\mathfrak{A}_2, \pi_2, \xi_2)
\]
defined by
\[
\pi(T) := P_{(\mathcal{H}_{1,0}, \xi_1) \ast (\mathcal{H}_2, \xi_2)} T |_{(\mathcal{H}_{1,0}, \xi_1) \ast (\mathcal{H}_2, \xi_2)}
\]
where $P_{(\mathcal{H}_{1,0}, \xi_1) \ast (\mathcal{H}_2, \xi_2)}$ is the orthogonal projection of $(\mathcal{H}_1, \xi_1) \ast (\mathcal{H}_2, \xi_2)$ onto $(\mathcal{H}_{1,0}, \xi_1) \ast (\mathcal{H}_2, \xi_2)$.
\par
If $J \in \mathfrak{J}$ then it is easily seen that $J|_{(\mathcal{H}_{1,0}, \xi_1) \ast (\mathcal{H}_2, \xi_2)} = 0$ as $\pi_{1,0}(q(J)) = \pi_{1,0}(0) = 0$.  Therefore the algebraic ideal generated by $\mathfrak{J}$ in $(\mathfrak{A}_1, \pi_1, \xi_1) \ast (\mathfrak{A}_2, \pi_2, \xi_2)$ is in the kernel of $\pi$ and thus $\langle \mathfrak{J} \rangle_{\mathfrak{A}_1 \ast \mathfrak{A}_2} \subseteq ker(\pi)$. Hence we can consider the sequence of C$^*$-algebras
\[
0 \rightarrow \langle \mathfrak{J} \rangle_{\mathfrak{A}_1 \ast \mathfrak{A}_2} \stackrel{i}{\rightarrow} (\mathfrak{A}_1, \pi_1, \xi_1) \ast (\mathfrak{A}_2, \pi_2, \xi_2) \stackrel{\pi}{\rightarrow}(\mathfrak{A}_1/\mathfrak{J}, \pi_{1,0}, \xi_1) \ast (\mathfrak{A}_2, \pi_2, \xi_2) \rightarrow 0
\]
where $i$ is the inclusion map.  Clearly $i$ is injective, $\pi$ is surjective, and $\langle \mathfrak{J} \rangle_{\mathfrak{A}_1 \ast \mathfrak{A}_2} \subseteq ker(\pi)$.  Hence the sequence is exact if and only if $ker(\pi) \subseteq \langle \mathfrak{J} \rangle_{\mathfrak{A}_1 \ast \mathfrak{A}_2}$; that is there is no element of $(\mathfrak{A}_1, \pi_1, \xi_1) \ast (\mathfrak{A}_2, \pi_2, \xi_2) \setminus \langle \mathfrak{J}\rangle_{\mathfrak{A}_1 \ast \mathfrak{A}_2}$ that is zero on the copy of $(\mathcal{H}_{1,0}, \xi_1) \ast (\mathcal{H}_2, \xi_2)$ inside $(\mathcal{H}_1, \xi_1) \ast (\mathcal{H}_2, \xi_2)$.
\end{cons}
The requirements on $\pi_{1,0}$, $\pi_1$, and $\pi_2$ are necessary to ensure we are considering objects related directly to $\mathfrak{A}_1/\mathfrak{J}$, $\mathfrak{A}_1$, and $\mathfrak{A}_2$.  The conditions on $\pi_{1,0}$, $\pi_1$, $\pi_2$, $\xi_1$, and $\xi_2$ are also designed so the vectors $\xi_1$ and $\xi_2$ give rise to vector states on our C$^*$-algebras.  Moreover $\pi_{1,0}$, $\pi_1$, and $\pi_2$ are assumed to be unital so the C$^*$-algebras under consideration are truly reduced free products of C$^*$-algebras.  Finally the consideration of $(\mathfrak{A}_1, \pi_1, \xi_1) \ast (\mathfrak{A}_2, \pi_2, \xi_2)$ was necessary to ensure the $^*$-homomorphism $\pi$ existed.
\par
Our main goal is to prove the following result.
\begin{thm}
\label{main}
Let $\mathfrak{A}_1$ and $\mathfrak{A}_2$ be unital C$^*$-algebras, let $\mathfrak{J}$ be an ideal of $\mathfrak{A}_1$, let $\pi_{1,0} : \mathfrak{A}_1/\mathfrak{J} \to \mathcal{B}(\mathcal{H}_{1,0})$, $\pi_{1,1} : \mathfrak{A}_1 \to \mathcal{B}(\mathcal{H}_{1,1})$, and $\pi_2 : \mathfrak{A}_2 \to \mathcal{B}(\mathcal{H}_2)$ be unital representations such that $\pi_{1,0}$ and $\pi_2$ are faithful and, if $\mathcal{H}_1 := \mathcal{H}_{1,0} \oplus \mathcal{H}_{1,1}$ and $q : \mathfrak{A}_1 \to \mathfrak{A}_1/\mathfrak{J}$ is the canonical quotient map, $\pi_1 := (\pi_{1,0} \circ q) \oplus \pi_{1,1} : \mathfrak{A}_1 \to \mathcal{B}(\mathcal{H}_1)$ is faithful, and let $\xi_1 \in \mathcal{H}_{1,0}$ and $\xi_2 \in \mathcal{H}_2$ be unit vectors. Under these assumptions, the sequence of C$^*$-algebras
\[
0 \rightarrow \langle \mathfrak{J}\rangle_{\mathfrak{A}_1 \ast \mathfrak{A}_2} \stackrel{i}{\rightarrow} (\mathfrak{A}_1, \pi_1, 
\xi_1) \ast (\mathfrak{A}_2, \pi_2, \xi_2) \stackrel{\pi}{\rightarrow}(\mathfrak{A}_1/\mathfrak{J}, \pi_{1,0}, \xi_1) \ast (\mathfrak{A}_2, \pi_2, \xi_2) \rightarrow 0
\]
is exact.
\end{thm}
\begin{rem}
\label{2.1.4}
By the above discussion there exists a $^*$-homomorphism
\[
\frac{(\mathfrak{A}_1, \pi_1, \xi_1) \ast (\mathfrak{A}_2, \pi_2, \xi_2)}{\langle \mathfrak{J}\rangle_{\mathfrak{A}_1 \ast \mathfrak{A}_2}} \to (\mathfrak{A}_1/\mathfrak{J}, \pi_{1,0}, \xi_1) \ast (\mathfrak{A}_2, \pi_2, \xi_2)
\]
and the question of whether or not the above sequence is exact is equivalent to this $^*$-homomorphism being injective.  Thus to prove the sequence is exact it would suffice to construct an inverse map.  It is tempting to believe that such an inverse map exists due to the universal property of the reduced free products of C$^*$-algebras (see Theorem 4.7.2 of \cite{BO}).  However, to apply said property, we would need to know the vector states defined by $\xi_1$ and $\xi_2$ on $\mathfrak{A}_1/\mathfrak{J}$ and $\mathfrak{A}_2$ respectively had faithful GNS representations and we would need to know the state on $((\mathfrak{A}_1, \pi_1, \xi_1) \ast (\mathfrak{A}_2, \pi_2, \xi_2)) /\langle \mathfrak{J}\rangle_{\mathfrak{A}_1 \ast \mathfrak{A}_2}$ induced by the vector state on $(\mathfrak{A}_1, \pi_1, \xi_1) \ast (\mathfrak{A}_2, \pi_2, \xi_2)$ from the distinguished vector has a faithful GNS representation.  It is this later condition that appears to provide the greatest obstacle.   Thus, once Theorem \ref{main} has been established, the state on $((\mathfrak{A}_1, \pi_1, \xi_1) \ast (\mathfrak{A}_2, \pi_2, \xi_2)) /\langle \mathfrak{J}\rangle_{\mathfrak{A}_1 \ast \mathfrak{A}_2}$ induced by the vector state on $(\mathfrak{A}_1, \pi_1, \xi_1) \ast (\mathfrak{A}_2, \pi_2, \xi_2)$ from the distinguished vector has a faithful GNS representation whenever the vector states defined by $\xi_1$ and $\xi_2$ on $\mathfrak{A}_1/\mathfrak{J}$ and $\mathfrak{A}_2$ respectively have faithful GNS representations.
\end{rem}
The proof of Theorem \ref{main} will be demonstrated over the next three sections.  In Section 2.2 we will examine the ideal $\langle \mathfrak{J}\rangle_{\mathfrak{A}_1 \ast \mathfrak{A}_2}$ by describing a set of operators with dense span.  In Section 2.3 we will construct short exact sequence of C$^*$-algebras involving Toeplitz-Pimsner algebras.  In Section 2.4 we will embed each sequence under consideration from Theorem \ref{main} into a sequence from Section 2.3 and, with a little work, this will complete the proof of Theorem \ref{main}.

\subsection{Structure of the Ideal}

In our goal to prove Theorem \ref{main} we will begin by analyzing the structure of the ideals $\langle \mathfrak{J}\rangle_{\mathfrak{A}_1 \ast \mathfrak{A}_2}$ under consideration.  A set of operators consisting of products of elements from $\mathfrak{J}$, $\mathfrak{A}_1$, and $\mathfrak{A}_2$ will be shown to be dense in $\langle \mathfrak{J}\rangle_{\mathfrak{A}_1 \ast \mathfrak{A}_2}$ using common arguments pertaining to reduced free products.  We will then analyze the action of each of these operators on $(\mathcal{H}_1,\xi_1)\ast(\mathcal{H}_2,\xi_2)$.
\begin{disc}
\label{2.4.1}
Let $\mathfrak{A}_1$ and $\mathfrak{A}_2$ be unital C$^*$-algebras, let $\mathfrak{J}$ be an ideal of $\mathfrak{A}_1$, let $\pi_{1,0} : \mathfrak{A}_1/\mathfrak{J} \to \mathcal{B}(\mathcal{H}_{1,0})$, $\pi_{1,1} : \mathfrak{A}_1 \to \mathcal{B}(\mathcal{H}_{1,1})$, and $\pi_2 : \mathfrak{A}_2 \to \mathcal{B}(\mathcal{H}_2)$ be unital representations such that $\pi_{1,0}$ and $\pi_2$ are faithful and, if $\mathcal{H}_1 := \mathcal{H}_{1,0} \oplus \mathcal{H}_{1,1}$ and $q : \mathfrak{A}_1 \to \mathfrak{A}_1/\mathfrak{J}$ is the canonical quotient map, $\pi_1 := (\pi_{1,0} \circ q) \oplus \pi_{1,1} : \mathfrak{A}_1 \to \mathcal{B}(\mathcal{H}_1)$ is faithful, and let $\xi_1 \in \mathcal{H}_{1,0}$ and $\xi_2 \in \mathcal{H}_2$ be unit vectors.   To determine the structure of $\langle \mathfrak{J}\rangle_{\mathfrak{A}_1 \ast \mathfrak{A}_2}$ inside $(\mathfrak{A}_1, \pi_1, \xi_1) \ast (\mathfrak{A}_2, \pi_2, \xi_2)$, note that
\[
span\{ A_1 B_1 \cdots A_nB_n J B'_1 A'_1 \cdots B'_m A'_m \, \mid \, n,m\geq 0, A_i, A'_j \in\mathfrak{A}_1, B_i, B'_j \in \mathfrak{A}_2, J \in \mathfrak{J}\}
\]
(where we can alway begin and end with an element of $\mathfrak{A}_1$ as $\mathfrak{A}_1$ is unital and $\pi_1$ is unital) is dense in $\langle \mathfrak{J}\rangle_{\mathfrak{A}_1 \ast \mathfrak{A}_2}$. For $i \in \{1,2\}$ let
\[
\mathfrak{A}_i^0 := \{A \in \mathfrak{A}_i \, \mid \, \langle A\xi_i, \xi_i\rangle_{\mathcal{H}_i} = 0\}
\]
so $\mathfrak{A}_i = \mathbb{C} I_{\mathfrak{A}_i} + \mathfrak{A}_i^0$.  Using the fact that $\mathfrak{J}$ is an ideal of $\mathfrak{A}_1$ and the fact that $I_{\mathfrak{A}_i}$ are the identity elements of $(\mathfrak{A}_1, \pi_1, \xi_1) \ast (\mathfrak{A}_2, \pi_2, \xi_2)$, 
\[
span\left\{\left.
\begin{array}{c}
  A_1B_{1} \cdots A_nB_nJB'_mA'_m\cdots B'_1A'_1, \\
   B_{1}A_2 \cdots A_nB_nJB'_mA'_m\cdots B'_1A'_1, \\
   A_1B_{1} \cdots A_nB_nJB'_mA'_m\cdots A'_2B'_1, \\
    B_{1}A_2 \cdots A_nB_nJB'_mA'_m\cdots A'_2B'_1
\end{array}
 \, \right| \, n,m\geq 0, A_i, A'_j \in\mathfrak{A}_1^0, B_i, B'_j \in \mathfrak{A}_2^0, J \in \mathfrak{J}\right\}
\]
is dense in $\langle \mathfrak{J}\rangle_{\mathfrak{A}_1 \ast \mathfrak{A}_2}$.  Notice $\mathfrak{J} \subseteq \mathfrak{A}^0_1$.
\end{disc}
\begin{disc}
\label{idealdis1}
To begin the analysis of $\langle \mathfrak{J}\rangle_{\mathfrak{A}_1 \ast \mathfrak{A}_2}$, for $i \in \{1,2\}$ let $\mathcal{H}_i^0 := \mathcal{H}_i \ominus \mathbb{C}\xi_i$ and recall that $(\mathfrak{A}_1, \pi_1, \xi_1) \ast (\mathfrak{A}_2, \pi_2, \xi_2)$ acts on 
\[
(\mathcal{H}_1, \xi_1) \ast (\mathcal{H}_2,\xi_2) = \mathbb{C} \xi_0 \oplus \left( \bigoplus_{n\geq 1} \left( \bigoplus_{\begin{array}{cc}
\{i_k\}^n_{k=1} \subseteq \{1,2\},  \\
 i_k \neq i_{k+1} \mbox{ for }k \in \{1, \ldots, n-1\}
\end{array}
} \mathcal{H}^0_{i_1} \otimes \mathcal{H}^0_{i_2} \otimes \cdots \otimes \mathcal{H}^0_{i_n}  \right)\right)
\]
as described in Notation \ref{freeprod}.  
\par
Fix $n,m\geq 0, \{A_i\}^n_{i=1}, \{A'_j\}^m_{j=1} \subseteq \mathfrak{A}_1^0$, $\{B_i\}^n_{i=1}, \{B'_j\}^m_{j=1} \subseteq \mathfrak{A}_2^0$, and $J \in \mathfrak{J}$.  Let 
\[
T = A_1B_{1} \cdots A_nB_nJB'_mA'_m\cdots B'_1A'_1 \,\,\,\,\,\,\,\,\,\,\mbox{ and }\,\,\,\,\,\,\,\,\,\, R = B_{1}A_2 \cdots A_nB_nJB'_mA'_m\cdots B'_1A'_1.
\]
We desire to describe the actions of $T$ and $R$ on $(\mathcal{H}_1, \xi_1) \ast (\mathcal{H}_2,\xi_2)$.  First we claim that $T$ and $R$ are zero on 
\[
\mathbb{C} \xi_0 \oplus \left(\bigoplus_{k\geq 1}(\mathcal{H}_2^0 \otimes \mathcal{H}_1^0)^{\otimes k}\right) \oplus \left( \bigoplus_{k\geq 0} (\mathcal{H}_2^0 \otimes \mathcal{H}_1^0)^{\otimes k} \otimes \mathcal{H}_2^0\right) \subseteq (\mathcal{H}_1, \xi_1) \ast (\mathcal{H}_2,\xi_2).
\]
To see this, notice for all $k\geq 1$ that
\[
A'_1((\mathcal{H}_2^0 \otimes \mathcal{H}_1^0)^{\otimes k}) \subseteq \mathcal{H}_1^0 \otimes (\mathcal{H}_2^0 \otimes \mathcal{H}_1^0)^{\otimes k}
\]
(as $A'_1 \xi_1 \in \mathcal{H}^0_1$ by the assumption that $A'_1 \in \mathfrak{A}_1^0$),
\[
B'_1 (\mathcal{H}_1^0 \otimes (\mathcal{H}_2^0 \otimes \mathcal{H}_1^0)^{\otimes k}) \subseteq \mathcal{H}_2^0 \otimes \mathcal{H}_1^0 \otimes (\mathcal{H}_2^0 \otimes \mathcal{H}_1^0)^{\otimes k}
\]
(as $B'_1 \in \mathfrak{A}_2^0$), and thus, by continuing the pattern, $J$ will act on $(\mathcal{H}_2^0 \otimes \mathcal{H}_1^0)^{\otimes m+k}$.  However, as $J \xi_1 = 0$, $J$ acts as the zero operator on $(\mathcal{H}_2^0 \otimes \mathcal{H}_1^0)^{\otimes m+k}$ and thus $T$ and $R$ are zero on $(\mathcal{H}_2^0 \otimes \mathcal{H}_1^0)^{\otimes k}$.  The arguments for the other terms in the direct sum are similar.
\par
Thus it remains to describe the actions of $T$ and $R$ on
\[
\left(\bigoplus_{k\geq 1}(\mathcal{H}_1^0 \otimes \mathcal{H}_2^0)^{\otimes k}\right) \oplus \left( \bigoplus_{k\geq 0} (\mathcal{H}_1^0 \otimes \mathcal{H}_2^0)^{\otimes k} \otimes \mathcal{H}_1^0\right)   \subseteq (\mathcal{H}_1, \xi_1) \ast (\mathcal{H}_2,\xi_2).
\]
We claim that $T$ and $R$ are non-zero only on the direct summand
\[
\left(\bigoplus_{k\geq m+1}(\mathcal{H}_1^0 \otimes \mathcal{H}_2^0)^{\otimes k}\right) \oplus \left( \bigoplus_{k\geq m} (\mathcal{H}_1^0 \otimes \mathcal{H}_2^0)^{\otimes k} \otimes \mathcal{H}_1^0\right)   \subseteq (\mathcal{H}_1, \xi_1) \ast (\mathcal{H}_2,\xi_2),
\]
if $k \geq m$ and
\[
\eta_1 \otimes \zeta_1 \otimes \cdots \otimes \eta_m \otimes \zeta_m \otimes \eta_{m+1} \otimes \zeta_{m+1} \otimes \cdots \otimes \zeta_k \otimes \eta_{k+1} \in (\mathcal{H}_1^0 \otimes \mathcal{H}_2^0)^{\otimes k} \otimes \mathcal{H}_1^0
\]
then
\[
S(\eta_1 \otimes \zeta_1 \otimes \cdots \otimes \eta_m \otimes \zeta_m \otimes \eta_{m+1} \otimes \zeta_{m+1} \otimes \cdots \otimes \zeta_k \otimes \eta_{k+1}) = S(\eta_1 \otimes \zeta_1 \otimes \cdots \otimes \eta_m \otimes \zeta_m \otimes \eta_{m+1}) \otimes \zeta_{m+1} \otimes \cdots \otimes \zeta_k \otimes \eta_{k+1}
\]
for $S = T$ and $S = R$, and if $k \geq m+1$ and
\[
\eta_1 \otimes \zeta_1 \otimes \cdots \otimes \eta_m \otimes \zeta_m \otimes \eta_{m+1} \otimes \zeta_{m+1} \otimes \cdots \otimes \eta_{k} \otimes \zeta_{k} \in (\mathcal{H}_1^0 \otimes \mathcal{H}_2^0)^{\otimes k}
\]
then
\[
S(\eta_1 \otimes \zeta_1 \otimes \cdots \otimes \eta_m \otimes \zeta_m \otimes \eta_{m+1} \otimes \zeta_{m+1} \otimes \cdots \otimes \eta_{k+1} \otimes \zeta_{k+1}) = S(\eta_1 \otimes \zeta_1 \otimes \cdots \otimes \eta_m \otimes \zeta_m \otimes \eta_{m+1}) \otimes \zeta_{m+1} \otimes \cdots \otimes \eta_k \otimes \zeta_{k}
\]
for $S = T$ and $S = R$.  To prove this result we will proceed by induction on $m$.  For $m = 0$ notice if $k\geq 1$ then for all
\[
\eta_1 \otimes \zeta_1 \otimes \cdots \otimes \eta_k \otimes \zeta_k \in (\mathcal{H}_1^0 \otimes \mathcal{H}_2^0)^{\otimes k}
\]
we have
\[
J(\eta_1 \otimes \zeta_1 \otimes \cdots \otimes \eta_k \otimes \zeta_k) = (J\eta_1) \otimes \zeta_1 \otimes \cdots \otimes \eta_k \otimes \zeta_k \in (\mathcal{H}_1^0 \otimes \mathcal{H}_2^0)^{\otimes k}
\]
as $J(\mathcal{H}_1^0) \subseteq \mathcal{H}_1^0$.  As $B_n \in \mathfrak{A}_2^0$
\[
B_n((J\eta_1) \otimes \zeta_1 \otimes \cdots \otimes \eta_k \otimes \zeta_k) = B_n\xi_2 \otimes (J\eta_1) \otimes \zeta_1 \otimes \cdots \otimes \eta_k \otimes \zeta_k \in \mathcal{H}_2^0 \otimes(\mathcal{H}_1^0 \otimes \mathcal{H}_2^0)^{\otimes k}.
\]
Similarly $A_n \in \mathfrak{A}_1^0$ so 
\[
A_n(B_n\xi_2 \otimes (J\eta_1) \otimes \zeta_1 \otimes \cdots \otimes \eta_k \otimes \zeta_k) = A_n \xi_1 \otimes B_n\xi_2 \otimes (J\eta_1) \otimes \zeta_1 \otimes \cdots \otimes \eta_k \otimes \zeta_k \in \mathcal{H}_1^0 \otimes \mathcal{H}_2^0 \otimes(\mathcal{H}_1^0 \otimes \mathcal{H}_2^0)^{\otimes k}.
\]
By repetition
\[
T(\eta_1 \otimes \zeta_1 \otimes \cdots \otimes \eta_k \otimes \zeta_k) = A_1\xi_1 \otimes B_1\xi_2 \otimes \cdots \otimes A_n \xi_1 \otimes B_n\xi_2 \otimes (J\eta_1) \otimes \zeta_1 \otimes \cdots \otimes \eta_k \otimes \zeta_k
\]
and
\[
R(\eta_1 \otimes \zeta_1 \otimes \cdots \otimes \eta_k \otimes \zeta_k) =  B_1\xi_2  \otimes A_2\xi_1 \otimes \cdots \otimes A_n \xi_1 \otimes B_n\xi_2 \otimes (J\eta_1) \otimes \zeta_1 \otimes \cdots \otimes \eta_k \otimes \zeta_k.
\]
Similar arguments show for all $k\geq 0$ and
\[
\eta_1 \otimes \zeta_1 \otimes \cdots \otimes \zeta_k \otimes \eta_{k+1} \in (\mathcal{H}_1^0 \otimes \mathcal{H}_2^0)^{\otimes k} \otimes \mathcal{H}_1^0
\]
that
\[
T(\eta_1 \otimes \zeta_1 \otimes \cdots \otimes \zeta_k \otimes \eta_{k+1}) = A_1\xi_1 \otimes B_1\xi_2 \otimes \cdots \otimes A_n \xi_1 \otimes B_n\xi_2 \otimes (J\eta_1) \otimes \zeta_1 \otimes \cdots \otimes \zeta_k \otimes \eta_{k+1}
\]
and
\[
R(\eta_1 \otimes \zeta_1 \otimes \cdots \otimes \zeta_k \otimes \eta_{k+1}) =  B_1\xi_2 \otimes A_2\xi_1\otimes \cdots \otimes A_n \xi_1 \otimes B_n\xi_2 \otimes (J\eta_1) \otimes \zeta_1 \otimes \cdots \otimes \zeta_k \otimes \eta_{k+1}.
\]
Hence the base case is complete.
\par
Suppose $m \geq 1$ and the result is true for $m-1$.  Consider the action of $B'_1 A'_1$.  If $\eta_1 \in \mathcal{H}^0_1$ then
\begin{eqnarray} 
B'_1A'_1(\eta_1)\!\!\!\!&=&\!\!\!\!B'_1(\langle A'_1\eta_1, \xi_1\rangle_{\mathcal{H}_1} \xi_1 + (A'_1\eta_1 - \langle A'_1\eta_1, \xi_1\rangle_{\mathcal{H}_1} \xi_1)  ) \nonumber\\
&=&\!\!\!\!  \langle A'_1\eta_1, \xi_1\rangle_{\mathcal{H}_1} B'_1 \xi_2 +  B'_1\xi_2 \otimes (A'_1\eta_1 - \langle A'_1\eta_1, \xi_1\rangle_{\mathcal{H}_1} \xi_1)  \nonumber
\end{eqnarray}
which is an element of $\mathcal{H}^0_2 \oplus (\mathcal{H}^0_2 \otimes \mathcal{H}_1^0)$ as $B'_1 \in \mathfrak{A}_2^0$.  Since $JB'_mA'_m\cdots B'_2A'_2$ is zero on $\mathcal{H}^0_2 \oplus (\mathcal{H}^0_2 \otimes \mathcal{H}_1^0)$ by earlier discussions, $T$ and $R$ are zero on $\mathcal{H}^0_1$.  In addition, if $\eta_1 \in \mathcal{H}^0_1$ and $\zeta_1 \in \mathcal{H}^0_2$ then
\begin{eqnarray} 
B'_1A'_1(\eta_1 \otimes \zeta_1)\!\!\!\!&=&\!\!\!\! B'_1(\langle A'_1\eta_1, \xi_1\rangle_{\mathcal{H}_1} \zeta_1 + (A'_1\eta_1 - \langle A'_1\eta_1, \xi_1\rangle_{\mathcal{H}_1} \xi_1) \otimes \zeta_1  )\nonumber\\
&=&\!\!\!\!\langle A'_1\eta_1, \xi_1\rangle_{\mathcal{H}_1} \langle B'_1 \zeta_1, \xi_2\rangle_{\mathcal{H}_2}\xi_0 + \langle A'_1\eta, \xi_1\rangle_{\mathcal{H}_1}(B'_1\zeta_1 - \langle B'_1 \zeta_1, \xi_2\rangle_{\mathcal{H}_2}\xi_2) \nonumber\\
&&\!\!\!\!  +  B'_1\xi_2 \otimes (A'_1\eta_1 - \langle A'_1\eta_1, \xi_1\rangle_{\mathcal{H}_1} \xi_1) \otimes \zeta_1 \nonumber
\end{eqnarray}
which is an element of $\mathbb{C}\xi_0 \oplus \mathcal{H}^0_2 \oplus (\mathcal{H}^0_2 \otimes \mathcal{H}_1^0 \otimes \mathcal{H}^0_2)$ as $B'_1 \in \mathfrak{A}_2^0$.  Since $JB'_mA'_m\cdots B'_2A'_2$ is zero on $\mathbb{C}\xi_0 \oplus \mathcal{H}^0_2 \oplus (\mathcal{H}^0_2 \otimes \mathcal{H}_1^0 \otimes \mathcal{H}^0_2)$ by earlier discussions, $T$ and $R$ are zero on $\mathcal{H}^0_1 \otimes \mathcal{H}^0_2$.
\par
Now suppose $\eta_1 \in \mathcal{H}^0_1$, $\zeta_1 \in \mathcal{H}^0_2$, and
\[
\theta \in \left(\bigoplus_{k\geq 1}(\mathcal{H}_1^0 \otimes \mathcal{H}_2^0)^{\otimes k}\right) \oplus \left( \bigoplus_{k\geq 0} (\mathcal{H}_1^0 \otimes \mathcal{H}_2^0)^{\otimes k} \otimes \mathcal{H}_1^0\right).
\]
Then
\begin{eqnarray} 
B'_1A'_1(\eta_1 \otimes \zeta_1 \otimes \theta)\!\!\!\!&=&\!\!\!\!B'_1(\langle A'_1\eta_1, \xi_1\rangle_{\mathcal{H}_1} \zeta_1 \otimes \theta + (A'_1\eta_1 - \langle A'_1\eta_1, \xi_1\rangle_{\mathcal{H}_1} \xi_1) \otimes \zeta_1 \otimes \theta  ) \nonumber\\
&=&\!\!\!\!  \langle A'_1\eta_1, \xi_1\rangle_{\mathcal{H}_1} \langle B'_1 \zeta_1, \xi_2\rangle_{\mathcal{H}_2}\theta + \langle A'_1\eta_1, \xi_1\rangle_{\mathcal{H}_1}(B'_1\zeta_1 - \langle B'_1 \zeta_1, \xi_2\rangle_{\mathcal{H}_2}\xi_2) \otimes \theta \nonumber\\
&&\!\!\!\!  +  B'_1\xi_2 \otimes (A'_1\eta_1 - \langle A'_1\eta_1, \xi_1\rangle_{\mathcal{H}_1} \xi_1) \otimes \zeta_1 \otimes \theta. \nonumber
\end{eqnarray}
Therefore, since $JB'_mA'_m\cdots B'_2A'_2$ is zero on 
\[
\mathbb{C} \xi_0 \oplus \left(\bigoplus_{k\geq 1}(\mathcal{H}_2^0 \otimes \mathcal{H}_1^0)^{\otimes k}\right) \oplus \left( \bigoplus_{k\geq 0} (\mathcal{H}_2^0 \otimes \mathcal{H}_1^0)^{\otimes k} \otimes \mathcal{H}_2^0\right),
\]
we obtain that
\[
T(\eta_1 \otimes \zeta_1 \otimes \theta) = \langle A'_1\eta_1, \xi_1\rangle_{\mathcal{H}_1} \langle B'_1 \zeta_1, \xi_2\rangle_{\mathcal{H}_2} A_1B_{1} \cdots A_nB_nJB'_mA'_m\cdots B'_2A'_2(\theta)
\]
and
\[
R(\eta_1 \otimes \zeta_1 \otimes \theta) = \langle A'_1\eta_1, \xi_1\rangle_{\mathcal{H}_1} \langle B'_1 \zeta_1, \xi_2\rangle_{\mathcal{H}_2} B_{1}A_2 \cdots A_nB_nJB'_mA'_m\cdots B'_2A'_2(\theta). 
\]
Hence the result follows easily by the induction hypothesis.  
\par
The above proof shows that if 
\[
\eta_1 \otimes \zeta_1 \otimes \cdots \otimes \eta_m \otimes \zeta_m \otimes \eta \in (\mathcal{H}_1^0 \otimes \mathcal{H}_{2}^0)^{\otimes m} \otimes \mathcal{H}_1^0
\]
then
\begin{eqnarray} 
\!\!\!\!&&\!\!\!\! T(\eta_1 \otimes \zeta_1 \otimes \cdots \otimes \eta_m \otimes \zeta_m \otimes \eta)\nonumber\\
&=&\!\!\!\! \left(\prod_{k=1}^m \langle  \eta_k, (A'_k)^*\xi_1\rangle_{\mathcal{H}_1}\right)\left(\prod_{k=1}^m \langle \zeta_k, ( B'_k )^*\xi_2\rangle_{\mathcal{H}_2}\right) (A_1\xi_1) \otimes (B_1\xi_2) \otimes \cdots \otimes (A_n\xi_1) \otimes (B_n\xi_2) \otimes J \eta \nonumber
\end{eqnarray}
and
\begin{eqnarray} 
\!\!\!\!&&\!\!\!\! R(\eta_1 \otimes \zeta_1 \otimes \cdots \otimes \eta_m \otimes \zeta_m \otimes \eta)\nonumber\\
&=&\!\!\!\! \left(\prod_{k=1}^m \langle  \eta_k, (A'_k)^*\xi_1\rangle_{\mathcal{H}_1}\right)\left(\prod_{k=1}^m \langle \zeta_k, ( B'_k )^*\xi_2\rangle_{\mathcal{H}_2}\right) (B_1\xi_2) \otimes (A_2 \xi_1) \otimes \cdots \otimes (A_n\xi_1) \otimes (B_n\xi_2) \otimes J \eta.  \nonumber
\end{eqnarray}
\end{disc}
\begin{disc}
\label{idealdis2}
Similarly if $n\geq 0$, $m\geq 1$, $\{A_i\}^n_{i=1}, \{A'_j\}^m_{j=1} \subseteq \mathfrak{A}_1^0$, $\{B_i\}^n_{i=1}, \{B'_j\}^m_{j=1} \subseteq \mathfrak{A}_2^0$, and $J \in \mathfrak{J}$ then 
\[
T = A_1B_{1} \cdots A_nB_nJB'_mA'_m\cdots A'_2B'_1\,\,\,\,\,\,\,\,\,\,\mbox{ and }\,\,\,\,\,\,\,\,\,R = B_{1}A_2 \cdots A_nB_nJB'_mA'_m\cdots A'_2B'_1
\]
are non-zero only on the direct summand
\[
\left(\bigoplus_{k\geq m}(\mathcal{H}_2^0 \otimes \mathcal{H}_1^0)^{\otimes k}\right) \oplus \left( \bigoplus_{k\geq m} (\mathcal{H}_2^0 \otimes \mathcal{H}_1^0)^{\otimes k} \otimes \mathcal{H}_2^0\right)   \subseteq (\mathcal{H}_1, \xi_1) \ast (\mathcal{H}_2,\xi_2),
\]
if $k \geq m$ and
\[
\zeta_1 \otimes \eta_1 \otimes \cdots \otimes \zeta_m \otimes \eta_m \otimes \zeta_{m+1} \otimes \eta_{m+1} \otimes \cdots \otimes \eta_k \otimes \zeta_{k+1} \in (\mathcal{H}_2^0 \otimes \mathcal{H}_1^0)^{\otimes k} \otimes \mathcal{H}_2^0
\]
then
\[
S(\zeta_1 \otimes \eta_1 \otimes \cdots \otimes \zeta_m \otimes \eta_m \otimes \zeta_{m+1} \otimes \eta_{m+1} \otimes \cdots \otimes \eta_k \otimes \zeta_{k+1}) = S(\zeta_1 \otimes \eta_1 \otimes \cdots \otimes \zeta_m \otimes \eta_m ) \otimes \zeta_{m+1} \otimes \eta_{m+1} \otimes \cdots \otimes \eta_k \otimes \zeta_{k+1}
\]
for $S = T$ and $S = R$, if $k \geq m$ and
\[
\zeta_1 \otimes \eta_1 \otimes \cdots \otimes \zeta_m \otimes \eta_m \otimes \zeta_{m+1} \otimes \eta_{m+1} \otimes \cdots \otimes \zeta_{k} \otimes \eta_{k} \in (\mathcal{H}_2^0 \otimes \mathcal{H}_1^0)^{\otimes k}
\]
then
\[
S(\zeta_1 \otimes \eta_1 \otimes \cdots \otimes \zeta_m \otimes \eta_m \otimes \zeta_{m+1} \otimes \eta_{m+1} \otimes \cdots \otimes \zeta_{k+1} \otimes \eta_{k+1}) = S(\zeta_1 \otimes \eta_1 \otimes \cdots \otimes \zeta_m \otimes \eta_m)\otimes \zeta_{m+1} \otimes \eta_{m+1} \otimes \cdots \otimes \zeta_k \otimes \eta_{k}
\]
for $S = T$ and $S = R$, and if 
\[
\zeta_1 \otimes \eta_2 \otimes \cdots \otimes \eta_m \otimes \zeta_m \otimes \eta \in (\mathcal{H}_2^0 \otimes \mathcal{H}_{1}^0)^{\otimes m}
\]
then
\begin{eqnarray} 
\!\!\!\!&&\!\!\!\! T(\zeta_1 \otimes \eta_1 \otimes \cdots \otimes \zeta_m \otimes \eta_m)\nonumber\\
&=&\!\!\!\! \left(\prod_{k=2}^m \langle  \eta_{k}, (A'_k)^*\xi_1\rangle_{\mathcal{H}_1}\right)\left(\prod_{k=1}^m \langle \zeta_k, ( B'_k )^*\xi_2\rangle_{\mathcal{H}_2}\right)(A_1 \xi_1) \otimes  (B_1\xi_2) \otimes \cdots \otimes (A_n\xi_1) \otimes (B_n\xi_2) \otimes J \eta.  \nonumber
\end{eqnarray}
and
\begin{eqnarray} 
\!\!\!\!&&\!\!\!\! R(\zeta_1 \otimes \eta_1 \otimes \cdots \otimes \zeta_m \otimes \eta_m)\nonumber\\
&=&\!\!\!\! \left(\prod_{k=2}^m \langle  \eta_{k}, (A'_k)^*\xi_1\rangle_{\mathcal{H}_1}\right)\left(\prod_{k=1}^m \langle \zeta_k, ( B'_k )^*\xi_2\rangle_{\mathcal{H}_2}\right) (B_1\xi_2) \otimes (A_2 \xi_1) \otimes \cdots \otimes (A_n\xi_1) \otimes (B_n\xi_2) \otimes J \eta.  \nonumber
\end{eqnarray}
\end{disc}

\subsection{Another Exact Sequence}

In this section we will examine a short exact sequence of C$^*$-algebras involving Toeplitz-Pimsner algebras.  An outline of the proof will be given after the following construction.
\begin{cons}
\label{directsumconstruction}
Let $\mathfrak{A}_1$ and $\mathfrak{A}_2$ be unital C$^*$-algebras, let $\mathfrak{J}$ be an ideal of $\mathfrak{A}_1$, let $\pi_{1,0} : \mathfrak{A}_1/\mathfrak{J} \to \mathcal{B}(\mathcal{H}_{1,0})$, $\pi_{1,1} : \mathfrak{A}_1 \to \mathcal{B}(\mathcal{H}_{1,1})$, and $\pi_2 : \mathfrak{A}_2 \to \mathcal{B}(\mathcal{H}_2)$ be unital representations such that $\pi_{1,0}$ and $\pi_2$ are faithful and, if $\mathcal{H}_1 := \mathcal{H}_{1,0} \oplus \mathcal{H}_{1,1}$ and $q : \mathfrak{A}_1 \to \mathfrak{A}_1/\mathfrak{J}$ is the canonical quotient map, $\pi_1 := (\pi_{1,0} \circ q) \oplus \pi_{1,1} : \mathfrak{A}_1 \to \mathcal{B}(\mathcal{H}_1)$ is faithful, and let $\xi_1 \in \mathcal{H}_{1,0}$ and $\xi_2 \in \mathcal{H}_2$ be unit vectors.  For notational purposes let $\mathcal{H}_{2,0} := \mathcal{H}_2$ and let $\pi_{2,0} := \pi_2 : \mathfrak{A}_2 \to \mathcal{B}(\mathcal{H}_{2,0})$.
\par
Consider the Hilbert space $\mathcal{K} := \mathcal{K}_1 \oplus \mathcal{K}_2$ where
\[
\mathcal{K}_i := \bigoplus_{n\geq 1} \bigoplus_{
\begin{array}{c}
\{i_k\}^n_{k=1} \subseteq \{1,2\}, i_1 = i  \\
i_k \neq i_{k+1} \mbox{ for }k \in \{1,\ldots, n\}
\end{array} }
\mathcal{H}_{i_1} \otimes \cdots \otimes \mathcal{H}_{i_n}
\]
for $i \in \{1,2\}$. To simplify notation, for all $i \in \{1,2\}$ and $n \in \mathbb{N}$ let
\[
\mathcal{L}_{i,n} := \mathcal{H}_{i_1} \otimes \cdots \otimes \mathcal{H}_{i_n}
\]
where $\{i_k\}^n_{k=1} \subseteq \{1,2\}$, $i_1 = i$, and $i_k \neq i_{k+1} \mbox{ for }k \in \{1,\ldots, n\}$.  Thus
\[
\mathcal{K} = \bigoplus_{n \in \mathbb{N},i \in \{1,2\}}  \mathcal{L}_{i,n}.
\]
\par
Let $S \in \mathcal{B}(\mathcal{K})$ be the isometry defined by
\[
S(\eta_1 \otimes \cdots \otimes \eta_n) = \xi_1 \otimes \eta_1 \otimes \cdots \otimes \eta_n \in \mathcal{L}_{1,n+1}
\]
for all $\eta_1 \otimes \cdots \otimes \eta_n \in \mathcal{L}_{2,n}$,
\[
S(\eta_1 \otimes \cdots \otimes \eta_n) = \xi_2 \otimes \eta_1 \otimes \cdots \otimes \eta_n \in \mathcal{L}_{2,n+1}
\]
for all $\eta_1 \otimes \cdots \otimes \eta_n \in \mathcal{L}_{1,n}$, and by extending by linearity and density.  It is clear that the action of $S^* \in \mathcal{B}(\mathcal{K})$ is given by $S^*(\eta) = 0$ for all $\eta \in \mathcal{L}_{1,1} \oplus \mathcal{L}_{2,1}$ and
\[
S^*(\eta_1 \otimes \cdots \otimes \eta_n) = \langle \eta_1, \xi_i\rangle_{\mathcal{H}_i} \eta_2 \otimes \cdots \otimes \eta_n 
\]
for all $\eta_1 \otimes \cdots \otimes \eta_n \in \mathcal{L}_{i,n}$ and $i \in \{1,2\}$.
\par
Notice that $\mathfrak{A}_1 \oplus \mathfrak{A}_2$ has a faithful representation on $\mathcal{K}$ given by 
\[
(A_1 \oplus A_2)(\eta_1 \otimes \cdots \otimes \eta_n) = \pi_i(A_i)\eta_1 \otimes \cdots \otimes \eta_n
\]
for all $\eta_1 \otimes \cdots \otimes \eta_n \in \mathcal{L}_{i,n}$, $A_j \in \mathfrak{A}_j$, and $i \in \{1,2\}$.  Let $C^*(\mathfrak{A}_1 \oplus \mathfrak{A}_2, S)$ denote the C$^*$-subalgebra of $\mathcal{B}(\mathcal{K})$ generated by $\mathfrak{A}_1 \oplus \mathfrak{A}_2$ and $S$.   From this point onward we will suppress the representations $\pi_i$ and view $\mathfrak{A}_i \subseteq \mathfrak{A}_1 \oplus \mathfrak{A}_2 \subseteq C^*(\mathfrak{A}_1 \oplus \mathfrak{A}_2, S)$ canonically.  The C$^*$-algebra $C^*(\mathfrak{A}_1 \oplus \mathfrak{A}_2, S)$ is called a Toeplitz-Pimsner C$^*$-algebra (usually it is required that $\pi_1$ and $\pi_2$ are faithful GNS representations).
\par
Similarly consider the Hilbert space $\mathcal{K}_0 := \mathcal{K}_{1,0} \oplus \mathcal{K}_{2,0}$ where
\[
\mathcal{K}_{i,0} := \bigoplus_{n\geq 1} \bigoplus_{
\begin{array}{c}
\{i_k\}^n_{k=1} \subseteq \{1,2\}, i_1 = i  \\
i_k \neq i_{k+1} \mbox{ for }k \in \{1,\ldots, n\}
\end{array} }
\mathcal{H}_{i_1,0} \otimes \cdots \otimes \mathcal{H}_{i_n,0}
\]
for $i \in \{1,2\}$.  Let $S_0 \in \mathcal{B}(\mathcal{K})$ be the isometry defined by
\[
S_0(\eta_1 \otimes \cdots \otimes \eta_n) = \xi_1 \otimes \eta_1 \otimes \cdots \otimes \eta_n
\]
for all $\eta_1 \otimes \cdots \otimes \eta_n \in \mathcal{K}_{2,0}$,
\[
S_0(\eta_1 \otimes \cdots \otimes \eta_n) = \xi_2 \otimes \eta_1 \otimes \cdots \otimes \eta_n
\]
for all $\eta_1 \otimes \cdots \otimes \eta_n \in \mathcal{K}_{1,0}$, and by extending by linearity and density.  
\par
Notice that $(\mathfrak{A}_1/\mathfrak{J}) \oplus \mathfrak{A}_2$ has a faithful representation on $\mathcal{K}_0$ given by 
\[
(A_1 \oplus A_2)(\eta_1 \otimes \cdots \otimes \eta_n) = \pi_{i,0}(A_i)\eta_1 \otimes \cdots \otimes \eta_n
\]
for all $\eta_1 \otimes \cdots \otimes \eta_n \in \mathcal{K}_i$, $A_1 \in \mathfrak{A}_1/\mathfrak{J}$, $A_2 \in \mathfrak{A}_2$, and $i \in \{1,2\}$.  Let $C^*((\mathfrak{A}_1/\mathfrak{J}) \oplus \mathfrak{A}_2, S_0)$ denote the C$^*$-subalgebra of $\mathcal{B}(\mathcal{K}_0)$ generated by $(\mathfrak{A}_1/\mathfrak{J}) \oplus \mathfrak{A}_2$ and $S_0$.  From this point onward, we will suppress the representations $\pi_{i,0}$ and view $\mathfrak{A}_1/\mathfrak{J}, \mathfrak{A}_2 \subseteq C^*((\mathfrak{A}_1/\mathfrak{J}) \oplus \mathfrak{A}_2, S_0)$ canonically.
\par
Notice $\mathcal{K}_0$ may be viewed canonically as a Hilbert subspace of $\mathcal{K}$ since $\mathcal{H}_{1,0} \subseteq \mathcal{H}_1$ and $\mathcal{H}_{2,0} = \mathcal{H}_2$.  By considering the actions of $\mathfrak{A}_1$, $\mathfrak{A}_2$, $S$, and $S^*$, it is easy to see that $\mathcal{K}_0$ is a reducing subspace of $C^*(\mathfrak{A}_1 \oplus \mathfrak{A}_2, S)$ since $\xi_1 \in \mathcal{H}_{1,0}$ and $\xi_2 \in \mathcal{H}_{2,0}$.
\par
Let $\pi' : \mathcal{B}(\mathcal{K}) \to \mathcal{B}(\mathcal{K}_0)$ be the compression of $\mathcal{B}(\mathcal{K})$ onto $\mathcal{B}(\mathcal{K}_0)$; that is, if $P_{\mathcal{K}_0}$ is the orthogonal projection of $\mathcal{K}$ onto $\mathcal{K}_0$, 
\[
\pi'(T) = P_{\mathcal{K}_0} T|_{\mathcal{K}_0}
\]
for all $T \in \mathcal{B}(\mathcal{K})$.  It is trivial to verify that
\[
\pi'(S) = S_0
\]
and
\[
\pi'(A_1 \oplus A_2) = (A_1 + \mathfrak{J}) \oplus A_2 \in C^*((\mathfrak{A}_1/\mathfrak{J}) \oplus \mathfrak{A}_2, S_0)
\]
for all $A_1 \oplus A_2 \in \mathfrak{A}_1 \oplus \mathfrak{A}_2$.  Since $\mathcal{K}_0 \subseteq \mathcal{K}$ is a reducing subspace of $C^*(\mathfrak{A}_1 \oplus \mathfrak{A}_2, S)$, $\pi'|_{C^*(\mathfrak{A}_1 \oplus \mathfrak{A}_2, S)}$ is a surjective $^*$-homomorphism.
\par
Let $\langle \mathfrak{J}\rangle_{C^*(\mathfrak{A}_1 \oplus \mathfrak{A}_2, S)}$ be the ideal of $C^*(\mathfrak{A}_1 \oplus \mathfrak{A}_2, S)$ generated by $\mathfrak{J} \subseteq \mathfrak{A}_1$.  Since $\pi_{1,0}(J) = 0$ for all $J \in \mathfrak{J}$, it is clear that 
\[
\langle \mathfrak{J}\rangle_{C^*(\mathfrak{A}_1 \oplus \mathfrak{A}_2, S)} \subseteq ker(\pi').
\]
The main result of this section is the following.
\end{cons}
\begin{thm}
\label{dse}
With the notation as in Construction \ref{directsumconstruction}, the sequence
\[
0\rightarrow\langle \mathfrak{J}\rangle_{C^*(\mathfrak{A}_1 \oplus \mathfrak{A}_2, S)}\rightarrow C^*(\mathfrak{A}_1 \oplus \mathfrak{A}_2, S)  \stackrel{\pi'}{\rightarrow} C^*((\mathfrak{A}_1/\mathfrak{J}) \oplus \mathfrak{A}_2, S_0) \rightarrow 0
\]
is exact.
\end{thm}
To prove Theorem \ref{dse} we will split the proof into several smaller results.  The proof begins by examining some basic structural facts about $C^*(\mathfrak{A}_1 \oplus \mathfrak{A}_2, S)$.  Next an action of the unit circle on $C^*(\mathfrak{A}_1 \oplus \mathfrak{A}_2, S)$ is defined which enables us to create a `Fourier series' for $C^*(\mathfrak{A}_1 \oplus \mathfrak{A}_2, S)$.  It is then easy to see that the sequence is exact if and only if the possible `Fourier coefficients' of elements from $ker(\pi')$ and $\langle\mathfrak{J}\rangle_{C^*(\mathfrak{A}_1 \oplus \mathfrak{A}_2, S)}$ agree.  This fact that the Fourier coefficients agree is proved by directly analyzing the structure of these `Fourier coefficients'.
\begin{lem}
\label{S*}
For all $i \in \{1,2\}$ and $A \in \mathfrak{A}_i$
\[
S^*AS = \langle A\xi_i, \xi_i\rangle_{\mathcal{H}_i} P_{\mathcal{K}_j}
\]
where $j \in \{1,2\}\setminus \{i\}$ and $P_{\mathcal{K}_j}$ is the orthogonal projection of $\mathcal{K}$ onto $\mathcal{K}_j$.  In addition, for all $i \in \{1,2\}$, $j \in \{1,2\}\setminus \{i\}$, and $A, B \in \mathfrak{A}_i$ 
\[
ASB = 0,\,\,\,\,\,\,\,\,\, AS^*B = 0, \,\,\,\,\,\,\,\,\, P_{\mathcal{K}_j}SA = SA, \,\,\,\,\,\,\,\,\,\mbox{and}\,\,\,\,\,\,\,\,\, AS^*P_{\mathcal{K}_j} = AS^*.
\]
Whence 
\[
span\left\{ (A_{1}S)(A_{2}S) \cdots (A_{n}S)A_{{n+1}}(S^*A_{{n+2}}) \cdots (S^*A_{{n+m+1}})  \, \left| \,
\begin{array}{l}
n,m\geq 0, \{i_k\}^{n+m+1}_{k=1} \subseteq \{1,2\},  \\
i_k \neq i_{k+1} \mbox{ for }k \in \{1,\ldots, n\}, A_k \in \mathfrak{A}_{i_k} 
\end{array}   
\right.  \right\}
\]
is dense in $C^*(\mathfrak{A}_1 \oplus \mathfrak{A}_2, S)$.
\end{lem}
\begin{proof}
Fix $i \in \{1,2\}$, let $j \in \{1,2\}\setminus \{i\}$, and let $A \in \mathfrak{A}_i$.   If $\zeta \in \mathcal{K}_i$ then $S\zeta \in \mathcal{K}_2$ so $S^*AS\zeta = 0$.  However, if $\eta_1 \otimes \cdots\otimes \eta_n \in \mathcal{L}_{j,n}$ then
\begin{eqnarray} 
S^*AS(\eta_1 \otimes \cdots\otimes \eta_n)\!\!\!\!&=&\!\!\!\!S^*A(\xi_i \otimes \eta_1 \otimes \cdots\otimes \eta_n) \nonumber\\
&=&\!\!\!\! S^* ( (A\xi_i) \otimes \eta_1 \otimes \cdots\otimes \eta_n) \nonumber\\
&=&\!\!\!\! \langle A\xi_i, \xi_i\rangle_{\mathcal{H}_i} \eta_1 \otimes \cdots\otimes \eta_n\nonumber
\end{eqnarray}
Whence, by linearity and density, $S^*AS = \langle A\xi_i, \xi_i\rangle_{\mathcal{H}_i} P_{\mathcal{K}_j}$.  
\par
Fix $i \in \{1,2\}$, $j \in \{1,2\}\setminus \{i\}$, and $A, B \in \mathfrak{A}_i$.  To see $ASB = 0$ notice $A(\mathcal{L}_{j,n}) = \{0\}$ and $B(\mathcal{L}_{j,n}) = \{0\}$ for all $n$.  However $S(B(\mathcal{L}_{i,n})) \subseteq \mathcal{L}_{j,n+1}$ and thus $ASB = 0$.  Similarly $AS^*B = 0$.
To see $P_{\mathcal{K}_j}SA = SA$ notice $SA(\mathcal{L}_{j,n}) = \{0\}$ and $SA(\mathcal{L}_{i,n}) \subseteq \mathcal{L}_{j,n+1} \subseteq \mathcal{K}_j$.  Hence $P_{\mathcal{K}_j}SA = SA$.  Similarly $AS^*P_{\mathcal{K}_j} = AS^*$.
\par
Using the fact that $Alg(\mathfrak{A}_1, \mathfrak{A}_2, S, S^*)$ is dense in $C^*(\mathfrak{A}_1 \oplus \mathfrak{A}_2, S)$, $\mathfrak{A}_1 \oplus \mathfrak{A}_2$ is unital, the fact that $P_{\mathcal{K}_j} \in \mathfrak{A}_j$ for all $j \in \{1,2\}$, and the above results, we obtain that the desired span is dense in $C^*(\mathfrak{A}_1 \oplus \mathfrak{A}_2, S)$.  
\end{proof}
The next step in the proof is to define a action of the unit circle $\mathbb{T}$ on $\mathcal{B}(\mathcal{K})$.  For each $\theta \in [0,2\pi)$ define $U_\theta \in \mathcal{B}(\mathcal{K})$ by
\[
U_\theta(\eta_1 \otimes \cdots \otimes \eta_n) = e^{-n\theta\sqrt{-1}} \eta_1 \otimes \cdots \otimes \eta_n
\]
(where we use $\sqrt{-1}$ instead of $i$ as we commonly use $i$ as an index) for all $\eta_1 \otimes \cdots \otimes \eta_n \in \mathcal{L}_{i,n}$ and extend by linearity and density (that is, $U_\theta$ is multiplication by $e^{-n\theta\sqrt{-1}}$ on tensors of $\mathcal{K}$ of length $n$).  It is clear that $U_\theta$ is a unitary operator with $U_\theta^* = U_{-\theta}$ and $U_{\theta} U_{\beta} = U_{\theta + \beta}$ (where we view $\theta + \beta \mod 2\pi$).  Define the $^*$-homomorphisms $\alpha_\theta : \mathcal{B}(\mathcal{K}) \to \mathcal{B}(\mathcal{K})$ by $\alpha_\theta(T) = U^*_\theta T U_\theta$ for all $T \in \mathcal{B}(\mathcal{K})$. 
\begin{lem}
\label{lemma2.2.4}
If $T \in C^*(\mathfrak{A}_1 \oplus \mathfrak{A}_2, S)$ and $\alpha_\theta(T) = T$ for all $\theta \in [0,2\pi)$ then 
\[
T\left(\mathcal{L}_{i,n}\right) \subseteq \mathcal{L}_{i,n}
\]
for all $i \in \{1,2\}$ and for all $n \in \mathbb{N}$.
\end{lem}
\begin{proof} 
First it is clear that
\[
\left(\bigoplus_{n = 1 \mod{2}} \mathcal{L}_{1,n}\right) \oplus \left(\bigoplus_{n = 2 \mod{2}} \mathcal{L}_{2,n}\right)
\]
and
\[
\left(\bigoplus_{n = 1 \mod{2}} \mathcal{L}_{2,n}\right) \oplus \left(\bigoplus_{n = 2 \mod{2}} \mathcal{L}_{1,n}\right)
\]
are reducing subspaces of $C^*(\mathfrak{A}_1 \oplus \mathfrak{A}_2, S)$ since each is invariant under $\mathfrak{A}_1$, $\mathfrak{A}_2$, $S$, and $S^*$.
\par
Suppose otherwise that there exists an $i \in \{1,2\}$, an $m \in \mathbb{N}$, and an $h \in \mathcal{L}_{i,m}$ so that $T(h) \notin \mathcal{L}_{i,m}$.  Without loss of generality suppose
\[
\mathcal{L}_{i,m} \subseteq \left(\bigoplus_{n = 1 \mod{2}} \mathcal{L}_{1,n}\right) \oplus \left(\bigoplus_{n = 2 \mod{2}} \mathcal{L}_{2,n}\right).
\]
Thus $T(h) \in \left(\bigoplus_{n = 1 \mod{2}} \mathcal{L}_{1,n}\right) \oplus \left(\bigoplus_{n = 2 \mod{2}} \mathcal{L}_{2,n}\right)$.  Write
\[
T(h) = \bigoplus_{j \geq 1} h_j
\]
where $h_j \in \mathcal{L}_{1,j}$ if $j$ is odd and $h_j \in \mathcal{L}_{2,j}$ when $j$ is even.  Since $T(h) \notin \mathcal{L}_{i,m}$, there exists a $k \in \mathbb{N}\setminus \{m\}$ such that $h_k \neq 0$.  However
\[
\bigoplus_{j \geq 1} h_j = T(h) = \alpha_\theta(T)h =  U_{-\theta}TU_\theta h = U_{-\theta}T e^{-m\theta\sqrt{-1}} h = e^{-m\theta\sqrt{-1}} U_{-\theta}\left(\bigoplus_{j \geq 1} h_j\right) = \bigoplus_{j \geq 1} (e^{-(m-j)\theta\sqrt{-1}})h_j
\]
for all $\theta \in [0,2\pi)$.  Therefore $h_k = e^{-(m-k)\theta\sqrt{-1}} h_k$ for all $\theta \in [0,2\pi)$. As $k \neq m$ and $h_k \neq 0$, this is an impossibility.  
\end{proof}
Next we obtain some important results by considering the action of $\alpha_\theta$ on $C^*(\mathfrak{A}_1 \oplus \mathfrak{A}_2, S)$.
\begin{lem}
For all $\theta \in [0,2\pi)$ and all $A \in \mathfrak{A}_1 \oplus \mathfrak{A}_2$ 
\[
\alpha_\theta(S) = e^{\theta\sqrt{-1}} S \,\,\,\,\,\,\,\,\,\mbox{ and }\,\,\,\,\,\,\,\,\,\alpha_\theta(A) = A.
\]
Therefore $\theta \mapsto \alpha_\theta(T)$ is a continuous map for all $T \in C^*(\mathfrak{A}_1 \oplus \mathfrak{A}_2, S)$.  Hence the map 
\[
\mathcal{E} : C^*(\mathfrak{A}_1 \oplus \mathfrak{A}_2, S) \to C^*(\mathfrak{A}_1 \oplus \mathfrak{A}_2, S)
\]
given by
\[
\mathcal{E}(T) := \frac{1}{2\pi} \int^{2\pi}_0 \alpha_\theta(T) d\theta
\]
is a well-defined, contractive linear map with the property that $\alpha_\theta(\mathcal{E}(T)) = \mathcal{E}(T)$ for all $T \in C^*(\mathfrak{A}_1 \oplus \mathfrak{A}_2, S)$ and $\theta \in [0,2\pi)$.
\end{lem}
\begin{proof}
The fact that $\alpha_\theta(A) = A$ for all $\theta \in [0,2\pi)$ and $A \in \mathfrak{A}_1 \oplus \mathfrak{A}_2$ comes from the fact that each $\mathcal{L}_{i,n}$ is an invariant subspace of $\mathfrak{A}_1 \oplus \mathfrak{A}_2$ and thus $U_\theta \in (\mathfrak{A}_1 \oplus \mathfrak{A}_2)'$ (the commutant of $\mathfrak{A}_1 \oplus \mathfrak{A}_2$).  Notice for each $i \in \{1,2\}$ and each $\eta_1 \otimes \cdots \otimes \eta_n \in \mathcal{L}_{i,n}$ that
\begin{eqnarray} 
\alpha_\theta(S)(\eta_1 \otimes \cdots \otimes \eta_n)\!\!\!\!&=&\!\!\!\! U_{-\theta} S\left(e^{-n\theta\sqrt{-1}}\eta_1 \otimes \cdots \otimes \eta_n \right)\nonumber\\
&=&\!\!\!\!  U_{-\theta} \left(e^{-n\theta\sqrt{-1}} \xi_j \otimes\eta_1 \otimes \cdots \otimes \eta_n \right) \nonumber\\
&=&\!\!\!\! e^{(n+1)\theta\sqrt{-1}}e^{-n\theta\sqrt{-1}} \xi_j \otimes\eta_1 \otimes \cdots \otimes \eta_n \nonumber\\
&=&\!\!\!\! e^{\theta\sqrt{-1}}S(\eta_1\otimes\cdots\otimes\eta_n) \nonumber
\end{eqnarray}
where $j \in \{1,2\} \setminus \{i\}$.  Whence $\alpha_\theta(S) = e^{\theta\sqrt{-1}}S$ by linearity and density.
\par
To see $\theta \mapsto \alpha_\theta(T)$ is a continuous map for all $T \in C^*(\mathfrak{A}_1 \oplus \mathfrak{A}_2, S)$, notice the result holds for all $T \in Alg(\mathfrak{A}_1 \oplus \mathfrak{A}_2, S, S^*)$ by the above results.  Since each $\alpha_\theta$ is a contraction and $Alg(\mathfrak{A}_1 \oplus \mathfrak{A}_2, S, S^*)$ is dense in $C^*(\mathfrak{A}_1 \oplus \mathfrak{A}_2, S)$, the result follows.  
\par
The fact that $\mathcal{E}$ is a well-defined, contractive linear map is then trivial and the fact that $\alpha_\theta(\mathcal{E}(T)) = \mathcal{E}(T)$ for all $T \in C^*(\mathfrak{A}_1 \oplus \mathfrak{A}_2, S)$ follows from the fact that $\alpha_\theta \circ \alpha_\beta = \alpha_{\theta + \beta}$ (as $U_{\theta} U_{\beta} = U_{\theta + \beta}$) and the fact that the Lebesgue measure on the unit circle is translation invariant.  
\end{proof}
\par
The map $\mathcal{E}$ has many other properties (e.g. it is positive and faithful) that will not be needed.  What will be needed is the fact that $\mathcal{E}$ allows us to create a `Fourier series' for elements of $C^*(\mathfrak{A}_1 \oplus \mathfrak{A}_2, S)$.
\begin{lem}
\label{fourier}
For all $T \in C^*(\mathfrak{A}_1 \oplus \mathfrak{A}_2, S)$ and $n \in \mathbb{N}$ define
\[
\Sigma_n(T) := \sum^n_{j=0} \left( 1- \frac{j}{n+1}\right) (S^*)^j \mathcal{E}(S^j T) + \sum^n_{j=1}\left( 1- \frac{j}{n+1}\right)  \mathcal{E}(T(S^*)^j)S^j.
\]
Then $\lim_{n\to\infty}\left\|T - \Sigma_n(T)\right\| =0$.
\end{lem}
\begin{proof}
Notice for all $T \in C^*(\mathfrak{A}_1 \oplus \mathfrak{A}_2, S)$ that
\begin{eqnarray} 
\Sigma_n(T)\!\!\!\!&=&\!\!\!\! \frac{1}{2\pi} \int^{2\pi}_0\left(\sum^n_{j=0} \left( 1- \frac{j}{n+1}\right) (S^*)^j \alpha_\theta(S^j T) + \sum^n_{j=1}\left( 1- \frac{j}{n+1}\right)  \alpha_\theta(T(S^*)^j)S^j \right) d\theta\nonumber\\
&=&\!\!\!\! \frac{1}{2\pi} \int^{2\pi}_0\left(\sum^n_{j=0} \left( 1- \frac{j}{n+1}\right) e^{j\theta\sqrt{-1}}(S^*)^jS^j \alpha_\theta(T) + \sum^n_{j=1}\left( 1- \frac{j}{n+1}\right) e^{-j\theta\sqrt{-1}} \alpha_\theta(T) (S^*)^jS^j \right) d\theta \nonumber\\
&=&\!\!\!\! \frac{1}{2\pi} \int^{2\pi}_0\left(\sum^n_{j=-n} \left( 1- \frac{|j|}{n+1}\right) e^{j\theta\sqrt{-1}}\right)\alpha_\theta(T) d\theta \nonumber\\
&=&\!\!\!\! \frac{1}{2\pi} \int^{2\pi}_0\sigma_n(\theta)\alpha_\theta(T) d\theta \nonumber
\end{eqnarray}
where $\sigma_n(\theta) := \sum^n_{j=-n} \left( 1- \frac{|j|}{n+1}\right) e^{j\theta\sqrt{-1}}$ is Fej\'{e}r's kernel.  Recall $\left\{\frac{1}{2\pi} \sigma_n(\theta) d\theta\right\}_{n\geq 1}$ define probability measures on $\mathbb{T}$ that converge weak$^*$ (from $C(\mathbb{T})$) to the point mass at 0.  Thus
\[
\left\|\Sigma_n(T)\right\| \leq \frac{1}{2\pi}\int^{2\pi}_0 \sigma_n(\theta) \left\|\alpha_\theta(T)\right\|d\theta = \left\|T\right\|
\]
for all $n \in \mathbb{N}$.  Since $\mathcal{E}$ and thus $\Sigma_n$ is linear for all $n \in \mathbb{N}$ and each $\Sigma_n$ is a contraction, it suffices to prove the result on a set whose span is dense; namely $\{A_1SA_2S \cdots A_nSBS^*C_1S^*\cdots S^*B_m \, \mid \, n,m\geq 0, A_i, B, C_j \in \mathfrak{A}_1 \oplus \mathfrak{A}_2\}$ by Lemma \ref{S*}.
\par
To complete the proof, notice if $T = A_1SA_2S \cdots A_nSBS^*C_1S^*\cdots S^*B_m$ where $n,m\geq 0$ and $A_i, B, C_j \in \mathfrak{A}_1 \oplus \mathfrak{A}_2$ then
\[
\mathcal{E}(S^kT) = \left\{
\begin{array}{ll}
S^kT & \mbox{if } k+n = m \\
0 & \mbox{otherwise } 
\end{array} 
\right.
\]
and
\[
\mathcal{E}(T(S^*)^k) = \left\{
\begin{array}{ll}
T(S^*)^k & \mbox{if } n = m+k \\
0 & \mbox{otherwise } 
\end{array} 
\right.
\]
Whence $\Sigma_k(T) = \left(1 - \frac{|n-m|}{k+1}\right) T$ for all $k \geq |n-m|$ which clearly converges to $T$ as $k \to \infty$.  
\end{proof}
\begin{disc}
\label{disc2.2.7}
Let $\pi' : C^*(\mathfrak{A}_1 \oplus \mathfrak{A}_2, S) \to C^*((\mathfrak{A}_1/\mathfrak{J}) \oplus \mathfrak{A}_2, S_0)$ be the compression map under consideration in Theorem \ref{dse}.  To prove Theorem \ref{dse} it suffices to prove $ker(\pi') \subseteq \langle \mathfrak{J} \rangle_{C^*(\mathfrak{A}_1 \oplus \mathfrak{A}_2, S)}$.  Notice each element of $\mathfrak{J}$ is fixed by each $\alpha_\theta$.  Thus $\mathcal{E}$ maps the algebraic ideal generated by $\mathfrak{J}$ in $Alg(\mathfrak{A}, S, S^*)$ into $\langle \mathfrak{J}\rangle_{C^*(\mathfrak{A}_1 \oplus \mathfrak{A}_2, S)}$ and thus $\mathcal{E}\left(\langle \mathfrak{J}\rangle_{C^*(\mathfrak{A}_1 \oplus \mathfrak{A}_2, S)}\right) \subseteq \langle \mathfrak{J}\rangle_{C^*(\mathfrak{A}_1 \oplus \mathfrak{A}_2, S)}$.  
\par
Recall from Construction \ref{directsumconstruction} that $T \in C^*(\mathfrak{A}_1 \oplus \mathfrak{A}_2, S) \cap ker(\pi')$ if and only if $T$ acts as the zero operator on $\mathcal{K}_0 \subseteq \mathcal{K}$.  As $\mathcal{K}_0 \subseteq \mathcal{K}$ is an invariant subspace for each $U_\theta$, we see $T \in ker(\pi')$ if and only if $\alpha_\theta(T) \in ker(\pi')$ for all $\theta \in [0,2\pi)$.  Hence $\mathcal{E}(ker(\pi')) \subseteq ker(\pi')$.  Moreover we notice $\mathcal{E}\left(\langle \mathfrak{J}\rangle_{C^*(\mathfrak{A}_1 \oplus \mathfrak{A}_2, S)}\right) \subseteq \mathcal{E}(ker(\pi'))$.  The conclusion of the proof will be obtain by showing $\mathcal{E}\left(\langle \mathfrak{J}\rangle_{C^*(\mathfrak{A}_1 \oplus \mathfrak{A}_2, S)}\right) = \mathcal{E}(ker(\pi'))$ due to the following lemma.
\end{disc}
\begin{lem}
\label{lemma2.2.8}
If $\mathcal{E}\left(\langle \mathfrak{J}\rangle_{C^*(\mathfrak{A}_1 \oplus \mathfrak{A}_2, S)}\right) = \mathcal{E}(ker(\pi'))$ then $ker(\pi') = \langle \mathfrak{J}\rangle_{C^*(\mathfrak{A}_1 \oplus \mathfrak{A}_2, S)}$.
\end{lem}
\begin{proof}
By Construction \ref{directsumconstruction} it suffices to show that $ker(\pi') \subseteq \langle \mathfrak{J}\rangle_{C^*(\mathfrak{A}_1 \oplus \mathfrak{A}_2, S)}$.  Let $T \in ker(\pi')$.  Recall $T = \lim_{n\to\infty} \Sigma_n(T)$ by Lemma \ref{fourier}. Moreover $S^kT\in ker(\pi')$ and $T(S^*)^k \in ker(\pi')$ for all $k \geq 0$ as $T \in ker(\pi')$.  Whence $\mathcal{E}(S^kT), \mathcal{E}(T(S^*)^k) \in \mathcal{E}(ker(\pi')) = \mathcal{E}\left(\langle \mathfrak{J}\rangle_{C^*(\mathfrak{A}_1 \oplus \mathfrak{A}_2, S)}\right) \subseteq \langle \mathfrak{J}\rangle_{C^*(\mathfrak{A}_1 \oplus \mathfrak{A}_2, S)}$ for all $k\geq 0$.  This implies that $\Sigma_n(T) \in \langle \mathfrak{J}\rangle_{C^*(\mathfrak{A}_1 \oplus \mathfrak{A}_2, S)}$ for all $n$ and thus $T \in \langle \mathfrak{J}\rangle_{C^*(\mathfrak{A}_1 \oplus \mathfrak{A}_2, S)}$. 
\end{proof}
\begin{disc}
\label{directideal}
To begin the process of showing $\mathcal{E}\left(\langle \mathfrak{J}\rangle_{C^*(\mathfrak{A}_1 \oplus \mathfrak{A}_2, S)}\right) = \mathcal{E}(ker(\pi'))$ we will examine the structure of $\langle \mathfrak{J}\rangle_{C^*(\mathfrak{A}_1 \oplus \mathfrak{A}_2, S)}$.  Notice if $J \in \mathfrak{J}$ then $JS = 0$ as $J\xi_1 = 0$.  Whence $S^*J = 0$ for all $J \in \mathfrak{J}$.  Therefore, using the properties that the algebraic ideal generated by $\mathfrak{J}$ in $Alg(\mathfrak{A}_1, \mathfrak{A}_2, S, S^*)$ is dense in $\langle \mathfrak{J}\rangle_{C^*(\mathfrak{A}_1 \oplus \mathfrak{A}_2, S)}$ and the results and ideas from Lemma \ref{S*}, the span of
\[
\left\{(A_nS)(A_{n-1}S) \cdots (A_1S)J(S^*B_1)(S^*B_2)\cdots (S^*B_m) \, \left| \, 
\begin{array}{l}
n,m\geq 0, J \in \mathfrak{J},\\
A_i,B_j \in \mathfrak{A}_1 \mbox{ if }i,j=0 \mod 2, \\
A_i,B_j \in \mathfrak{A}_2 \mbox{ if }i,j=1 \mod 2
\end{array} 
\right. \right\}
\]
is dense in $\langle \mathfrak{J}\rangle_{C^*(\mathfrak{A}_1 \oplus \mathfrak{A}_2, S)}$.
\end{disc}
\begin{nota}
For each $n \geq 0$ let
\[
\mathfrak{J}_{(n)} := span\left\{(A_nS)(A_{n-1}S) \cdots (A_1S)J(S^*B_1)(S^*B_2)\cdots (S^*B_n) \, \left| \, 
\begin{array}{l}
J \in \mathfrak{J},\\
A_i,B_j \in \mathfrak{A}_1 \mbox{ if }i,j=0 \mod 2, \\
A_i,B_j \in \mathfrak{A}_2 \mbox{ if }i,j=1 \mod 2
\end{array} 
\right. \right\}
\]
and
\[
\mathfrak{A}_{1,(n)} := span\left\{(A_nS)(A_{n-1}S) \cdots (A_1S)A(S^*B_1)(S^*B_2)\cdots (S^*B_n) \, \left| \, 
\begin{array}{l}
A \in \mathfrak{A}_1,\\
A_i,B_j \in \mathfrak{A}_1 \mbox{ if }i,j=0 \mod 2, \\
A_i,B_j \in \mathfrak{A}_2 \mbox{ if }i,j=1 \mod 2
\end{array} 
\right. \right\}
\]
and
\[
\mathfrak{A}_{2,(n)} := span\left\{(A_nS)(A_{n-1}S) \cdots (A_1S)A(S^*B_1)(S^*B_2)\cdots (S^*B_n) \, \left| \, 
\begin{array}{l}
A \in \mathfrak{A}_2,\\
A_i,B_j \in \mathfrak{A}_2 \mbox{ if }i,j=0 \mod 2, \\
A_i,B_j \in \mathfrak{A}_1 \mbox{ if }i,j=1 \mod 2
\end{array} 
\right. \right\}.
\]
\end{nota}
\begin{lem}
\label{lemma2.2.9}
The span of $\bigcup_{n\geq 0} \mathfrak{J}_{(n)}$ is dense in $\mathcal{E}\left(\langle \mathfrak{J}\rangle_{C^*(\mathfrak{A}_1 \oplus \mathfrak{A}_2, S)}\right)$ and the span of $\bigcup_{n\geq 0} (\mathfrak{A}_{1,(n)} \cup \mathfrak{A}_{2,(n)})$ is dense in $\mathcal{E}(C^*(\mathfrak{A}_1 \oplus \mathfrak{A}_2, S))$.
\end{lem}
\begin{proof}
We will only prove the first claim as the second follows verbatim with the aid of Lemma \ref{S*}.  It is clear $\bigcup_{n\geq 0} \mathfrak{J}_{(n)} \subseteq \mathcal{E}\left(\langle \mathfrak{J}\rangle_{C^*(\mathfrak{A}_1 \oplus \mathfrak{A}_2, S)}\right)$.
\par
Let $T \in \mathcal{E}\left(\langle \mathfrak{J}\rangle_{C^*(\mathfrak{A}_1 \oplus \mathfrak{A}_2, S)}\right)$ and let $\epsilon > 0$.  As $T \in \langle \mathfrak{J}\rangle_{C^*(\mathfrak{A}_1 \oplus \mathfrak{A}_2, S)}$ there exists an 
\[
R \in span\left\{(A_nS)(A_{n-1}S) \cdots (A_1S)J(S^*B_1)(S^*B_2)\cdots (S^*B_m) \, \left| \, 
\begin{array}{l}
n,m\geq 0, J \in \mathfrak{J},\\
A_i,B_j \in \mathfrak{A}_1 \mbox{ if }i,j=0 \mod 2, \\
A_i,B_j \in \mathfrak{A}_2 \mbox{ if }i,j=1 \mod 2
\end{array} 
\right. \right\}
\]
such that $\left\|T-R\right\| < \epsilon$.  Then $\mathcal{E}(T) = T$ since $T \in \mathcal{E}\left(\langle \mathfrak{J}\rangle_{C^*(\mathfrak{A}_1 \oplus \mathfrak{A}_2, S)}\right)$ and thus
\[
\left\|T - \mathcal{E}(R)\right\| = \left\|\mathcal{E}(T - R)\right\| \leq \left\|T - R\right\|  < \epsilon.
\]
Clearly
\begin{eqnarray} 
&&\!\!\!\! \mathcal{E}\left(span\left\{(A_nS)(A_{n-1}S) \cdots (A_1S)J(S^*B_1)(S^*B_2)\cdots (S^*B_m) \, \left| \, 
\begin{array}{l}
n,m\geq 0, J \in \mathfrak{J},\\
A_i,B_j \in \mathfrak{A}_1 \mbox{ if }i,j=0 \mod 2, \\
A_i,B_j \in \mathfrak{A}_2 \mbox{ if }i,j=1 \mod 2
\end{array} 
\right. \right\}\right) \nonumber\\
&=&\!\!\!\! span\left\{(A_nS)(A_{n-1}S) \cdots (A_1S)J(S^*B_1)(S^*B_2)\cdots (S^*B_n) \, \left| \, 
\begin{array}{l}
n\geq 0, J \in \mathfrak{J},\\
A_i,B_j \in \mathfrak{A}_1 \mbox{ if }i,j=0 \mod 2, \\
A_i,B_j \in \mathfrak{A}_2 \mbox{ if }i,j=1 \mod 2
\end{array} 
\right. \right\} \nonumber\\
&=&\!\!\!\! span\left(\bigcup_{n\geq 0} \mathfrak{J}_{(n)}\right). \nonumber
\end{eqnarray}
Thus $\mathcal{E}(R) \in \overline{span\left(\bigcup_{n\geq 0} \mathfrak{J}_{(n)}\right)}$ which completes the claim. 
\end{proof}
\begin{disc}
\label{anjn}
Now we will examine how $\mathfrak{A}_{1,(n)}$, $\mathfrak{A}_{2,(n)}$, and $\mathfrak{J}_{(n)}$ act on $\mathcal{K}$. We will begin with the analysis of $\mathfrak{A}_{1,(n)}$ and $\mathfrak{J}_{(n)}$ as the analysis of $\mathfrak{A}_{2,(n)}$ will be similar. 
\par
Since $\mathcal{E}(T) = T$ for all $T \in \mathfrak{A}_{1,(n)}$ ($T \in \mathfrak{J}_{(n)}$), $T(\mathcal{L}_{i,m}) \subseteq \mathcal{L}_{i,m}$ for all $i \in \{1,2\}$ and $m \in \mathbb{N}$ by Lemma \ref{lemma2.2.4}.  Fix 
\[
T = (A_nS)(A_{n-1}S) \cdots (A_1S)A(S^*B_1)(S^*B_2)\cdots (S^*B_n)
\]
where $A_i,B_j \in \mathfrak{A}_1$ if $i,j=0 \mod 2$, $A_i,B_j \in \mathfrak{A}_2$ if $i,j=1 \mod 2$, and $A \in \mathfrak{A}_1$ ($A \in \mathfrak{J}$).  If $k \in \{1,2\}$ and $k = n \mod 2$ then $T(\mathcal{L}_{k,m}) = \{0\}$ for all $m \in \mathbb{N}$.  Fix $k \in \{1,2\}$ with $k \neq n \mod 2$ and let $\ell \in \{1,2\}\setminus \{k\}$.  Then for all $\eta_1 \otimes \cdots \otimes \eta_{n+1} \in \mathcal{L}_{k,n+1}$ 
\begin{eqnarray} 
&&\!\!\!\! T(\eta_1 \otimes \cdots \otimes \eta_{n+1}) \nonumber\\
&=&\!\!\!\! \langle B_n\eta_1, \xi_k\rangle_{\mathcal{H}_k}((A_nS)(A_{n-1}S) \cdots (A_1S)A(S^*B_1)(S^*B_2)\cdots (S^*B_{n-1}))(\eta_2 \otimes \cdots \otimes \eta_{n+1}) \nonumber\\
&\vdots&\!\!\!\!  \nonumber\\
&=&\!\!\!\! \left(\prod_{
\begin{array}{c}
j \in \{1,\ldots, n\} \\
j \mbox{ even}
\end{array} 
} \langle B_{j} \eta_{n-j+1}, \xi_1\rangle_{\mathcal{H}_1}\right)\left(\prod_{\begin{array}{c}
j \in \{1,\ldots, n\} \\
j \mbox{ odd}
\end{array}} \langle B_j \eta_{n-j+1}, \xi_2\rangle_{\mathcal{H}_2}\right) (A_nS)(A_{n-1}S) \cdots (A_1S)A\eta_{n+1}\nonumber\\
&=&\!\!\!\! \langle \eta_1\otimes \cdots \otimes \eta_n, B_n^*\xi_k \otimes B^*_{n-1} \xi_\ell \otimes \cdots \otimes B_2^*\xi_1 \otimes B_1^*\xi_2 \rangle_{ \mathcal{L}_{k,n}}(A_nS)(A_{n-1}S) \cdots (A_2S)(A_1\xi_2 \otimes A\eta_{n+1})\nonumber\\
&\vdots&\!\!\!\!  \nonumber\\
&=&\!\!\!\! \langle \eta_1\otimes \cdots \otimes \eta_n, B_n^*\xi_k \otimes B^*_{n-1} \xi_\ell \otimes \cdots \otimes B_2^*\xi_1 \otimes B_1^*\xi_2 \rangle_{ \mathcal{L}_{k,n}} (A_n\xi_k \otimes A_{n-1} \xi_\ell \otimes \cdots \otimes A_{2}\xi_1 \otimes A_1\xi_2 \otimes A\eta_{n+1}). \nonumber
\end{eqnarray}
Moreover $T$ acts as the zero operator on elements of $\mathcal{L}_{k,m}$ for all $m \leq n$ since the $m^{\mbox{th}}$ $S^*$ will act on an element of $\mathcal{L}_{1,1} \oplus \mathcal{L}_{2,1}$.  Finally if $m \geq n+1$ notice $\mathcal{L}_{k,m} = \mathcal{L}_{k,n+1} \otimes \mathcal{L}_{k, m-n-1}$ if $n+1$ is even and $\mathcal{L}_{k,m} = \mathcal{L}_{k,n+1} \otimes \mathcal{L}_{\ell, m-n-1}$ if $n +1$ is odd. Therefore, if $R = T|_{\mathcal{L}_{k,n+1}}$, it is easy to see that $T$ acts on $\mathcal{L}_{k,m}$ as $R \otimes I_{\mathcal{L}_{k, m-n-1}}$ when $n+1$ is even and $m \geq n+1$ and $T$ acts on $\mathcal{L}_{k,m}$ as $R \otimes I_{\mathcal{L}_{\ell, m-n-1}}$ when $n+1$ is odd and $m \geq n+1$.
\par
For $i \in \{1,2\}$ let $\mathcal{N}_i = \overline{\mathfrak{A}_i \xi_i}$ which is a Hilbert subspace of $\mathcal{H}_{i,0} \subseteq \mathcal{H}_i$ containing $\xi_i$.  Thus, for a fixed $k \in \{1,2\}$ with $k \neq n \mod 2$ and with $\ell \in \{1,2\}\setminus \{k\}$, by restricting to $\mathcal{L}_{k,n}$ and taking limits of elements of $\mathfrak{A}_{1,(n)}$ ($\mathfrak{J}_{(n)}$) over $A_1, \ldots, A_n, B_1, \ldots, B_n$ where $A_i,B_j \in \mathfrak{A}_1$ if $i,j=0 \mod 2$ and $A_i,B_j \in \mathfrak{A}_2$ if $i,j=1 \mod 2$, every operator in $\mathcal{B}(\mathcal{L}_{k,n})$ of the form
\[
\eta_1 \otimes \cdots \otimes \eta_{n+1} \mapsto \langle \eta_1 \otimes \cdots \otimes \eta_n, \zeta_1 \otimes \cdots \otimes \zeta_n\rangle_{\mathcal{L}_{k,n}} (\zeta'_1 \otimes \cdots \otimes \zeta'_n \otimes A\eta_{n+1})
\]
where $\zeta_i, \zeta'_j \in \mathcal{N}_k$ if $i,j = 1 \mod 2$, $\zeta_i, \zeta'_j \in \mathcal{N}_\ell$ if $i,j = 0 \mod 2$, and $A \in \mathfrak{A}_1$ ($A \in \mathfrak{J}$) may be obtained. 
\par
For completion we will describe the action of $\mathfrak{A}_{2,(n)}$ on $\mathcal{K}$.  Since $\mathcal{E}(T) = T$ for all $T \in \mathfrak{A}_{2,(n)}$, $T(\mathcal{L}_{i,m}) \subseteq \mathcal{L}_{i,m}$ for all $i \in \{1,2\}$ and $m \in \mathbb{N}$ by Lemma \ref{lemma2.2.4}.  Fix 
\[
T = (A_nS)(A_{n-1}S) \cdots (A_1S)A(S^*B_1)(S^*B_2)\cdots (S^*B_n)
\]
where $A_i,B_j \in \mathfrak{A}_2$ if $i,j=0 \mod 2$, $A_i,B_j \in \mathfrak{A}_1$ if $i,j=1 \mod 2$, and $A \in \mathfrak{A}_2$.  If $k \in \{1,2\}$ and $k \neq n \mod 2$ then $T(\mathcal{L}_{k,m}) = \{0\}$ for all $m \in \mathbb{N}$.  Fix $k \in \{1,2\}$ with $k = n \mod 2$ and let $\ell \in \{1,2\}\setminus \{k\}$.  Then for all $\eta_1 \otimes \cdots \otimes \eta_{n+1} \in \mathcal{L}_{k,n+1}$ 
\begin{eqnarray} 
&&\!\!\!\! T(\eta_1 \otimes \cdots \otimes \eta_{n+1}) \nonumber\\
&=&\!\!\!\! \langle \eta_1\otimes \cdots \otimes \eta_n, B_n^*\xi_k \otimes B^*_{n-1} \xi_\ell \otimes \cdots \otimes B_2^*\xi_2 \otimes B_1^*\xi_1 \rangle_{ \mathcal{L}_{k,n}} (A_n\xi_k \otimes A_{n-1} \xi_\ell \otimes \cdots \otimes A_{2}\xi_2 \otimes A_1\xi_1 \otimes A\eta_{n+1}). \nonumber
\end{eqnarray}
Moreover $T$ acts as the zero operator on elements of $\mathcal{L}_{k,m}$ for all $m \leq n$ since the $m^{\mbox{th}}$ $S^*$ will act on an element of $\mathcal{L}_{1,1} \oplus \mathcal{L}_{2,1}$.  Finally, if $R = T|_{\mathcal{L}_{k,n+1}}$ it is easy to see that $T$ acts on $\mathcal{L}_{k,m}$ as $R \otimes I_{\mathcal{L}_{k, m-n-1}}$ when $n+1$ is even and $m \geq n+1$ and $T$ acts on $\mathcal{L}_{k,m}$ as $R \otimes I_{\mathcal{L}_{\ell, m-n-1}}$ when $n+1$ is odd and $m \geq n+1$.
\par
Lastly notice if $T \in \mathfrak{A}_{1,(n)}$, $R \in\mathfrak{A}_{2,(n)}$, $k,\ell \in \{1,2\}$, $k = n\mod 2$, and $\ell \neq n\mod 2$ then the actions of $T$ and $R$ are completely determined by their actions on $\mathcal{L}_{1,n+1} \oplus \mathcal{L}_{2,n+1}$ with $T(\mathcal{L}_{k,m}) = \{0\}$ for all $m\in \mathbb{N}$, $R(\mathcal{L}_{\ell,m}) = \{0\}$ for all $m \in \mathbb{N}$, and
\[
\left\|T + R\right\| = \max\{ \left\|T|_{\mathcal{L}_{\ell, n+1}}\right\|, \left\|R|_{\mathcal{L}_{k, n+1}}\right\| \}.
\]
\par
The above structure will be important as we will consider the restriction of $\mathcal{E}\left(\langle \mathfrak{J}\rangle_{C^*(\mathfrak{A}_1 \oplus \mathfrak{A}_2, S)}\right)$ and $\mathcal{E}(ker(\pi'))$ to the subspaces $\mathcal{L}_{1,n} \oplus \mathcal{L}_{2,n}$ of $\mathcal{K}$.  For $m,n \in \mathbb{N}$ and $i \in \{1,2\}$ let $P_{i,m}$ be the orthogonal projection of $\mathcal{K}$ onto $\mathcal{L}_{i,m}$, let $P_m$ be the orthogonal projection of $\mathcal{K}$ onto $\mathcal{L}_{1,m} \oplus \mathcal{L}_{2,m}$ (so $P_m = P_{1,m} + P_{2,m}$), and let $Q_n = \sum^n_{j=1} P_j$ which is the orthogonal projection of $\mathcal{K}$ onto $\bigoplus^n_{k = 1} (\mathcal{L}_{1,k} \oplus \mathcal{L}_{2,k})$.  Therefore, using the above discussion, $\mathfrak{A}_{1,(n)} P_m = \{0\}$ and $\mathfrak{A}_{2,(n)} P_m = \{0\}$ for all $m \leq n$ and if $m > n+1$ then each element $T \in \mathfrak{A}_{1,(n)} \cup \mathfrak{A}_{2,(n)}$ acts on $\mathcal{L}_{1,m}$ by $(P_{1,n+1}TP_{1,n+1}) \otimes I$ where $I$ is the appropriate identity and acts on $\mathcal{L}_{2,m}$ by $(P_{2,n+1}TP_{2,n+1}) \otimes I$ where $I$ is the appropriate identity. Using the $P_n$'s and the above information about the actions of $\mathfrak{A}_{1,(n)}$, $\mathfrak{A}_{2,(n)}$, and $\mathfrak{J}_{(n)}$ on $\mathcal{L}_{1,n} \oplus \mathcal{L}_{2,n}$, we obtain the following.
\end{disc}
\begin{lem}
\label{complemma}
The set $Q_n \left(span \left( \bigcup_{m < n} \mathfrak{J}_{(m)} \right)\right) Q_n$ is dense in $Q_n \mathcal{E}(\ker(\pi')) Q_n$ for all $n \in \mathbb{N}$.
\end{lem}
\begin{proof}
Clearly $Q_n \left(span \left( \bigcup_{m < n} \mathfrak{J}_{(m)} \right)\right) Q_n \subseteq Q_n \mathcal{E}(\ker(\pi')) Q_n$ for all $n \in \mathbb{N}$.  Thus it suffices to show that each element of $\mathcal{E}(\ker(\pi'))$ can be approximated uniformly on $Q_m \mathcal{K}$ by an element of $span \left( \bigcup_{m < n} \mathfrak{J}_{(m)} \right)$.  We proceed by induction on $n$.
\par
Let $T \in \mathcal{E}(ker(\pi'))$ and let $\epsilon > 0$.  As $T \in \mathcal{E}(C^*(\mathfrak{A}_1 \oplus \mathfrak{A}_2, S))$ Lemma \ref{lemma2.2.9} implies that there exists an $m \in \mathbb{N}$, $T_{1,j} \in \mathfrak{A}_{1,(j)}$, and $T_{2,j} \in \mathfrak{A}_{2,(j)}$ such that $\left\|T - \sum^2_{i=1}\sum^m_{j=0} T_{i,j}\right\| < \epsilon$.  Therefore
\[
\left\|P_1TP_1 - P_1T_{1,0}P_1 - P_1T_{2,0}P_1\right\| = \left\|P_1\left(T - \sum^2_{i=1}\sum^m_{j=0} T_{i,j}\right)P_1\right\| < \epsilon
\]
as $\mathfrak{A}_{i,(j)} P_1 = \{0\}$ for all $j \geq 1$ and $i \in \{1,2\}$.  Note $T_{1,0} \in \mathfrak{A}_{1,(0)} = \mathfrak{A}_1$ and $T_{2,0} \in \mathfrak{A}_{2,(0)} = \mathfrak{A}_2$.  However $P_1TP_1(\mathcal{H}_{1,0}) = \{0\}$ as $T \in ker(\pi')$ and $P_1T_{2,0}P_1(\mathcal{H}_{1,0}) = \{0\}$ as $T_{2,0} \in \mathfrak{A}_2$.  Whence 
\[
\left\|T_{1,0} h\right\| = \left\|P_1TP_1h - P_1T_{1,0}P_1h - P_1T_{2,0}P_1h\right\| \leq \epsilon \left\|h\right\|
\]
for all $h \in \mathcal{H}_{1,0}$.  Since $\mathfrak{A}_1$ act on $\mathcal{H}_{1,0}$ via $\pi_{1,0} \circ q$, $\left\|\pi_{1,0}(q(T_0))\right\| < \epsilon$.  Thus $\left\|q(T_0)\right\|_{\mathfrak{A}_1/\mathfrak{J}} < \epsilon$ as  $\pi_{1,0}$ is a faithful representation of $\mathfrak{A}_1/\mathfrak{J}$.  Hence there exists a $J \in \mathfrak{J}$ such that $\left\|T_{1,0} - J\right\| < \epsilon$.  Similarly $P_1TP_1(\mathcal{H}_2) = \{0\}$ as $T \in ker(\pi')$ and $T_{1,0}(\mathcal{H}_2) = \{0\}$ as $T_{1,0} \in \mathfrak{A}_1$.  Hence, as $\pi_2$ was a faithful representation, $\left\|T_{2,0}\right\| < \epsilon$.  Thus $J \in \mathfrak{J}_{(0)}$ and 
\[
\left\|P_1 T P_1 - P_1 J P_1\right\| \leq \left\| P_1TP_1 - P_1T_{1,0}P_1 - P_1T_{2,0}P_1\right\| + \left\|T_{2,0}\right\| + \left\|T_{1,0} - J\right\| < 3\epsilon
\]
as desired.
\par
Suppose the result is true for some $n \geq 1$.  Let $T \in \mathcal{E}(ker(\pi'))$ and let $\epsilon > 0$.  As $T \in \mathcal{E}(C^*(\mathfrak{A}_1 \oplus \mathfrak{A}_2, S))$ Lemma \ref{lemma2.2.9} implies that there exists an $m \in \mathbb{N}$, $T_{1,j} \in \mathfrak{A}_{1,(j)}$, and $T_{2,j} \in \mathfrak{A}_{1,(j)}$ such that $\left\|T - \sum^2_{i=1}\sum^m_{j=0} T_{i,j}\right\| < \epsilon$.  Therefore
\[
\left\|Q_nTQ_n - Q_n\left(\sum^2_{i=1}\sum^{n-1}_{j=0} T_{i,j}\right)Q_n\right\| = \left\|Q_n\left(T - \sum^2_{i=1}\sum^m_{j=0} T_{i,j}\right)Q_n\right\| < \epsilon
\]
as $\mathfrak{A}_{i,(j)} P_k = \{0\}$ for all $j \geq k$ and all $i \in \{1,2\}$.  By the inductive hypothesis there exists an $R \in span \left( \bigcup_{m < n} \mathfrak{J}_{(m)} \right)$ such that
\[
\left\|Q_nTQ_n - Q_nRQ_n\right\| < \epsilon
\]
and thus
\[
\left\|Q_nRQ_n - Q_n\left(\sum^2_{i=1}\sum^{n-1}_{j=0} T_{i,j}\right)Q_n\right\|  < 2\epsilon.
\]
However, as $R \in span \left( \bigcup_{m < n} \mathfrak{J}_{(m)} \right)$ and $\sum^2_{i=1}\sum^{n-1}_{j=0} T_{i,j} \in span \left( \bigcup_{m < n} (\mathfrak{A}_{1,(m)} \cup \mathfrak{A}_{2,(m)}) \right)$, 
\[
\left\|P_{n+1}RP_{n+1} - P_{n+1}\left(\sum^2_{i=1}\sum^{n-1}_{j=0} T_{i,j}\right)P_{n+1}\right\| = \left\|P_{n}RP_{n} - P_{n}\left(\sum^2_{i=1}\sum^{n-1}_{j=0} T_{i,j}\right)P_{n}\right\| < 2\epsilon
\]
by Discussion \ref{anjn}.  Therefore, by considering direct sums,
\[
\left\|Q_{n+1}RQ_{n+1} - Q_{n+1}\left(\sum^2_{i=1}\sum^{n-1}_{j=0} T_{i,j}\right)Q_{n+1}\right\|  < 2\epsilon.
\]
Hence
\[
\left\|Q_{n+1}TQ_{n+1} - Q_{n+1}RQ_{n+1}- Q_{n+1}T_{1,n}Q_{n+1} - Q_{n+1}T_{2,n}Q_{n+1}\right\| < 3\epsilon\,\,\,\,\,\,\,\,\,\,\,(*).
\]
Thus it suffices to approximate $T_{1,n} + T_{2,n} \in \mathfrak{A}_{1,(n)} + \mathfrak{A}_{2,(n)}$ uniformly on $Q_{n+1}\mathcal{K}$ with an element of $\mathfrak{J}_{(n)}$.  Since elements of $\mathfrak{A}_{1,(n)}$, $\mathfrak{A}_{2,(n)}$ and $\mathfrak{J}_{(n)}$ are zero when restricted to $Q_{n}\mathcal{K}$, it suffices to perform the approximation on $\mathcal{L}_{1,n+1} \oplus \mathcal{L}_{2,n+1}$.
\par
First we claim that $\left\|T_{2,n}\right\| < 3\epsilon$.  Fix $k \in \{1,2\}$ with $k = n \mod 2$ and let $\ell \in \{1,2\}\setminus \{k\}$.  By the description of the action of $T_{2,n}$ on $\mathcal{K}$ it suffices to show that $\left\|P_{k,n+1}T_{2,n}|_{\mathcal{L}_{k,n+1}}\right\| < 3\epsilon$.  However, by the description of the action of $T_{2,n}$ on $\mathcal{K}$ as given in Discussion \ref{anjn}, $P_{k,n+1}T_{2,n}|_{\mathcal{L}_{k,n+1}}$ is supported on
\[
\mathcal{H}_{i_1,0} \otimes \mathcal{H}_{i_2,0} \otimes \cdots \otimes \mathcal{H}_{i_n,0} \otimes \mathcal{H}_{i_{n+1}}
\]
(since $\mathfrak{A}_i\xi_i \subseteq \mathcal{H}_{i,0}$ for all $i \in \{1,2\}$) where $i_j = k$ if $j$ is odd and $i_j = \ell$ if $j$ is even (and automatically $i_{n+1} = 2$).  However $T$ and $R$ vanish on the above Hilbert space as they are elements of $ker(\pi')$ and $T_{1,n}$ vanishes on the above Hilbert space by Discussion \ref{anjn}.  Hence $\left\|T_{2,n}\right\| < 3\epsilon$ as claimed.
\par
Next we desire to approximate $T_{1,n}$ with an element of $\mathfrak{J}_{(n)}$.  To begin fix $k \in \{1,2\}$ with $k \neq n \mod 2$ and let $\ell \in \{1,2\}\setminus \{k\}$.  Hence $T_{1,n}$ is completely determined by its action on $\mathcal{L}_{k,n+1}$ and
\[
\left\|P_{k,n+1}TP_{k,n+1} - P_{k,n+1}RP_{k,n+1}- P_{k,n+1}T_{1,n}P_{k,n+1}\right\| < 3\epsilon\,\,\,\,\,\,\,\,\,\,\,(**)
\]
by $(*)$ and the fact that $P_{k,n+1}T_{2,n}P_{k,n+1} = 0$ by Discussion \ref{anjn}.
\par
Write 
\[
T_{1,n} = \sum^p_{q=1} A^{(q)}_n S A_{n-1}^{(q)} S \dots A^{(q)}_1 S A^{(q)} S^* B^{(q)}_1 S^* B^{(q)}_2\cdots S^*B^{(q)}_n
\]
where $A^{(q)}_i, B^{(q)}_j \in \mathfrak{A}_1$ if $i,j = 0 \mod 2$, $A^{(q)}_i, B^{(q)}_j \in \mathfrak{A}_2$ if $i,j = 1 \mod 2$, and $A^{(q)} \in \mathfrak{A}_1$ for all $q \in \{1,\ldots, p\}$.  By earlier discussion $T_{1,n}$ acts on $\mathcal{L}_{k,n+1}$ by
\begin{eqnarray} 
&&\!\!\!\! T_{1,n}(\eta_1 \otimes \cdots \otimes \eta_{n+1}) \nonumber\\
&=&\!\!\!\! \sum^p_{q=1}\left\langle \eta_1\otimes \cdots \otimes \eta_n, \left(B^{(q)}_n\right)^*\xi_k \otimes  \cdots \otimes \left(B^{(q)}_1\right)^*\xi_2 \right\rangle_{ \mathcal{L}_{k,n}} \left(A^{(q)}_n\xi_k \otimes  \cdots \otimes A^{(q)}_1\xi_2 \otimes A\eta_{n+1}\right). \nonumber
\end{eqnarray}
By viewing $T_{1,n} \in \mathcal{B}(\mathcal{L}_{k,n+1})$ and by applying the Gram-Schmidt Orthogonalization Process, we can write $T_{1,n}$ as the map
\[
T_{1,n} : \eta_1 \otimes \cdots \otimes \eta_{n+1} \mapsto \sum^\ell_{i=1}\sum^\ell_{j=1} \langle \eta_1 \otimes \cdots \otimes\eta_n, \zeta_j\rangle (\omega_i \otimes A_{i,j}\eta_{n+1})
\]
where $\{A_{i,j}\}^p_{i,j=1} \subseteq span(\{ A^{(q)} \, \mid \, q \in \{1,\ldots,p\}\}) \subseteq \mathfrak{A}_1$ and 
\[
\zeta_j, \omega_i \in \mathcal{N}_{i_1} \otimes \mathcal{N}_{i_2} \otimes \cdots \otimes \mathcal{N}_{i_n}
\]
(where for $i \in \{1,2\}$ $\mathcal{N}_i = \overline{\mathfrak{A}_i \xi_i}$, $i_j = k$ if $j$ is odd, and $i_j = \ell$ if $j$ is even (and automatically $i_{n} = 2$)) are such that $\{\zeta_j\}^p_{j=1}$ and $\{\omega_i\}^p_{i=1}$ are orthonormal sets.
\par
Suppose $\zeta \in \mathcal{L}_{k,n+1}$ is such that $\left\|\zeta\right\| \leq1$.  Then we can write $\zeta = \sum^p_{j=1} \zeta_j \otimes \eta_j + \sum_{\gamma \in \Gamma} \zeta_\gamma \otimes \eta_\gamma$ where $\{\zeta_\gamma\}_{\gamma \in \Gamma} \subseteq \mathcal{L}_{k,n}$ extends $\{\zeta_j\}^p_{j=1}$ to an orthonormal basis of $\mathcal{L}_{k,n}$ and $\eta_j, \eta_\gamma \in \mathcal{H}_1$.  Thus $\sum^p_{j=1} \left\|\eta_j\right\|^2 \leq \left\|\zeta\right\|^2 \leq 1$ and
\begin{eqnarray} 
\left\|T_{1,n}\zeta\right\|_{\mathcal{L}_{k,n+1}}\!\!\!\!&=&\!\!\!\!\left\|\sum^p_{j=1}\sum^p_{i=1} \left\|\zeta_j\right\|^2 \omega_i \otimes A_{i,j}\eta_j\right\|_{\mathcal{L}_{k,n+1}} \nonumber\\
&=&\!\!\!\! \left\|\sum^p_{i=1}\omega_i \otimes \left(\sum^p_{j=1}A_{i,j}\eta_j\right)\right\|_{\mathcal{L}_{k,n+1}} \nonumber\\
&=&\!\!\!\! \left(\sum^p_{i=1}\left\|\sum^p_{j=1}A_{i,j}\eta_j\right\|_{\mathcal{H}_1}^2\right)^\frac{1}{2}\,\,\,\,\,\,\,\,\,\,\,(***). \nonumber
\end{eqnarray}
This final expression is directly related to the norm of $[A_{i,j}] \in \mathcal{M}_p(\mathfrak{A}_1)$.  Indeed recall that $\mathfrak{A}_1$ is acting on $\mathcal{H}_1 = \mathcal{H}_{1,0} \oplus \mathcal{H}_{1,1}$ via $(\pi_{1,0}\circ q) \oplus \pi_{1,1}$ and define $\sigma_p : \mathcal{M}_p(\mathfrak{A}_1) \to \mathcal{B}(\mathcal{H}_1^{\oplus p})$ by
\[
\sigma_p([A'_{i,j}])(h_1 \oplus \cdots \oplus h_\ell) = \bigoplus^p_{i=1}\left(\sum^p_{j=1} A'_{i,j}h_j\right)
\]
for all $[A'_{i,j}] \in \mathcal{M}_p(\mathfrak{A}_1)$.  Clearly $\sigma_p$ is a faithful representation of $\mathcal{M}_p(\mathfrak{A}_1)$ since $(\pi_{1,0}\circ q) \oplus \pi_{1,1}$ is a faithful representation of $\mathfrak{A}_1$.  Notice $\sigma_p([A'_{i,j}])$ is zero on $\mathcal{H}_{1,0}^{\oplus p} \subseteq \mathcal{H}_1^{\oplus p}$ if and only if each $A'_{i,j}$ is zero on $\mathcal{H}_{1,0}$ if and only if $[A'_{i,j}] \in \mathcal{M}_p(\mathfrak{J})$.  Since $\mathcal{M}_p(\mathfrak{J})$ is an ideal of $\mathcal{M}_p(\mathfrak{A}_1)$, $\mathcal{H}_{1,0}^{\oplus p}$ is a reducing subspace for $\sigma_p(\mathcal{M}_p(\mathfrak{A}_1))$, and $\mathcal{M}_p(\mathfrak{J}) = ker\left(\sigma_p|_{\mathcal{H}_{1,0}^{\oplus p}}\right)$, we obtain that $\sigma_p|_{\mathcal{H}_{1,0}^{\oplus p}}$ is a faithful representation of $\mathcal{M}_p(\mathfrak{A}_1)/\mathcal{M}_p(\mathfrak{J}) \simeq \mathcal{M}_p(\mathfrak{A}_1/\mathfrak{J})$.
\par
Using $(**)$ and the fact that $P_{k,n+1}TP_{k,n+1} - P_{k,n+1}RP_{k,n+1}$ is zero on 
\[
\mathcal{H}_{i_{1},0} \otimes \mathcal{H}_{i_{2},0} \otimes \cdots \otimes \mathcal{H}_{i_{n},0} \otimes \mathcal{H}_{i_{n+1},0}
\]
where $i_j = k$ if $j$ is odd and $i_j = \ell$ if $j$ is even (and automatically $i_{n+1} = 1$) (as $T \in \mathcal{E}(ker(\pi'))$ and $R \in span \left( \bigcup_{m < n} \mathfrak{J}_{(m)} \right)$), 
\[
\left\|P_{n+1}T_{1,n}P_{n+1}|_{\mathcal{H}_{i_{1},0} \otimes \mathcal{H}_{i_{2},0} \otimes \cdots \otimes \mathcal{H}_{i_{n},0} \otimes \mathcal{H}_{i_{n+1},0}}\right\| < 3\epsilon.
\]
 As $\omega_i, \zeta_j \in \mathcal{H}_{i_{1},0} \otimes \mathcal{H}_{i_{2},0} \otimes \cdots \otimes \mathcal{H}_{i_{n},0}$ for all $i,j \in \{1,\ldots, p\}$, using $(***)$ we see that
\[
\left(\sum^p_{i=1}\left\|\sum^p_{j=1}A_{i,j}\eta_j\right\|_{\mathcal{H}_{1,0}}^2\right)^\frac{1}{2} < 3\epsilon
\]
for all $\eta_1, \ldots, \eta_p \in \mathcal{H}_{1,0}$ with $\sum^p_{j=1} \left\|\eta_j\right\|^2 \leq 1$.  Whence $\left\|\sigma_p([A_{i,j}])|_{\mathcal{H}_{1,0}^{\oplus p}}\right\| < 3\epsilon$.  Since $\sigma_p|_{\mathcal{H}_{1,0}^{\oplus p}}$ is a faithful representation of $\mathcal{M}_p(\mathfrak{A}_1)/\mathcal{M}_p(\mathfrak{J}) \simeq \mathcal{M}_p(\mathfrak{A}_1/\mathfrak{J})$, there exists a $[J_{i,j}] \in \mathcal{M}_{p}(\mathfrak{J})$ such that $\left\|[A_{i,j}]-[J_{i,j}]\right\|_{\mathcal{M}_p(\mathfrak{A}_1)} < 3\epsilon$.  Thus
\[
\left(\sum^p_{i=1}\left\|\sum^p_{j=1}(A_{i,j}- J_{i,j})\eta_j\right\|_{\mathcal{H}_1}^2\right)^\frac{1}{2} < 3\epsilon
\]
for all $\eta_1, \ldots, \eta_\ell \in \mathcal{H}_1$ with $\sum^p_{j=1} \left\|\eta_j\right\|^2 \leq 1$. 
\par
Define $R' \in \mathcal{B}(\mathcal{L}_{k,n+1})$ by
\[
R': \eta_1 \otimes \cdots \otimes \eta_{n+1} \mapsto \sum^p_{i=1}\sum^p_{j=1} \langle \eta_1 \otimes \cdots \otimes\eta_n, \zeta_j\rangle (\omega_i \otimes J_{i,j}\eta_{n+1})
\]
and extend by linearity and density.  Note that repeating $(***)$ twice shows $R'$ is indeed a bounded linear map and
\[
\left\|R' - P_{k,n+1}T_{1,n}P_{k,n+1}\right\|_{\mathcal{B}(\mathcal{L}_{k,n+1})} < 3\epsilon.
\]
As 
\[
\zeta_j, \omega_i \in \mathcal{N}_{i_1} \otimes \mathcal{N}_{i_2} \otimes \cdots \otimes \mathcal{N}_{i_n},
\]
Discussion \ref{anjn} implies that there exists a $R_0 \in \mathfrak{J}_{(n)}$ such that
\[
\left\|R' - P_{k,n+1}R_0P_{k,n+1}\right\|_{\mathcal{B}(\mathcal{L}_{k,n+1})} < \epsilon.
\]
Whence
\[
\left\|P_{k,n+1}R_0P_{k,n+1} - P_{k,n+1}T_{1,n}P_{k,n+1}\right\| < 4\epsilon.
\]
Since $R_0 \in \mathfrak{J}_{(n)}$ and $T_{1,n} \in \mathfrak{A}_{1,(n)}$, 
\[
\left\|Q_{n+1}R_0Q_{n+1} - Q_{n+1}T_{1,n}Q_{n+1}\right\| = \left\|P_{k,n+1}R_0P_{k,n+1} - P_{k,n+1}T_{1,n}P_{k,n+1}\right\| < 4\epsilon.
\]
By combining all of our approximations
\[
\left\|Q_{n+1}TQ_{n+1} - Q_{n+1}(R + R_0)Q_{n+1}\right\| < 10\epsilon
\]
and, as $R + R_0 \in span\left( \bigcup_{m < n+1} \mathfrak{J}_{(m)}\right)$, the result follows. 
\end{proof}
The above result shows we can approximate elements of $\mathcal{E}(ker(\pi'))$ uniformly on $Q_n \mathcal{K}$ by elements of $span\left(\bigcup_{m < n} \mathfrak{J}_{(m)}\right)$.  The following result shows that this is enough to prove the assumptions of Lemma \ref{lemma2.2.8}.
\begin{lem}
\label{tfinallemma}
Let $T \in \mathcal{E}(C^*(\mathfrak{A}_1 \oplus \mathfrak{A}_2, S))$ and let $\epsilon > 0$.  There exists an $n \in 2\mathbb{N}$ such that
\[
\left\|P_{j,m}TP_{j,m} - P_{j,n}TP_{j,n} \otimes I_{\mathcal{L}_{j,m-n}}\right\|_{\mathcal{B}(\mathcal{L}_{j,m})} < \epsilon
\]
for all $m \geq n$ and $j \in \{1,2\}$ (where $\mathcal{L}_{j,m} \simeq \mathcal{L}_{j,n} \otimes \mathcal{L}_{j,m-n}$ canonically as $n$ is even).
\end{lem}
\begin{proof}
Fix $\epsilon > 0$.  By Lemma \ref{lemma2.2.9} there exists an $n \in 2\mathbb{N}$ and a $R\in span\left( \bigcup_{m < 2n} (\mathfrak{A}_{1,(m)} \cup \mathfrak{A}_{2,(m)})\right)$ such that $\left\|T-R\right\| < \epsilon$.  Fix $m \geq n$.  Then for all $j \in \{1,2\}$ $R$ acts on $\mathcal{L}_{j,m}$ as $P_{j,n}RP_{j,n} \otimes I_{\mathcal{L}_{j,m-n}}$ and thus 
\[
\left\|P_{j,m}TP_{j,m} - P_{j,n}RP_{j,n} \otimes I_{\mathcal{L}_{j,m-n}}\right\|_{\mathcal{B}(\mathcal{L}_{j,m})} \leq \left\|T - R\right\| < \epsilon.
\]
However $\left\|P_{j,n}TP_{j,n} - P_{j,n}RP_{j,n}\right\| \leq \left\|T - R\right\| < \epsilon$ for all $j \in \{1,2\}$ so
\[
\left\|P_{j,m}TP_{j,m} - P_{j,n}TP_{j,n} \otimes I_{\mathcal{L}_{j,m-n}}\right\|_{\mathcal{B}(\mathcal{L}_{j,m})} < 2\epsilon
\]
as desired.  
\end{proof}
\begin{proof}[Proof of Theorem \ref{dse}] Recall $\mathcal{E}\left(\langle \mathfrak{J}\rangle_{C^*(\mathfrak{A}_1 \oplus \mathfrak{A}_2, S)}\right) \subseteq \mathcal{E}(ker(\pi'))$ by Discussion \ref{disc2.2.7}.  Let $T \in \mathcal{E}(ker(\pi'))$ and let $\epsilon > 0$.  By Lemma \ref{tfinallemma} there exists an $n \in 2\mathbb{N}$ so that 
\[
\left\|P_{j,m}TP_{j,m} - P_{j,n}TP_{j,n} \otimes I_{\mathcal{L}_{j,m-n}}\right\|_{\mathcal{B}(\mathcal{L}_{j,m})} < \epsilon
\]
for all $m \geq n$ and $j \in \{1,2\}$.  By Lemma \ref{complemma} there exists an $R \in span \left( \bigcup_{m < n} \mathfrak{J}_{(m)} \right)$ so that
\[
\left\|Q_n (T - R)Q_n\right\| < \epsilon.
\]
As $T, R \in \mathcal{E}(C^*(\mathfrak{A}_1 \oplus \mathfrak{A}_2, S))$
\[
\left\|T-R\right\| = \sup_{m \geq 1} \max_{j \in \{1,2\}} \left\|P_{j,m}(T-R)P_{j,m}\right\|
\]
by Lemma \ref{lemma2.2.4} and the above inequality implies 
\[
\left\|P_{j,m}(T-R)P_{j,m}\right\|_{\mathcal{L}_{j,m}} < \epsilon
\]
for all $m \leq n$ and $j \in \{1,2\}$.  Thus $\left\|P_{j,n}TP_{j,n} - P_{j,n}RP_{j,n}\right\| < \epsilon$ for all $j \in \{1,2\}$.  However, since $R \in span \left( \bigcup_{m < n} \mathfrak{J}_{(m)} \right)$, $P_{j,m}RP_{j,m} = P_{j,n}RP_{j,n} \otimes I_{\mathcal{L}_{j,m-n}}$ for all $m \geq n$ and $j \in \{1,2\}$ (as $n$ is even) and thus
\[
\left\|P_{j,m}TP_{j,m} - P_{j,m}RP_{j,m}\right\| \leq \epsilon + \left\|P_{j,n}TP_{j,n} \otimes I_{\mathcal{L}_{j,m-n}} - P_{j,n}RP_{j,n} \otimes I_{\mathcal{L}_{j,m-n}}\right\|_{\mathcal{L}_{j,m}} \leq 2\epsilon
\]
for all $m \geq n$ and $j \in \{1,2\}$.  Whence $\left\|T-R\right\| \leq 2\epsilon$.  As $R \in \mathcal{E}\left(\langle \mathfrak{J}\rangle_{C^*(\mathfrak{A}_1 \oplus \mathfrak{A}_2, S)}\right) \subseteq \langle \mathfrak{J}\rangle_{C^*(\mathfrak{A}_1 \oplus \mathfrak{A}_2, S)}$, we obtain that $T \in \langle \mathfrak{J}\rangle_{C^*(\mathfrak{A}_1 \oplus \mathfrak{A}_2, S)}$ and thus $T \in \mathcal{E}\left(\langle \mathfrak{J}\rangle_{C^*(\mathfrak{A}_1 \oplus \mathfrak{A}_2, S)}\right)$.  Hence $\mathcal{E}(ker(\pi')) = \mathcal{E}\left(\langle \mathfrak{J}\rangle_{C^*(\mathfrak{A}_1 \oplus \mathfrak{A}_2, S)}\right)$.  Therefore $ker(\pi') = \langle \mathfrak{J}\rangle_{C^*(\mathfrak{A}_1 \oplus \mathfrak{A}_2, S)}$ by Lemma \ref{lemma2.2.8} as desired.  
\end{proof}

\subsection{Proof of Theorem \ref{main}}

In this section we will complete the proof of Theorem \ref{main}.  By Theorem \ref{dse} we know certain short sequences of C$^*$-algebras are exact and we will use the proof of Theorem 4.8.2 in \cite{BO} to construct a commutative diagram of short sequences.  The proof of Theorem 4.8.2 in \cite{BO} is concrete and allows us to demonstrate that the compression of $\langle \mathfrak{J}\rangle_{C^*(\mathfrak{A}_1 \oplus \mathfrak{A}_2, S)}$ corresponds with the description of $\langle \mathfrak{J}\rangle_{\mathfrak{A}_1 \ast \mathfrak{A}_2}$ given in Section 2.2. The remainder of the proof is then trivial.  We start the proof by re-describing the context.
\begin{disc}
Let $\mathfrak{A}_1$ and $\mathfrak{A}_2$ be unital C$^*$-algebras, let $\mathfrak{J}$ be an ideal of $\mathfrak{A}_1$, let $\pi_{1,0} : \mathfrak{A}_1/\mathfrak{J} \to \mathcal{B}(\mathcal{H}_{1,0})$, $\pi_{1,1} : \mathfrak{A}_1 \to \mathcal{B}(\mathcal{H}_{1,1})$, and $\pi_2 : \mathfrak{A}_2 \to \mathcal{B}(\mathcal{H}_2)$ be unital representations such that $\pi_{1,0}$ and $\pi_2$ are faithful and, if $\mathcal{H}_1 := \mathcal{H}_{1,0} \oplus \mathcal{H}_{1,1}$ and $q : \mathfrak{A}_1 \to \mathfrak{A}_1/\mathfrak{J}$ is the canonical quotient map, $\pi_1 := (\pi_{1,0} \circ q) \oplus \pi_{1,1} : \mathfrak{A}_1 \to \mathcal{B}(\mathcal{H}_1)$ is faithful, and let $\xi_1 \in \mathcal{H}_{1,0}$ and $\xi_2 \in \mathcal{H}_2$ be unit vectors.  For notational purposes let $\mathcal{H}_{2,0} := \mathcal{H}_2$ and let $\pi_{2,0} := \pi_2 : \mathfrak{A}_2 \to \mathcal{B}(\mathcal{H}_{2,0})$.
\par
Using the notation of Discussion \ref{directsumconstruction}, the sequence
\[
0\rightarrow\langle \mathfrak{J}\rangle_{C^*(\mathfrak{A}_1 \oplus \mathfrak{A}_2, S)}\rightarrow C^*(\mathfrak{A}_1 \oplus \mathfrak{A}_2, S)  \stackrel{\pi'}{\rightarrow} C^*((\mathfrak{A}_1/\mathfrak{J}) \oplus \mathfrak{A}_2, S_0) \rightarrow 0
\]
is exact by Theorem \ref{dse}.  In order to show that the sequence
\[
0 \rightarrow \langle \mathfrak{J}\rangle_{\mathfrak{A}_1 \ast \mathfrak{A}_2} \stackrel{i}{\rightarrow} (\mathfrak{A}_1, \pi_1, \xi_1) \ast (\mathfrak{A}_2, \pi_2, \xi_2) \stackrel{\pi}{\rightarrow}(\mathfrak{A}_1/\mathfrak{J}, \pi_{1,0}, \xi_1) \ast (\mathfrak{A}_2, \pi_2, \xi_2) \rightarrow 0
\]
as described in Construction \ref{primecons} is exact, we desire to embed the later sequence into the first.  
\end{disc}
\begin{nota}
\label{mainproofnotation}
Let $\mathfrak{A}_{1,0} = \mathfrak{A}_1/\mathfrak{J}$ and let $\mathfrak{A}_{2,0} = \mathfrak{A}_2$.  Using the notation of Discussion \ref{directsumconstruction}, let 
\[
P = I - S^2(S^*)^2 \in C^*(\mathfrak{A}_1 \oplus \mathfrak{A}_2, S)
\]
and let 
\[
U = P(S+S^*)P \in C^*(\mathfrak{A}_1 \oplus \mathfrak{A}_2, S).
\]
Define the unital, completely positive maps $\psi_i : \mathfrak{A}_i \to PC^*(\mathfrak{A}_1 \oplus \mathfrak{A}_2, S)P$ by 
\[
\psi_i(A) = PAP + UAU
\]
for all $A \in \mathfrak{A}_i$.  Similarly let 
\[
P_0 = I - S_0^2(S_0^*)^2 \in C^*(\mathfrak{A}_{1,0} \oplus \mathfrak{A}_{2,0}, S_0)
\]
and let 
\[
U_0 = P_0(S_0+S_0^*)P_0 \in C^*(\mathfrak{A}_{1,0} \oplus \mathfrak{A}_{2,0}, S_0).
\]
Define the unital, completely positive maps $\psi_{i,0} : \mathfrak{A}_{i,0} \to P_0C^*(\mathfrak{A}_{1,0} \oplus \mathfrak{A}_{2,0}, S_0)P_0$ by 
\[
\psi_{i,0}(A) = P_0AP_0 + U_0AU_0
\]
for all $A \in \mathfrak{A}_{i,0}$.
\par
For all $i \in \{1,2\}$ let 
\[
\mathfrak{A}_i^0 := \{ A \in \mathfrak{A}_i \, \mid \, \langle A\xi_i, \xi_i\rangle_{\mathcal{H}_i} = 0\}\,\,\,\,\,\,\,\,\,\,\mbox{ and let }\,\,\,\,\,\,\,\,\,\,\mathfrak{A}^0_{i,0} := \{ A \in \mathfrak{A}_{i,0} \, \mid \, \langle A\xi_i, \xi_i\rangle_{\mathcal{H}_{i,0}} = 0\}.
\]
 Thus $\mathfrak{A}_i = \mathbb{C}I_{\mathfrak{A}_i} + \mathfrak{A}^0_i$ and $\mathfrak{A}_{i,0} = \mathbb{C}I_{\mathfrak{A}_{i,0}} + \mathfrak{A}^0_{i,0}$ for all $i \in \{1,2\}$.
\end{nota}
\begin{lem}
\label{BOannoy}
There exists a unital, completely positive map 
\[
\Psi : (\mathfrak{A}_1, \pi_1, \xi_1) \ast (\mathfrak{A}_2, \pi_2, \xi_2) \to PC^*(\mathfrak{A}_1 \oplus \mathfrak{A}_2, S)P
\]
such that
\[
\Psi(A_1 \cdots A_n) = \psi_{i_1}(A_1) \cdots \psi_{i_n}(A_n)
\]
whenever $A_k \in \mathfrak{A}^0_{i_k}$, $\{i_k\}^n_{k=1} \subseteq \{1,2\}$, and $i_k \neq i_{k+1}$ for all $k \in \{1,\ldots, n-1\}$.  Moreover there exists a $^*$-homomorphism 
\[
\sigma : C^*(\Psi((\mathfrak{A}_1, \pi_1, \xi_1) \ast (\mathfrak{A}_2, \pi_2, \xi_2))) \to (\mathfrak{A}_1, \pi_1, \xi_1) \ast (\mathfrak{A}_2, \pi_2, \xi_2)
\]
such that $\sigma \circ \Psi = Id_{(\mathfrak{A}_1, \pi_1, \xi_1) \ast (\mathfrak{A}_2, \pi_2, \xi_2)}$.  In fact $\sigma$ is the compression map of $\mathcal{B}(\mathcal{K})$ to $\mathcal{B}(\mathcal{K}_{1,1})$ where $\mathcal{K}_{1,1} \subseteq \mathcal{K}$ is a Hilbert space isomorphic to $(\mathcal{H}_1, \xi_1) \ast (\mathcal{H}_2, \xi_2)$.
\par
Similarly there exists a unital, completely positive map 
\[
\Psi_0 : (\mathfrak{A}_{1,0}, \pi_{1,0}, \xi_1) \ast (\mathfrak{A}_{2,0}, \pi_{2,0}, \xi_2) \to P_0C^*(\mathfrak{A}_{1,0} \oplus \mathfrak{A}_{2,0}, S_0)P_0
\]
such that
\[
\Psi_0(A_1 \cdots A_n) = \psi_{i_1,0}(A_1) \cdots \psi_{i_n,0}(A_n)
\]
whenever $A_k \in \mathfrak{A}^0_{i_k,0}$, $\{i_k\}^n_{k=1} \subseteq \{1,2\}$, and $i_k \neq i_{k+1}$ for all $k \in \{1,\ldots, n-1\}$. 
\end{lem}
\begin{proof}
The proof of the above result is contained in Theorem 4.8.2 of \cite{BO}.  Note that the proof in \cite{BO} is done under the assumptions that $\pi_1$, $\pi_2$, and $\pi_{1,0}$ are the faithful representations corresponding to a GNS construction.  However these assumptions are not used in the proof.
\par
For later purposes we remark that the Hilbert subspace $\mathcal{K}_{1,1}$ of $\mathcal{K}$ is the subspace
\[
\mathcal{H}_1 \oplus \left( \bigoplus_{n\geq 0} \mathcal{H}_1 \otimes (\mathcal{H}_2^0 \otimes \mathcal{H}_{1}^0)^{\otimes n} \otimes \mathcal{H}_2^0 \otimes \mathcal{H}_1  \right)
\]
and is isomorphic to $(\mathcal{H}_1, \xi_1) \ast (\mathcal{H}_2, \xi_2)$ via the standard identifications $\mathbb{C}\xi_1 \otimes \mathcal{H}_2^0 \simeq \mathcal{H}_2^0$ and $\mathcal{H}_2^0 \otimes \mathbb{C}\xi_1 \simeq \mathcal{H}_2^0$.
\end{proof}
\begin{lem}
\label{comm}
With $\Psi$ and $\Psi_0$ as in Lemma \ref{BOannoy}, the diagram
\[
\begin{array}{ccc}
(\mathfrak{A}_1, \pi_1, \xi_1) \ast (\mathfrak{A}_2, \pi_2, \xi_2) &\stackrel{\pi}{\rightarrow}&(\mathfrak{A}_1/\mathfrak{J}, \pi_{1,0}, \xi_1) \ast (\mathfrak{A}_2, \pi_2, \xi_2) \\
\downarrow \Psi & & \downarrow \Psi_0 \\
C^*(\mathfrak{A}_1 \oplus \mathfrak{A}_2, S)  &\stackrel{\pi'}{\rightarrow} &C^*((\mathfrak{A}_1/\mathfrak{J}) \oplus \mathfrak{A}_2, S_0)
\end{array} 
\]
commutes.
\end{lem}
\begin{proof}
Recall that the span of $I_{(\mathfrak{A}_1, \pi_1, \xi_1) \ast (\mathfrak{A}_2, \pi_2, \xi_2)}$ and
\[
\{A_1A_2 \cdots A_n \, \mid \, A_k \in \mathfrak{A}_{i_k}^0, \{i_k\}^n_{k=1} \subseteq \{1,2\}, i_k \neq i_{k+1} \mbox{ for all }k \in \{1,\ldots, n-1\}\}\,\,\,\,\,\,\,\,\,\,(*)
\]
is dense in $(\mathfrak{A}_1, \pi_1, \xi_1) \ast (\mathfrak{A}_2, \pi_2, \xi_2)$.  Therefore, to prove the diagram commutes, it suffices by linearity and density to verify that $\Psi_0\left(\pi\left(I_{(\mathfrak{A}_1, \pi_1, \xi_1) \ast (\mathfrak{A}_2, \pi_2, \xi_2)}\right)\right) = \pi'\left(\Psi\left(I_{(\mathfrak{A}_1, \pi_1, \xi_1) \ast (\mathfrak{A}_2, \pi_2, \xi_2)}\right)\right)$ and $\Psi_0(\pi(T)) = \pi'(\Psi_0(T))$ for all $T$ in the set given in $(*)$.  However $\Psi_0\left(\pi\left(I_{(\mathfrak{A}_1, \pi_1, \xi_1) \ast (\mathfrak{A}_2, \pi_2, \xi_2)}\right)\right) = \pi'\left(\Psi\left(I_{(\mathfrak{A}_1, \pi_1, \xi_1) \ast (\mathfrak{A}_2, \pi_2, \xi_2)}\right)\right)$ is trivial as $\Psi$, $\Psi_0$, $\pi$, and $\pi'$ are unital.  If $A_k \in \mathfrak{A}_{i_k}^0$ for $k \in \{1,\ldots, n\}$, $\{i_k\}^n_{k=1} \subseteq \{1,2\}$, and $i_k \neq i_{k+1}$ for all $k \in \{1,\ldots, n-1\}$ then
\[
\Psi_0(\pi(A_1A_2 \cdots A_n))=\Psi_0(\pi(A_1)\pi(A_2) \cdots \pi(A_n)) = \psi_{i_1,0}(\pi(A_1))\psi_{i_2,0}(\pi(A_2)) \cdots \psi_{i_n,0}(\pi(A_n)) 
\]
since $\pi$ is a $^*$-homomorphism, $\pi(A_k)$ is an element of $\mathfrak{A}^0_{i_k, 0}$ for all $k \in \{1,\ldots, n\}$, and by the properties of $\Psi_0$ from Lemma \ref{BOannoy}, and
\[ 
\pi'(\Psi(A_1A_2 \cdots A_n)) = \pi'(\psi_{i_1}(A_1)\psi_{i_2}(A_2) \cdots \psi_{i_n}(A_n))= \pi'(\psi_{i_1}(A_1))\pi'(\psi_{i_2}(A_2)) \cdots \pi'(\psi_{i_n}(A_n)) 
\]
by the properties of $\Psi$ from Lemma \ref{BOannoy} and since $\pi'$ is a $^*$-homomorphism.  However, for all $i \in \{1,2\}$ and $A \in \mathfrak{A}_i$,
\[
\psi_{i,0}(\pi(A)) = P_0 \pi(A) P_0 + U_0 \pi(A) U_0 = \pi'(PAP + UAU) = \pi'(\psi_i(A)).
\]
Hence $\Psi_0 \circ \pi = \pi' \circ \Psi$ on a set with dense span in $(\mathfrak{A}_1, \pi_1, \xi_1) \ast (\mathfrak{A}_2, \pi_2, \xi_2)$ and thus the result follows by linearity and density.
\end{proof}
The final technical challenge of the proof of Theorem \ref{main} is the following.
\begin{lem}
\label{compideal}
Let $\sigma : \mathcal{B}(\mathcal{K}) \to \mathcal{B}((\mathcal{H}_1, \xi_1) \ast (\mathcal{H}_2, \xi_2))$ be the compression map from Lemma \ref{BOannoy}.  If $T \in \langle \mathfrak{J}\rangle_{C^*(\mathfrak{A}_1 \oplus \mathfrak{A}_2, S)}$ then $\sigma(T) \in \langle \mathfrak{J} \rangle_{\mathfrak{A}_1 \ast \mathfrak{A}_2}$.
\end{lem}
\begin{proof}
By Discussion \ref{directideal} and the notation in \ref{mainproofnotation}, it is easy to see that the span of 
\[
\left\{\left.
\begin{array}{l}
(A_1S)(B_{1}S) \cdots (A_nS)(B_nS)J(S^*B'_m)(S^*A'_m)\cdots (S^*B'_1)(S^*A'_1) \\
(B_{1}S)(A_{2}S) \cdots (A_nS)(B_nS)J(S^*B'_m)(S^*A'_m)\cdots (S^*B'_1)(S^*A'_1) \\ 
(A_1S)(B_{1}S) \cdots (A_nS)(B_nS)J(S^*B'_m)(S^*A'_m)\cdots (S^*A'_2)(S^*B'_1) \\
(B_{1}S)(A_{2}S) \cdots (A_nS)(B_nS)J(S^*B'_m)(S^*A'_m)\cdots (S^*A'_2)(S^*B'_1)
\end{array}
\, \right| \,
\begin{array}{l}
n,m\geq 0, J \in \mathfrak{J}, \\
A_i, A'_j \in \mathfrak{A}^0_1 \cup\{I_{\mathfrak{A}_1}\}, \\
B_i, B'_j \in \mathfrak{A}^0_2 \cup \{I_{\mathfrak{A}_2}\}
\end{array} 
\right\}\,\,\,\,\,\,\,\,\,(*)
\]
is dense in $\langle \mathfrak{J}\rangle_{C^*(\mathfrak{A}_1 \oplus \mathfrak{A}_2, S)}$.  Recall that $\sigma$ was the compression of $\mathcal{K}$ onto
\[
\mathcal{K}_{1,1} = \mathcal{H}_1 \oplus \left( \bigoplus_{n\geq 0} \mathcal{H}_1 \otimes (\mathcal{H}_2^0 \otimes \mathcal{H}_{1}^0)^{\otimes n} \otimes \mathcal{H}_2^0 \otimes \mathcal{H}_1  \right) \subseteq \mathcal{K}
\]
and $\mathcal{K}_{1,1}$ is isomorphic to $(\mathcal{H}_1, \xi_1) \ast (\mathcal{H}_2, \xi_2)$ via the standard identifications $\mathbb{C}\xi_1 \otimes \mathcal{H}_2^0 \simeq \mathcal{H}_2^0$ and $\mathcal{H}_2^0 \otimes \mathbb{C}\xi_1 \simeq \mathcal{H}_2^0$.  Therefore it suffices to show that the compression of each element in $(*)$ to $\mathcal{K}$ corresponds to an element in $\langle \mathfrak{J}\rangle_{\mathfrak{A}_1 \ast \mathfrak{A}_2}$ as described in Section 2.2.
\par
Notice $B(\mathcal{K}_{1,1}) = \{0\}$ for all $B \in \mathfrak{A}_2$.  Therefore it suffices to consider only the set
\[
\left\{
\begin{array}{l}
(A_1S)(B_{1}S) \cdots (A_nS)(B_nS)J(S^*B'_m)(S^*A'_m)\cdots (S^*B'_1)(S^*A'_1) \\
(B_{1}S)(A_{2}S) \cdots (A_nS)(B_nS)J(S^*B'_m)(S^*A'_m)\cdots (S^*B'_1)(S^*A'_1)\end{array}
\, \left| \,
\begin{array}{l}
n,m\geq 0, J \in \mathfrak{J}, \\
A_i, A'_j \in \mathfrak{A}^0_1 \cup\{I_{\mathfrak{A}_1}\}, \\
B_i, B'_j \in \mathfrak{A}^0_2 \cup \{I_{\mathfrak{A}_2}\}
\end{array} 
\right.
\right\}\,\,\,\,\,\,\,\,\,(**).
\]
Fix $n\geq 0$, $m\geq 1$, $J \in \mathfrak{J}$, $\{A_i\}^n_{i=1}, \{A'_j\}^m_{j=1} \subseteq \mathfrak{A}^0_1 \cup\{I_{\mathfrak{A}_1}\}$, and $\{B_i\}^n_{i=1}, \{B'_j\}^m_{j=1} \subseteq \mathfrak{A}^0_2 \cup\{I_{\mathfrak{A}_2}\}$ and let 
\[
T = (A_1S)(B_{1}S) \cdots (A_nS)(B_nS)J(S^*B'_m)(S^*A'_m)\cdots (S^*B'_1)(S^*A'_1).
\]
If 
\[
\eta_1 \otimes \zeta_1 \otimes \cdots \otimes \eta_m \otimes \zeta_m \otimes \eta \in \mathcal{H}_1 \otimes (\mathcal{H}_2^0 \otimes \mathcal{H}_{1}^0)^{\otimes m-1} \otimes \mathcal{H}_2^0 \otimes \mathcal{H}_1 
\]
then
\begin{eqnarray} 
\!\!\!\!&&\!\!\!\! T(\eta_1 \otimes \zeta_1 \otimes \cdots \otimes \eta_m \otimes \zeta_m \otimes \eta)\nonumber\\
&=&\!\!\!\! \langle A'_1 \eta_1, \xi_1\rangle_{\mathcal{H}_1}  (A_1S)(B_{1}S) \cdots (A_nS)(B_nS)J(S^*B'_m)(S^*A'_m)\cdots (S^*B'_1)(\zeta_1 \otimes \cdots \otimes \eta_m \otimes \zeta_m \otimes \eta) \nonumber\\
&\vdots&\!\!\!\!  \nonumber\\
&=&\!\!\!\! \left(\prod_{k=1}^m \langle A'_k \eta_k, \xi_1\rangle_{\mathcal{H}_1}\right)\left(\prod_{k=1}^m \langle B'_k \zeta_k, \xi_2\rangle_{\mathcal{H}_2}\right) (A_1S)(B_{1}S) \cdots (A_nS)(B_nS)J \eta \nonumber\\
&=&\!\!\!\! \left(\prod_{k=1}^m \langle  \eta_k, (A'_k)^*\xi_1\rangle_{\mathcal{H}_1}\right)\left(\prod_{k=1}^m \langle \zeta_k, ( B'_k )^*\xi_2\rangle_{\mathcal{H}_2}\right) (A_1S)(B_{1}S) \cdots (A_nS)(B_n\xi_2 \otimes J \eta)  \nonumber\\
&\vdots&\!\!\!\!  \nonumber\\
&=&\!\!\!\! \left(\prod_{k=1}^m \langle  \eta_k, (A'_k)^*\xi_1\rangle_{\mathcal{H}_1}\right)\left(\prod_{k=1}^m \langle \zeta_k, ( B'_k )^*\xi_2\rangle_{\mathcal{H}_2}\right) (A_1\xi_1) \otimes (B_1\xi_2) \otimes \cdots \otimes (A_n\xi_1) \otimes (B_n\xi_2) \otimes J \eta.  \nonumber
\end{eqnarray}
It is easy to see that $T$ is zero on $\mathcal{H}_1 \otimes (\mathcal{H}_2^0 \otimes \mathcal{H}_{1}^0)^{\otimes k} \otimes \mathcal{H}_2^0 \otimes \mathcal{H}_1 $ for all $k \leq m-1$ and if $k \geq m-1$ and
\[
\eta_1 \otimes \zeta_1 \otimes \cdots \otimes \eta_m \otimes \zeta_m \otimes \eta_{m+1} \otimes \zeta_{m+1} \otimes \cdots \otimes \eta_{k+2} \in \mathcal{H}_1 \otimes (\mathcal{H}_2^0 \otimes \mathcal{H}_{1}^0)^{\otimes k} \otimes \mathcal{H}_2^0 \otimes \mathcal{H}_1 
\]
then
\[
T(\eta_1 \otimes \zeta_1 \otimes \cdots \otimes \eta_m \otimes \zeta_m \otimes \eta_{m+1} \otimes \zeta_{m+1} \otimes \cdots \otimes \eta_{k+2}) = T(\eta_1 \otimes \zeta_1 \otimes \cdots \otimes \eta_m \otimes \zeta_m \otimes \eta_{m+1}) \otimes \zeta_{m+1} \otimes \cdots \otimes \eta_{k+2}.
\]
\par
If $B_j = I_{\mathfrak{A}_2}$ for some $j \in \{1,\ldots, n\}$ then $B_j\xi_2$ is orthogonal to $\mathcal{H}_2^0$ and thus the image of $T|_{\mathcal{K}_{1,1}}$ is orthogonal to $\mathcal{K}_{1,1}$ so $\sigma(T) = 0$.  Similarly if $B'_j = I_{\mathfrak{A}_2}$ for some $j \in \{1,\ldots, m\}$ then $\langle \zeta_j, (B'_j)^*\xi_2\rangle_{\mathcal{H}_2} = 0$ for all $\zeta_j \in \mathcal{H}^0_2$ so $T$ is zero on $\mathcal{K}_{1,1}$ and thus $\sigma(T) = 0$.  Therefore we may assume that $\{B_i\}^n_{i=1}, \{B'_j\}^m_{j=1} \subseteq \mathfrak{A}^0_2$.
\par
If $A_j = I_{\mathfrak{A}_1}$ for some $j \in \{2,\ldots, n\}$ then $A_j\xi_1$ is orthogonal to $\mathcal{H}_1^0$ and thus the image of $T|_{\mathcal{K}_{1,1}}$ is orthogonal to $\mathcal{K}_{1,1}$ so $\sigma(T) = 0$.  Similarly if $A'_j = I_{\mathfrak{A}_1}$ for some $j \in \{2,\ldots, m\}$ then $\langle \eta_j, (A'_j)^*\xi_1\rangle_{\mathcal{H}_1} = 0$ for all $\eta_j \in \mathcal{H}^0_1$ so $T$ is zero on $\mathcal{K}_{1,1}$ and thus $\sigma(T) = 0$.  Therefore we may assume that $\{A_i\}^n_{i=2}, \{A'_j\}^m_{j=2} \subseteq \mathfrak{A}^0_1$.
\par
By examining Discussion \ref{idealdis1} and Discussion \ref{idealdis2}, notice if $A_1, A'_1 \in \mathfrak{A}_1^0$ then $\sigma(T)$ corresponds to the operator 
\[
A_1B_{1} \cdots A_nB_nJB'_mA'_m\cdots B'_1A'_1 \in \langle \mathfrak{J}\rangle_{\mathfrak{A}_1 \ast \mathfrak{A}_2}.
\]
If $A_1 \in \mathfrak{A}_1^0$ and $A'_1 = I_{\mathfrak{A}_1}$ then $\sigma(T)$ corresponds to the operator 
\[
A_1B_{1} \cdots A_nB_nJB'_mA'_m\cdots A'_2B'_1 \in \langle \mathfrak{J}\rangle_{\mathfrak{A}_1 \ast \mathfrak{A}_2}.
\]
If $A'_1 \in \mathfrak{A}_1^0$ and $A_1 = I_{\mathfrak{A}_1}$ then $\sigma(T)$ corresponds to the operator 
\[
B_{1}A_2 \cdots A_nB_nJB'_mA'_m\cdots B'_1A'_1 \in \langle \mathfrak{J}\rangle_{\mathfrak{A}_1 \ast \mathfrak{A}_2}.
\]
Finally if $A_1, A'_1 = I_{\mathfrak{A}_1}$ then $\sigma(T)$ corresponds to the operator 
\[
B_{1}A_2 \cdots A_nB_nJB'_mA'_m\cdots A'_2B'_1 \in \langle \mathfrak{J}\rangle_{\mathfrak{A}_1 \ast \mathfrak{A}_2}.
\]
Hence $\sigma(T) \in \langle \mathfrak{J}\rangle_{\mathfrak{A}_1 \ast \mathfrak{A}_2}$ for all $T$ of this form.  The case where $m = 0$ is identical (as $J\xi_1 = 0$ for all $J \in \mathfrak{J}$).
\par
Using similar arguments to those listed above, the image of $\mathcal{K}_{1,1}$ under 
\[
(B_{1}S)(A_{2}S) \cdots (A_nS)(B_nS)J(S^*B'_m)(S^*A'_m)\cdots (S^*B'_1)(S^*A'_1)
\]
is in the orthogonal complement of $\mathcal{K}_{1,1}$ and thus these operators have zero compression for all $n,m\geq 0$,  $J \in \mathfrak{J}$, $\{A_i\}^n_{i=1}, \{A'_j\}^m_{j=1} \subseteq \mathfrak{A}^0_1 \cup\{I_{\mathfrak{A}_1}\}$, and $\{B_i\}^n_{i=1}, \{B'_j\}^m_{j=1} \subseteq \mathfrak{A}^0_2 \cup\{I_{\mathfrak{A}_2}\}$.  
\par
As the span of $(*)$ is dense in $\langle \mathfrak{J}\rangle_{C^*(\mathfrak{A}_1 \oplus \mathfrak{A}_2, S)}$, the result follows.
\end{proof}
\begin{proof}[Proof of Theorem \ref{main}]  Suppose $T \in ((\mathfrak{A}_1, \pi_1, \xi_1) \ast (\mathfrak{A}_2, \pi_2, \xi_2)) \cap ker(\pi)$.  Therefore
\[
\pi'(\Psi(T)) = \Psi_0(\pi(T)) = 0
\]
by Lemma \ref{comm}.  By Theorem \ref{dse}
\[
\Psi(T) \in \langle \mathfrak{J}\rangle_{C^*(\mathfrak{A}_1 \oplus \mathfrak{A}_2, S)}. 
\]
Therefore
\[
\sigma(\Psi(T)) \in \langle \mathfrak{J}\rangle_{(\mathfrak{A}_1, \pi_1, \xi_1) \ast (\mathfrak{A}_2, \pi_2, \xi_2)}
\]
by Lemma \ref{compideal}.  However $T = \sigma(\Psi(T))$ by Lemma \ref{BOannoy} so
\[
T \in \langle \mathfrak{J}\rangle_{(\mathfrak{A}_1, \pi_1, \xi_1) \ast (\mathfrak{A}_2, \pi_2, \xi_2)}
\]
as desired.  
\end{proof}

\section{Limits of Free Products}
\label{sect:LimitsFreeProducts}

With the modification to the third equivalence of Theorem \ref{exactnessTheorem} complete, we turn our attention to developing the analog of the fifth equivalence of Theorem \ref{exactnessTheorem} in the context of reduced free products.  The following is the adaptation of the fifth equivalence of Theorem \ref{exactnessTheorem} to reduced free products and is a generalization of the appendix of \cite{Ma} due to Shlyakhtenko (where, if $\mathfrak{A}_i$ are C$^*$-algebras with states $\varphi_i$ that have faithful GNS representations, $(\mathfrak{A}_1, \varphi_1) \ast (\mathfrak{A}_2, \varphi_2)$ is the reduced free product, $\varphi_1 \ast \varphi_2$ is the vector state on $(\mathfrak{A}_1, \varphi_1) \ast (\mathfrak{A}_2, \varphi_2)$ corresponding to the distinguished vector, $\mathbb{C} \langle t_1, \ldots, t_n\rangle$ denotes set of all complex polynomials in $n$ non-commuting variables and their complex conjugates, and a pair $(\mathfrak{A}, \tau)$ is said to be a non-commutative probability space if $\mathfrak{A}$ is a unital C$^*$-algebra and $\tau$ is a state on $\mathfrak{A}$ with a faithful GNS representation):
\begin{thm2}
\label{ultraproduct}
For each $k \in \mathbb{N}$ let $\left\{X_i^{(k)}\right\}^n_{i=1}$ be generators for a non-commutative probability space $(\mathfrak{A}_k, \tau_k)$.  Let $\{X_i\}^n_{i=1}$ be generators for a non-commutative probability space $(\mathfrak{A}, \tau)$ and let $\{Y_i\}^m_{i=1}$ be generators for a non-commutative probability space $(\mathfrak{B}, \varphi)$.   Suppose that 
\begin{enumerate}
	\item $\limsup_{k\to\infty}\left\|q\left(X_1^{(k)}, \ldots, X_n^{(k)}\right)\right\|_{\mathfrak{A}_k} = \left\|q(X_1,\ldots, X_n)\right\|_{\mathfrak{A}}$ for all $q \in \mathbb{C} \langle t_1, \ldots, t_n\rangle$, and
	\item $\lim_{k\to\infty} \tau_k\left(q\left(\left(X_1^{(k)}, \ldots, X_n^{(k)}\right)\right)\right) = \tau(q(X_1,\ldots,X_n))$ for all $q \in \mathbb{C} \langle t_1, \ldots, t_n\rangle$.
\end{enumerate}
Then
\[
\lim_{k\to\infty} \left\|p\left(X_1^{(k)}, \ldots, X^{(k)}_n, Y_1, \ldots, Y_m\right)\right\|_{(\mathfrak{A}_k, \tau_k) \ast (\mathfrak{B}, \varphi)} = \left\|p(X_1, \ldots, X_n, Y_1, \ldots, Y_m)\right\|_{(\mathfrak{A}, \tau) \ast (\mathfrak{B}, \varphi)}
\]
for all $p \in \mathbb{C} \langle t_1, \ldots, t_{n+m}\rangle$.
\end{thm2}
By examining the fifth equivalence of Theorem \ref{exactnessTheorem}, it can easily be seen that the above question is connected with the notion of an exact C$^*$-algebra by replacing tensor products with reduced free products.  To begin the proof of Theorem \ref{ultraproduct}, we note one inequality is trivially implied.
\par
\begin{lem2}
\label{oneinequal}
With the assumptions and notation of Theorem \ref{ultraproduct}, 
\[
\liminf_{k\to\infty} \left\|p\left(X_1^{(k)}, \ldots, X^{(k)}_n, Y_1, \ldots, Y_m\right)\right\|_{(\mathfrak{A}_k, \tau_k) \ast (\mathfrak{B}, \varphi)} \geq \left\|p(X_1, \ldots, X_n, Y_1, \ldots, Y_m)\right\|_{(\mathfrak{A}, \tau) \ast (\mathfrak{B}, \varphi)}
\]
for all $p \in \mathbb{C} \langle t_1, \ldots, t_{n+m}\rangle$.
\end{lem2}
\begin{proof}
First we claim if $p \in \mathbb{C} \langle t_1, \ldots, t_{n+m}\rangle$ is arbitrary then
\[
(\tau \ast \varphi)(p(X_1, \ldots, X_n, Y_1, \ldots, Y_m)) = \lim_{k\to\infty} (\tau_k \ast \varphi)\left(p\left(X_1^{(k)}, \ldots, X^{(k)}_n, Y_1, \ldots, Y_m\right)\right).
\]
To see this, notice $p(t_1,\ldots, t_{n+m})$ can be written as 
\[
p(t_1,\ldots, t_{n+m}) = \sum^N_{\ell =1} \prod^{z_\ell}_{w=1} p_{\ell,w}(t_1, \ldots, t_n)q_{\ell,w}(t_{n+1}, \ldots, t_{n+m})
\]
where $\tau(p_{\ell,w}(X_1, \ldots, X_n)) = 0$ and $\varphi(q_{\ell,w}(Y_1, \ldots, Y_m)) = 0$ for all $\ell \in \{1,\ldots, N\}$ and $w \in \{1,\ldots, z_\ell\}$ except possible for possible $p_{\ell, 1}$ and $q_{\ell, z_\ell}$ which can be constant functions.  Thus
\[
(\tau \ast \varphi)(p(X_1, \ldots, X_n, Y_1, \ldots, Y_m)) = \sum^N_{\ell =1} \prod^{z_\ell}_{w=1}\tau(p_{\ell,w}(X_1, \ldots, X_n))\varphi(q_{\ell,w}(Y_1, \ldots, Y_m))
\]
by freeness.  Moreover
\[
(\tau_k \ast \varphi)\left(p\left(X_1^{(k)}, \ldots, X^{(k)}_n, Y_1, \ldots, Y_m\right)\right) = \sum^N_{\ell =1} (\tau_k \ast \varphi)\left(\prod^{z_\ell}_{w=1} p_{\ell,w}\left(X^{(k)}_1, \ldots, X^{(k)}_n\right) q_{\ell,w}(Y_1, \ldots, Y_m)  \right)
\]
by linearity.  Thus it suffices to show
\[
\prod^{z_\ell}_{w=1}\tau(p_{\ell,w}(X_1, \ldots, X_n))\varphi(q_{\ell,w}(Y_1, \ldots, Y_m)) = \lim_{k\to\infty} (\tau_k \ast \varphi)\left( \prod^{z_\ell}_{w=1}  p_{\ell,w}\left(X^{(k)}_1, \ldots, X^{(k)}_n \right) q_{\ell,w}(Y_1, \ldots, Y_m) \right)
\]
for all $\ell \in \{1,\ldots, N\}$.  However, in the case that $\tau(p_{\ell,w}(X_1, \ldots, X_n)) = 0$ and $\varphi(q_{\ell,w}(Y_1, \ldots, Y_m)) = 0$ for all $w \in \{1,\ldots, z_\ell\}$, notice
\begin{eqnarray} 
&&\!\!\!\!  \prod^{z_\ell}_{w=1}  \left(p_{\ell,w}\left(X^{(k)}_1, \ldots, X^{(k)}_n \right) - \tau_k\left(p_{\ell,w}\left(X^{(k)}_1, \ldots, X^{(k)}_n \right) \right)I_{\mathfrak{A}_k} \right) q_{\ell,w}(Y_1, \ldots, Y_m) \nonumber\\
&=&\!\!\!\!  \left(\prod^{z_\ell}_{w=1}  p_{\ell,w}\left(X^{(k)}_1, \ldots, X^{(k)}_n \right) q_{\ell,w}(Y_1, \ldots, Y_m) \right) + T_k \nonumber
\end{eqnarray}
where $T_k$ is the sum of products of terms in 
\[
\left\{p_{\ell,w}\left(X^{(k)}_1, \ldots, X^{(k)}_n \right), \tau_k\left(p_{\ell,w}\left(X^{(k)}_1, \ldots, X^{(k)}_n \right)\right), q_{\ell,w}(Y_1, \ldots, Y_m)\, \mid \, w \in \{1,\ldots, z_\ell\}    \right\}\,\,\,\,\,\,\,\,\,(*)
\]
where each product contains at least one $\tau_k\left(p_{\ell,w}\left(X^{(k)}_1, \ldots, X^{(k)}_n\right) \right)$ and $T_{k'}$ can be obtained from $T_k$ by exchanging the index $k$ with $k'$.   Hence
\begin{eqnarray} 
&&\!\!\!\!  (\tau_k \ast \varphi)\left(\prod^{z_\ell}_{w=1}  p_{\ell,w}\left(X^{(k)}_1, \ldots, X^{(k)}_n \right) q_{\ell,w}(Y_1, \ldots, Y_m)\right) + (\tau_k \ast \varphi)(T_k) \nonumber\\
&=&\!\!\!\!  (\tau_k \ast \varphi) \left(\prod^{z_\ell}_{w=1}  \left(p_{\ell,w}\left(X^{(k)}_1, \ldots, X^{(k)}_n \right) - \tau_k\left(p_{\ell,w}\left(X^{(k)}_1, \ldots, X^{(k)}_n \right) \right)I_{\mathfrak{A}_k} \right)q_{\ell,w}(Y_1, \ldots, Y_m)\right) \nonumber\\
&=&\!\!\!\!  \prod^{z_\ell}_{w=1}  \tau_k\left(p_{\ell,w}\left(X^{(k)}_1, \ldots, X^{(k)}_n \right) - \tau_k\left(p_{\ell,w}\left(X^{(k)}_1, \ldots, X^{(k)}_n \right) \right)I_{\mathfrak{A}_k} \right) \varphi(q_{\ell,w}(Y_1, \ldots, Y_m)) \nonumber\\
&=&\!\!\!\!  0 \nonumber
\end{eqnarray}
by freeness.  Since $\tau(p_{\ell,w}(X_1, \ldots, X_n)) = 0$ and $\varphi(q_{\ell,w}(Y_1, \ldots, Y_m)) = 0$ for all $w \in \{1,\ldots, z_\ell\}$, the assumptions of the lemma imply
\[
\lim_{k\to \infty} (\tau_k \ast \varphi)(T_k) = 0
\]
as every term in $(*)$ is bounded by the first assumption of the lemma and
\[
\lim_{k\to \infty} \tau_k\left(p_{\ell,w}\left(X^{(k)}_1, \ldots, X^{(k)}_n \right)\right) = \tau(p_{\ell,w}(X_1, \ldots, X_n)) = 0
\]
for all $w \in \{1,\ldots, z_\ell\}$ by the second assumption of the lemma.  Hence
\[
 \lim_{k\to \infty} (\tau_k \ast \varphi)\left(\prod^{z_\ell}_{w=1}  p_{\ell,w}\left(X^{(k)}_1, \ldots, X^{(k)}_n \right) q_{\ell,w}(Y_1, \ldots, Y_m)\right) = 0 = \prod^{z_\ell}_{w=1}\tau(p_{\ell,w}(X_1, \ldots, X_n))\varphi(q_{\ell,w}(Y_1, \ldots, Y_m)).
\]
As similar computations hold when $p_{\ell, 1}$ and/or $q_{\ell, z_\ell}$ are constants, the claim has been proven.
\par
For each $k \in \mathbb{N}$ let $\left\|T\right\|_{2, \tau_k \ast \varphi} = (\tau_k \ast \varphi)(T^*T)^\frac{1}{2}$ for all $T \in (\mathfrak{A}_k, \tau_k) \ast (\mathfrak{B}, \varphi)$ and let $\left\|T\right\|_{2, \tau \ast \varphi} = (\tau \ast \varphi)(T^*T)^\frac{1}{2}$ for all $T \in (\mathfrak{A}, \tau) \ast (\mathfrak{B}, \varphi)$.  Notice the above implies that
\[
\lim_{k\to\infty} \left\|p\left(X_1^{(k)}, \ldots, X^{(k)}_n, Y_1, \ldots, Y_m \right)\right\|_{2,\tau_k\ast \varphi} = \left\|p(X_1, \ldots, X_n, Y_1, \ldots, Y_m)\right\|_{2,\tau\ast\varphi}
\]
for all polynomials $p \in \mathbb{C}\langle t_1, \ldots, t_{n+m}\rangle$.  However, by considering the construction of the reduced free product, for a fixed polynomial $p \in \mathbb{C}\langle t_1, \ldots, t_{n+m}\rangle$
\begin{eqnarray} 
&&\!\!\!\! \left\|p(X_1, \ldots, X_n, Y_1, \ldots, Y_m)\right\|\nonumber\\
&=&\!\!\!\!\sup\left\{|(\tau\ast\varphi)((p\cdot p_1 \cdot p_2)(X_1, \ldots, X_n, Y_1, \ldots, Y_m))| \, \left|\, 
\begin{array}{l}
p_i \in \mathbb{C}\langle t_1,\ldots, t_{n+m}\rangle,  \\
\left\|p_i(X_1, \ldots, X_n, Y_1, \ldots, Y_m)\right\|_{2,\tau\ast\varphi} < 1
\end{array}
\right.
\right\}. \nonumber
\end{eqnarray}
However, for all $p, p_1, p_2 \in \mathbb{C}\langle t_1, \ldots, t_{n+m}\rangle$,
\begin{eqnarray} 
\!\!\!\!&&\!\!\!\! |(\tau\ast\varphi)((p\cdot p_1 \cdot p_2)(X_1, \ldots, X_n, Y_1, \ldots, Y_m))|\nonumber\\
&=&\!\!\!\! \lim_{k\to\infty} \left|(\tau_k\ast\varphi)\left((p\cdot p_1 \cdot p_2)\left(X^{(k)}_1, \ldots, X^{(k)}_n, Y_1, \ldots, Y_m\right)\right)\right|\nonumber\\
&\leq&\!\!\!\! \liminf_{k\to\infty} \left(\prod_{i=1,2}\left\|p_i\left(X_1^{(k)}, \ldots, X^{(k)}_n, Y_1, \ldots, Y_m\right)\right\|_{2,\tau_k\ast \varphi} \right) \left\|p\left(X^{(k)}_1, \ldots, X^{(k)}_n, Y_1, \ldots, Y_m\right)\right\|_{(\mathfrak{A},\tau_k)\ast(\mathfrak{B},\varphi)}  \nonumber\\
&=&\!\!\!\! \left(\prod_{i=1,2}\left\|p_i(X_1, \ldots, X_n, Y_1, \ldots, Y_m)\right\|_{2,\tau\ast\varphi}\right) \liminf_{k\to\infty} \left\|p\left(X^{(k)}_1, \ldots, X^{(k)}_n, Y_1, \ldots, Y_m\right)\right\|_{(\mathfrak{A},\tau_k)\ast(\mathfrak{B},\varphi)} \nonumber
\end{eqnarray}
Hence, as the above holds for all $ p_1, p_2 \in \mathbb{C}\langle t_1, \ldots, t_{n+m}\rangle$, the result follows. 
\end{proof}
\begin{rem2}
\label{3.3}
Using the notation in Theorem \ref{ultraproduct}, consider the C$^*$-subalgebra $\mathfrak{C}$ of 
\[
\frac{\prod_{k\geq1} \left((\mathfrak{A}_k, \tau_k) \ast (\mathfrak{B},\varphi) \right)}{\bigoplus_{k\geq1} \left((\mathfrak{A}_k, \tau_k) \ast (\mathfrak{B},\varphi)\right)}
\]
generated by 
\[
\left\{  \left(X^{(k)}_j\right)_{k\geq 1} + \bigoplus_{k\geq1} \left((\mathfrak{A}_k, \tau_k) \ast (\mathfrak{B},\varphi)\right) \, \mid \, j \in \{1,\ldots, n\}\right\}
\]
and
\[
\left\{  \left( Y_j\right)_{k\geq 1} + \bigoplus_{k\geq1} \left((\mathfrak{A}_k, \tau_k) \ast (\mathfrak{B},\varphi)\right) \, \mid \, j \in \{1,\ldots, m\}\right\}.
\]
Lemma \ref{oneinequal} tells us that there exists a map
\[
\Psi : \mathfrak{C} \to (\mathfrak{A}, \tau) \ast (\mathfrak{B}, \varphi)
\]
defined by
\[
\Psi\left(\left(X^{(k)}_j\right)_{k\geq 1} + \bigoplus_{k\geq1} \left((\mathfrak{A}_k, \tau_k) \ast (\mathfrak{B},\varphi)\right) \right) = X_j
\]
for all $j \in \{1,\ldots, n\}$ and
\[
\Psi\left(\left(Y_j\right)_{k\geq 1} + \bigoplus_{k\geq1} \left((\mathfrak{A}_k, \tau_k) \ast (\mathfrak{B},\varphi)\right) \right) = Y_j
\]
for all $j \in \{1,\ldots, m\}$.  Moreover $\Psi$ is an isomorphism if and only if 
\[
\limsup_{k\to\infty} \left\|p\left(X_1^{(k)}, \ldots, X^{(k)}_n, Y_1, \ldots, Y_m\right)\right\|_{(\mathfrak{A}_k, \tau_k) \ast (\mathfrak{B}, \varphi)} \leq \left\|p(X_1, \ldots, X_n, Y_1, \ldots, Y_m)\right\|_{(\mathfrak{A}, \tau) \ast (\mathfrak{B}, \varphi)}
\]
for all polynomials $p \in \mathbb{C}\langle t_1,\ldots, t_{n+m}\rangle$.  Thus Theorem \ref{ultraproduct} is true if and only if $\Psi$ is an isomorphism.  The question of whether $\Psi$ is an isomorphism can be considered as a modification of the fourth equivalence of Theorem \ref{exactnessTheorem}.
\end{rem2}
Our next goal is to prove Theorem \ref{ultraproduct} provided that $\mathfrak{B}$ is an exact C$^*$-algebra.  To do this we reprove the following known results from the appendix of \cite{Ma} that prove Theorem \ref{ultraproduct} when the $Y_j$ are free creation operators on a Fock space. 
\begin{lem2} 
\label{3.4}
Let $\mathfrak{A}$ be a unital C$^*$-algebra with a state $\tau$ with a faithful GNS representation and let $\mathfrak{B}$ be the universal C$^*$-algebra generated by $\mathfrak{A}$ and elements $L_1, \ldots, L_n$ satisfying $L_i^* A L_j = \delta_{i,j} \tau(A)$ for all $A \in \mathfrak{A}$ (where $\delta_{i,j}$ is the Kronecker delta function).  Let $\psi$ be the linear functional on $^*$-$Alg(\mathfrak{A}, \{L_j\}^n_{j=1})$ defined by $\psi|_\mathfrak{A} = \tau$ and
\[
\psi(A_0 L_{i_1} A_1 \cdots A_{k-1} L_{i_k} A_k A'_0 L_{j_1}^* A'_1 \cdots A'_{\ell-1} L_{j_\ell}^* A'_\ell) = 0
\]
whenever $A_1, \ldots, A_k, A'_1,\ldots, A'_\ell \in \mathfrak{A}$ and at least one of $k$ and $\ell$ is non-zero.  Then $\psi$ extends to a state on $\mathfrak{B}$ having a faithful GNS representation.   Moreover, if $(\mathfrak{A}, \tau) \ast (\mathcal{E}, \phi)$ where $(\mathcal{E}, \phi)$ is the C$^*$-algebra generated by $n$ free creation operators $\ell_1,\ldots, \ell_n$ on the full Fock space $\mathcal{F}(\mathbb{C}^n)$ and $\phi$ is the vacuum expectation, there exists an isomorphism $\Phi : (\mathfrak{B}, \psi) \to (\mathfrak{A}, \tau) \ast (\mathcal{E}, \phi)$ such that $\Phi(A) = A$ for all $A \in \mathfrak{A}$ and $\Phi(L_j) = \ell_j$ for all $j \in \{1,\ldots, m\}$.
\end{lem2}
\begin{proof}
Let $\left(\widehat{\mathfrak{B}}, \widehat{\psi}\right)$ be the reduced free product $(\mathfrak{A}, \tau) \ast (\mathcal{E}, \phi)$.  By Corollary 2.5 of \cite{Sh} $\ell_i A \ell_j^* = \delta_{i,j} \tau(A)$ for all $A \in \mathfrak{A}$ and
\[
\widehat{\psi}(A_0 \ell_{i_1} A_1 \cdots A_{k-1} \ell_{i_k} A_k A'_0 \ell_{j_1}^* A'_1 \cdots A'_{\ell-1} \ell_{j_\ell}^* A'_\ell) = 0
\]
whenever $A_1, \ldots, A_k, A'_1,\ldots, A'_\ell \in \mathfrak{A}$ and at least one of $k$ and $\ell$ is non-zero.  Hence, by the universal property of $\mathfrak{B}$, there exists a $^*$-homomorphism $\Phi : \mathfrak{B} \to \widehat{\mathfrak{B}}$ such that $\psi = \widehat{\psi} \circ \Phi$.  
\par
To complete the lemma it suffices to prove $\Phi$ is injective.  However, by \cite{Pi} (and by applying the same `Fourier series'-like argument as in Section 2.3), it suffices to check that the linear span of $\{A L_i^* B L_jC \, \mid \, i,j \in \{1,\ldots, n\}, A,B,C \in\mathfrak{A}\}$ is dense in $\mathfrak{A}$ and that there exists a homomorphism $\alpha : \{z \in \mathbb{C} \, \mid \, |z| = 1\} \to Hom(\widehat{\mathfrak{B}})$ such that $\alpha_z(A) = A$ for all $A \in \mathfrak{A}$ and $\alpha_z(\ell_j) = z\ell_j$ for all $j \in \{1,\ldots, n\}$.  However the first claim is trivial by taking $i = j$, $B = I_\mathfrak{A} = C$.  Since it is trivial to verify that there exists a homomorphism $\alpha : \{z \in \mathbb{C} \, \mid \, |z| = 1\} \to Hom(\widehat{\mathfrak{B}})$ such that $\alpha_z(\ell_j) = z\ell_j$ for all $j \in \{1,\ldots, n\}$, taking the free product with the identity map on $\mathfrak{A}$ will complete the lemma.  
\end{proof}
\begin{lem2}
\label{3.5}
Theorem \ref{ultraproduct} is true with the additional assumptions that $\mathfrak{B}$ is the C$^*$-algebra generated by $m$ creation operators $\ell_1, \ldots, \ell_m$ on a Fock space and $\varphi$ is the vector state of the vacuum vector.
\end{lem2}
\begin{proof}
Consider the C$^*$-algebra
\[
\mathfrak{D} := \frac{\prod_{k\geq 1} \left((\mathfrak{A}_k, \tau_k) \ast (C^*(\ell_1, \ldots,\ell_m), \varphi)\right)}{\bigoplus_{k\geq 1} \left((\mathfrak{A}_k, \tau_k) \ast (C^*(\ell_1, \ldots,\ell_n), \varphi)\right)}.
\]
Let
\[
X'_j := \left(X^{(k)}_j\right)_{k\geq 1} + \bigoplus_{k\geq 1} \left((\mathfrak{A}_k, \tau_k) \ast (C^*(\ell_1, \ldots,\ell_n), \varphi)\right)
\]
for $j \in \{1,\ldots, n\}$ and
\[
L_j := \left(\ell_j\right)_{k\geq 1} + \bigoplus_{k\geq 1} \left((\mathfrak{A}_k, \tau_k) \ast (C^*(\ell_1, \ldots,\ell_n), \varphi)\right)
\]
for $j \in \{1,\ldots, m\}$.  Notice, by the first assumption of the theorem, $\mathfrak{A}$ is isomorphic to the C$^*$-subalgebra of $\mathfrak{D}$ generated by $\{X'_j\}^n_{j=1}$.  Let $\mathfrak{C}$ be the C$^*$-subalgebra of $\mathfrak{D}$ generated by $\mathfrak{A}$ and $\{L_j\}^m_{j=1}$.  By Remarks \ref{3.3} there exists a $^*$-homomorphism $\Psi : \mathfrak{C} \to (\mathfrak{A},\tau) \ast (C^*(\ell_1, \ldots, \ell_m), \varphi)$ such that
$\Psi(X'_j) = X_j$ for all $j \in \{1,\ldots, n\}$ and $\Psi(L_j) = \ell_j$ for all $j \in \{1,\ldots, m\}$.
\par
We claim that $\Psi$ is an isomorphism.  To see this, notice for all polynomials $p \in \mathbb{C} \langle t_1,\ldots, t_n\rangle$ that
\begin{eqnarray} 
L^*_i p(X'_1, \ldots, X'_n)L_j \!\!\!\!&=&\!\!\!\! \left(\ell_i^* p\left(x^{(k)}_1, \ldots, x^{(k)}_n\right)\ell_j \right)_{k\geq 1} + \bigoplus_{k\geq 1} \left((\mathfrak{A}_k, \tau_k) \ast (C^*(\ell_1, \ldots,\ell_n), \varphi)\right) \nonumber\\
&=&\!\!\!\! \delta_{i,j}\left(\tau_k\left(p\left(x^{(k)}_1, \ldots, x^{(k)}_n\right)\right) \right)_{k\geq 1} + \bigoplus_{k\geq 1} \left((\mathfrak{A}_k, \tau_k) \ast (C^*(\ell_1, \ldots,\ell_n), \varphi)\right) \nonumber\\
&=&\!\!\!\! \delta_{i,j}\tau( p(X'_1, \ldots, X'_n)) \left(I_{(\mathfrak{A}_k, \tau_k) \ast (C^*(\ell_1, \ldots,\ell_n), \varphi)} \right)_{k\geq 1} + \bigoplus_{k\geq 1} \left((\mathfrak{A}_k, \tau_k) \ast (C^*(\ell_1, \ldots,\ell_n), \varphi)\right) \nonumber\\
&=&\!\!\!\! \delta_{i,j}\tau( p(X'_1, \ldots, X'_n)) I_\mathfrak{A} \nonumber
\end{eqnarray}
by the second assumption of the theorem.  Hence, by Lemma \ref{3.4} and by universality, there exists a $^*$-homomorphism $\Phi : (\mathfrak{A},\tau) \ast (C^*(\ell_1, \ldots, \ell_m), \varphi) \to \mathfrak{C}$ such that $\Phi(X_j) = X'_j$ for all $j \in \{1,\ldots, n\}$ and $\Phi(\ell_j) = L_j$ for all $j \in \{1,\ldots, m\}$.  Hence $\Psi$ is invertible with inverse $\Phi$.  Thus the result follows from Remarks \ref{3.3}.
\end{proof}
Our next goal is to prove Theorem \ref{ultraproduct} provided $C^*(Y_1,\ldots, Y_m)$ is exact.  To do this we desire an embedding of the reduced free product of two C$^*$-algebras $\mathfrak{A}$ and $\mathfrak{B}$ into a reduced free product involving $\mathfrak{A} \otimes_{\min} \mathfrak{B}$.
\begin{lem2}
\label{2.6.6}
Let $\mathfrak{A}$ and $\mathfrak{B}$ be unital C$^*$-algebras, let $\varphi$ and $\psi$ be states on $\mathfrak{A}$ and $\mathfrak{B}$ respectively with faithful GNS representations.  Let $\ell_1$ be the unilateral forward shift on $\ell_2(\mathbb{N})$, let $\{e_n\}_{n\geq1}$ be the standard orthonormal basis for $\ell_2(\mathbb{N})$, and let $\phi : C^*(\ell_1) \to \mathbb{C}$ be defined by $\phi(T) = \langle Te_1,e_1\rangle$ for all $T \in C^*(\ell_1)$.  Then there exists a unitary $U \in C^*(\ell_1)$ (independent of $\mathfrak{A}$ and $\mathfrak{B}$) and an injective $^*$-homomorphism 
\[
\Psi : (\mathfrak{A}, \varphi) \ast (\mathfrak{B}, \psi) \to (\mathfrak{A} \otimes_{\min} \mathfrak{B}, \varphi \otimes \psi) \ast (C^*(\ell_1), \phi)
\]
such that $\Psi(A) = A \otimes I_\mathfrak{B}$ and $\Psi(B) = U^*(I_\mathfrak{A} \otimes B)U$ for all $A \in \mathfrak{A}$ and $B \in \mathfrak{B}$. 
\end{lem2}
\begin{proof}
See Proposition 4.2 of \cite{DS}.  Alternatively a (not necessarily injective) $^*$-homomorphism can be constructed by Theorem 4.7.2 of \cite{BO} and by showing that $\mathfrak{A} \otimes I_\mathfrak{B}$ and $U^*(I_\mathfrak{A} \otimes B)U$ are free with respect to any self-adjoint unitary $U$ such that $Ue_1 = e_2$ and $Ue_2 = e_1$.  The proof that $\Psi$ is then injective can be done by constructing a compression map from the Hilbert space that $(\mathfrak{A} \otimes_{\min} \mathfrak{B}, \varphi \otimes \psi) \ast (C^*(\ell_1), \phi)$ acts on to an isomorphic copy of the Hilbert space $(\mathfrak{A}, \varphi) \ast (\mathfrak{B}, \psi)$ acts on.
\end{proof}
\begin{lem2}
\label{ultraexact}
Theorem \ref{ultraproduct} is true under the additional assumption that $\mathfrak{B}$ is an exact C$^*$-algebra.
\end{lem2}
\begin{proof}
Since $\mathfrak{B}$ is exact, by the fifth equivalence of Theorem \ref{exactnessTheorem} and by the first assumption of the lemma, we obtain that
\begin{eqnarray} 
&&\!\!\!\! \limsup_{k\to\infty} \left\|p\left(X_1^{(k)} \otimes I, \ldots, X^{(k)}_n\otimes I, I \otimes Y_1, \ldots, I \otimes Y_m\right)\right\|_{\mathfrak{A}_k \otimes_{\min} \mathfrak{B}}\nonumber\\
&=&\!\!\!\! \left\|p(X_1\otimes I, \ldots, X_n\otimes I, I \otimes Y_1, \ldots, I \otimes Y_m)\right\|_{\mathfrak{A} \otimes \mathfrak{B}} \nonumber
\end{eqnarray}
for all $p \in \mathbb{C} \langle t_1, \ldots, t_{n+m}\rangle$.  By the structure of the states on the tensor products and by the second assumption of the lemma,
\[
\lim_{k\to\infty} (\tau_k \otimes \varphi)\left(p\left(X_1^{(k)} \otimes I, \ldots, X^{(k)}_n\otimes I, I \otimes Y_1, \ldots, I \otimes Y_m\right)\right) = (\tau \otimes \varphi)(p(X_1\otimes I, \ldots, X_n\otimes I, I \otimes Y_1, \ldots, I \otimes Y_m))
\]
for all $p \in \mathbb{C} \langle t_1, \ldots, t_{n+m}\rangle$.  Therefore Lemma \ref{3.5} implies 
\begin{eqnarray} 
&&\!\!\!\! \lim_{k\to\infty} \left\|p\left(X_1^{(k)} \otimes I, \ldots, X^{(k)}_n\otimes I, I \otimes Y_1, \ldots, I \otimes Y_m, T\right)\right\|_{(\mathfrak{A}_k \otimes_{\min} \mathfrak{B}, \tau_k \otimes \varphi) \ast (C^*(\ell_1), e_1) } \nonumber\\
&=&\!\!\!\! \left\|p(X_1\otimes I, \ldots, X_n\otimes I, I \otimes Y_1, \ldots, I \otimes Y_m, T)\right\|_{(\mathfrak{A} \otimes_{\min}\mathfrak{B}, \tau \otimes \varphi) \ast (C^*(\ell_1), e_1) } \nonumber
\end{eqnarray}
for all $p \in \mathbb{C} \langle t_1, \ldots, t_{n+m+1}\rangle$ and for all $T \in C^*(\ell_1)$.  By using $T = U$ where $U$ is a unitary as in Lemma \ref{2.6.6} and by viewing $(\mathfrak{A}_k, \tau_k) \ast (\mathfrak{B}, \varphi)$ and $(\mathfrak{A}, \tau) \ast (\mathfrak{B}, \varphi)$ in $(\mathfrak{A}_k \otimes_{\min} \mathfrak{B}, \tau_k \otimes \varphi) \ast (C^*(\ell_1), e_1)$ and $(\mathfrak{A} \otimes_{\min} \mathfrak{B}, \tau \otimes \varphi) \ast (C^*(\ell_1), e_1)$ respectively, the result follows.  
\end{proof}
Just as Lemma \ref{ultraexact} upgraded Lemma \ref{3.5} to exact C$^*$-algebras by use of Lemma \ref{2.6.6} and tensor products, we will use Lemma \ref{ultraexact} along with the following lemma involving direct sums to prove Theorem \ref{ultraproduct}.
\begin{lem2}
\label{directsumlemma}
For $i \in \{1,2\}$ let $(\mathfrak{A}_i,\tau_i)$ be non-commutative probability space.  Let $\tau : \mathfrak{A}_1 \oplus \mathfrak{A}_2 \to \mathbb{C}$ be the state given by
\[
\tau(A_1 \oplus A_2) = \frac{1}{2}(\tau_1(A_1) + \tau_2(A_2))
\]
for all $A_1 \in \mathfrak{A}_1$ and $A_2 \in \mathfrak{A}_2$.  
\par
Let $\mathcal{O}_2$ be the Cuntz algebra, let $\mathcal{F}_2$ be the canonical CAR C$^*$-subalgebra of $\mathcal{O}_2$, let $\tau' : \mathcal{F}_2 \to \mathbb{C}$ be the unique normalized trace on $\mathcal{F}_2$, let $\mathcal{E} : \mathcal{O}_2 \to \mathcal{F}_2$ be the canonical conditional expectation of $\mathcal{O}_2$ onto $\mathcal{F}_2$, and let $\sigma:= \tau' \circ \mathcal{E} : \mathcal{O}_2 \to \mathbb{C}$.  Note $\sigma$ is a faithful state.
\par
Let $\mathfrak{C}$ be any C$^*$-algebra with a state $\rho$ such that there exists a unitary $U \in \mathfrak{C}$ such that $\rho|_{C^*(U)}$ is faithful, $\rho(U) = 0$, and the GNS representation of $\mathfrak{C}$ with respect to $\rho$ is faithful.  
\par
Then there exists an injective $^*$-homomorphism $\pi : (\mathfrak{A}_1, \tau_1) \ast (\mathfrak{A}_2, \tau_2) \to ((\mathfrak{A}_1\oplus \mathfrak{A}_2) \otimes \mathcal{O}_2, \tau \otimes \sigma) \ast (\mathfrak{C}, \rho)$ such that there exists elements $X, Y, Z, W \in C^*(I \otimes \mathcal{O}_2, \mathfrak{C}) \subseteq ((\mathfrak{A}_1\oplus \mathfrak{A}_2) \otimes \mathcal{O}_2, \tau \otimes \sigma) \ast (\mathfrak{C}, \rho)$ independent of the choice of $\mathfrak{A}_1$ and $\mathfrak{A}_2$ such that $\pi(A_1) = X(A_1 \oplus 0)Y$ for all $A_1 \in \mathfrak{A}_1$ and $\pi(A_2) = Z(0 \oplus A_2)W$ for all $A_2 \in \mathfrak{A}_2$.
\end{lem2}
\begin{proof}
See Lemma 5.6 of \cite{BDS}.
\end{proof}
\begin{proof}[Proof of Theorem \ref{ultraproduct}]
For each $k \in \mathbb{N}$ define the state $\psi_k : \mathfrak{A}_k \oplus \mathfrak{B} \to \mathbb{C}$ by $\psi_k(A \oplus B) = \frac{1}{2}(\tau_k(A) + \varphi(B))$ for all $A \in \mathfrak{A}_k$ and $B \in \mathfrak{B}$ and define the state $\psi : \mathfrak{A} \oplus \mathfrak{B} \to \mathbb{C}$ by $\psi(A \oplus B) = \frac{1}{2}(\tau(A) + \varphi(B))$ for all $A \in \mathfrak{A}$ and $B \in \mathfrak{B}$.  By the first assumption of the theorem, it is clear that
\[
\limsup_{k\to\infty}\left\|p\left(X_1^{(k)} \oplus 0, \ldots, X^{(k)}_n\oplus 0, 0 \oplus Y_1, \ldots, 0 \oplus Y_m\right)\right\|_{\mathfrak{A}_k \oplus \mathfrak{B}}
\]
is
\[
\left\|p(X_1\oplus 0, \ldots, X_n\oplus 0, 0 \oplus Y_1, \ldots, 0 \oplus Y_m)\right\|_{\mathfrak{A} \oplus \mathfrak{B}}
\]
for all $p \in \mathbb{C} \langle t_1, \ldots, t_{n+m}\rangle$.
\par
Let $\ell_1$ and $\ell_2$ be two isometries that generated the Cuntz algebra.  Since $\mathcal{O}_2$ is exact, the fifth equivalence of Theorem \ref{exactnessTheorem} implies that the $\limsup_{k\to\infty}$ of
\[
\left\|p\left(\left(X_1^{(k)} \oplus 0\right)\otimes I, \ldots, \left(X^{(k)}_n\oplus 0\right)\otimes I, (0 \oplus Y_1)\otimes I, \ldots, (0 \oplus Y_m)\otimes I, (I \oplus I) \otimes \ell_1, (I \oplus I) \otimes \ell_2\right)\right\|_{(\mathfrak{A}_k \oplus \mathfrak{B}) \otimes_{\min} \mathcal{O}_2}
\]
is
\[
\left\|p\left((X_1 \oplus 0)\otimes I, \ldots, (X_n\oplus 0)\otimes I, (0 \oplus Y_1)\otimes I, \ldots, (0 \oplus Y_m)\otimes I, (I \oplus I) \otimes \ell_1, (I \oplus I) \otimes \ell_2\right)\right\|_{(\mathfrak{A} \oplus \mathfrak{B}) \otimes_{\min} \mathcal{O}_2} 
\]
for all $p \in \mathbb{C} \langle t_1, \ldots, t_{n+m+2}\rangle$.  
\par
Let $\sigma$ be the faithful state from Lemma \ref{directsumlemma}.  Therefore
\[
(\psi_k \otimes \sigma)\left(p\left(\left(X_1^{(k)} \oplus 0\right)\otimes I, \ldots, \left(X^{(k)}_n\oplus 0\right)\otimes I, (0 \oplus Y_1)\otimes I, \ldots, (0 \oplus Y_m)\otimes I, (I \oplus I) \otimes \ell_1, (I \oplus I) \otimes \ell_2\right)\right)
\]
converges to
\[
(\psi \otimes \sigma) \left( p\left((X_1 \oplus 0)\otimes I, \ldots, (X_n\oplus 0)\otimes I, (0 \oplus Y_1)\otimes I, \ldots, (0 \oplus Y_m)\otimes I, (I \oplus I) \otimes \ell_1, (I \oplus I) \otimes \ell_2\right)  \right)
\]
as $k \to \infty$ for all $p \in \mathbb{C} \langle t_1, \ldots, t_{n+m+2}\rangle$ by the second assumption of the theorem, the structure of the $\psi_k$'s, and the structure of the tensor products of states.
\par
Let $\mathfrak{C} = \mathcal{M}_2(\mathbb{C})$, let $\rho$ be the faithful normalized trace on $\mathfrak{C}$, and let
\[
U := \left[  \begin{array}{cc} 0&1\\1&0  \end{array} \right].
\]
Since $\mathfrak{C}$ can be generated by a single operator free from $\rho$, Lemma $\ref{ultraexact}$ implies if $p \in \mathbb{C} \langle t_1, \ldots, t_{n+m+3}\rangle$, $T \in\mathcal{M}_2(\mathbb{C})$,
\begin{eqnarray} 
q_{p}\!\!\!\!&:=&\!\!\!\! p\left((X_1 \oplus 0)\otimes I, \ldots, (X_n\oplus 0)\otimes I, (0 \oplus Y_1)\otimes I, \ldots, (0 \oplus Y_m)\otimes I, (I \oplus I) \otimes \ell_1, (I \oplus I) \otimes \ell_2, T\right)\nonumber\\
&\in&\!\!\!\! ((\mathfrak{A} \oplus \mathfrak{B}) \otimes_{\min} \mathcal{O}_2, \psi \otimes \sigma) \ast (\mathcal{M}_2(\mathbb{C}), \rho), \nonumber
\end{eqnarray}
and
\begin{eqnarray} 
q_{p,k}\!\!\!\!&:=&\!\!\!\! p\left(\left(X_1^{(k)} \oplus 0\right)\otimes I, \ldots, \left(X^{(k)}_n\oplus 0\right)\otimes I, (0 \oplus Y_1)\otimes I, \ldots, (0 \oplus Y_m)\otimes I, (I \oplus I) \otimes \ell_1, (I \oplus I) \otimes \ell_2, T\right) \nonumber\\
&\in&\!\!\!\! ((\mathfrak{A}_k \oplus \mathfrak{B}) \otimes_{\min} \mathcal{O}_2, \psi_k \otimes \sigma) \ast (\mathcal{M}_2(\mathbb{C}),\rho) \nonumber
\end{eqnarray}
for all $k \in \mathbb{N}$ then
\[
\lim_{k\to\infty} \left\| q_{p,k}\right\|_{((\mathfrak{A}_k \oplus \mathfrak{B}) \otimes_{\min} \mathcal{O}_2, \psi_k \otimes \sigma) \ast (\mathcal{M}_2(\mathbb{C}),\rho)} = \left\| q_p\right\|_{((\mathfrak{A} \oplus \mathfrak{B}) \otimes_{\min} \mathcal{O}_2, \psi \otimes \sigma) \ast (\mathcal{M}_2(\mathbb{C}), \rho)}.
\]
Therefore the result clearly follows by the embedding properties given by Lemma \ref{directsumlemma}.
\end{proof}
Combining Theorem \ref{main}, Theorem \ref{ultraproduct}, and Remarks \ref{3.3}, we have the following analog of Lemma \ref{tensorImbed} for reduced free products.
\begin{cor2}
\label{freeImbed}
For each $k \in \mathbb{N}$ let $\left\{X_i^{(k)}\right\}^n_{i=1}$ be generators for a non-commutative probability space $(\mathfrak{A}_k, \tau_k)$.  Let $\{X_i\}^n_{i=1}$ be generators for a non-commutative probability space $(\mathfrak{A}, \tau)$ and let $\{Y_i\}^m_{i=1}$ be generators for a non-commutative probability space $(\mathfrak{B}, \varphi)$.   Suppose that 
\begin{enumerate}
	\item $\limsup_{k\to\infty}\left\|q\left(X_1^{(k)}, \ldots, X_n^{(k)}\right)\right\|_{\mathfrak{A}_k} = \left\|q(X_1,\ldots, X_n)\right\|_{\mathfrak{A}}$ for all $q \in \mathbb{C} \langle t_1, \ldots, t_n\rangle$, and
	\item $\lim_{k\to\infty} \tau_k\left(q\left(\left(X_1^{(k)}, \ldots, X_n^{(k)}\right)\right)\right) = \tau(q(X_1,\ldots,X_n))$ for all $q \in \mathbb{C} \langle t_1, \ldots, t_n\rangle$.
\end{enumerate}
Let $\mathfrak{D}$ be the unital C$^*$-subalgebra of $\prod_{k\geq 1}\mathfrak{A}_k$ generated by $\left\{\left(X_i^{(k)}\right)_{k\geq1}\right\}^n_{i=1}$ and let $\mathfrak{J} := \mathfrak{D} \cap \left(\bigoplus_{k\geq1} \mathfrak{A}_k\right)$.  Then $\mathfrak{J}$ is an ideal of $\mathfrak{D}$ such that $\mathfrak{D}/\mathfrak{J} \simeq \mathfrak{A}$.
\par
Let $\sigma: \mathfrak{B} \to \mathcal{B}(\mathcal{K})$ be the GNS representation of $\varphi$ with unit cyclic vector $\eta$, let $\pi_0 : \mathfrak{A} \to \mathcal{B}(\mathcal{H}_{0})$ be the GNS representation of $\tau$ with unit cyclic vector $\xi$, let $\pi_1 : \mathfrak{D} \to \mathcal{B}(\mathcal{H}_1)$ be a faithful, unital representation, let $q : \mathfrak{D} \to \mathfrak{A}$ be the canonical quotient map, and let $\pi := (\pi_0 \circ q) \oplus \pi_1 : \mathfrak{D} \to \mathcal{B}(\mathcal{H}_0 \oplus \mathcal{H}_1)$ which is a faithful, unital representation.  Then there exists an injective $^*$-homomorphism
\[
\Phi : \frac{(\mathfrak{D}, \pi, \xi) \ast (\mathfrak{B}, \sigma, \eta)}{\langle \mathfrak{J}\rangle_{\mathfrak{D} \ast \mathfrak{B}}} \to \frac{\prod_{k\geq 1} ((\mathfrak{A}_k, \tau_k) \ast (\mathfrak{B}, \varphi))}{\bigoplus_{k\geq 1} ((\mathfrak{A}_k, \tau_k) \ast (\mathfrak{B}, \varphi))}
\]
such that
\[
\Phi\left(\left(X_i^{(k)}\right)_{k\geq1} +  \langle \mathfrak{J}\rangle_{\mathfrak{D} \ast \mathfrak{B}} \right) = \left(X_i^{(k)}\right)_{k\geq1} + \bigoplus_{k\geq 1} ((\mathfrak{A}_k, \tau_k) \ast (\mathfrak{B}, \varphi))
\]
for all $i \in \{1,\ldots, n\}$ and
\[
\Phi(B + \langle \mathfrak{J}\rangle_{\mathfrak{D} \ast \mathfrak{B}}) = (B)_{k\geq1} + \bigoplus_{k\geq 1} ((\mathfrak{A}_k, \tau_k) \ast (\mathfrak{B}, \varphi))
\]
for all $B \in \mathfrak{B}$.
\end{cor2}
Recently in \cite{Pi2}, Pisier has developed a direct proof of Theorem \ref{ultraproduct} using the non-commutative Khintchine inequalities developed in \cite{RX}.  In fact, Pisier's result holds when the $Y$ variables are allowed to vary (see below).  We remark that the proof of Theorem \ref{ultraproduct} generalizes to this setting as well.
\begin{thm2}[Pisier]
For each $k \in \mathbb{N}$ let $\left\{X_i^{(k)}\right\}^n_{i=1}$ be generators for a non-commutative probability space $(\mathfrak{A}_k, \tau_k)$ and let $\left\{Y_i^{(k)}\right\}^m_{i=1}$ be generators for a non-commutative probability space $(\mathfrak{B}_k, \varphi_k)$.  Let $\{X_i\}^n_{i=1}$ be generators for a non-commutative probability space $(\mathfrak{A}, \tau)$ and let $\{Y_i\}^m_{i=1}$ be generators for a non-commutative probability space $(\mathfrak{B}, \varphi)$.   Suppose that 
\begin{enumerate}
	\item $\limsup_{k\to\infty}\left\|q\left(X_1^{(k)}, \ldots, X_n^{(k)}\right)\right\|_{\mathfrak{A}_k} = \left\|q(X_1,\ldots, X_n)\right\|_{\mathfrak{A}}$ for all $q \in \mathbb{C} \langle t_1, \ldots, t_n\rangle$, 
		\item $\limsup_{k\to\infty}\left\|q\left(Y_1^{(k)}, \ldots, Y_m^{(k)}\right)\right\|_{\mathfrak{B}_k} = \left\|q(Y_1,\ldots, Y_m)\right\|_{\mathfrak{B}}$ for all $q \in \mathbb{C} \langle t_1, \ldots, t_m\rangle$, 
	\item $\lim_{k\to\infty} \tau_k\left(q\left(\left(X_1^{(k)}, \ldots, X_n^{(k)}\right)\right)\right) = \tau(q(X_1,\ldots,X_n))$ for all $q \in \mathbb{C} \langle t_1, \ldots, t_n\rangle$, and
	\item $\lim_{k\to\infty} \varphi_k\left(q\left(\left(Y_1^{(k)}, \ldots, Y_m^{(k)}\right)\right)\right) = \varphi(q(Y_1,\ldots,Y_m))$ for all $q \in \mathbb{C} \langle t_1, \ldots, t_m\rangle$.
\end{enumerate}
Then
\[
\lim_{k\to\infty} \left\|p\left(X_1^{(k)}, \ldots, X^{(k)}_n, Y^{(k)}_1, \ldots, Y^{(k)}_m\right)\right\|_{(\mathfrak{A}_k, \tau_k) \ast (\mathfrak{B}_k, \varphi_k)} = \left\|p(X_1, \ldots, X_n, Y_1, \ldots, Y_m)\right\|_{(\mathfrak{A}, \tau) \ast (\mathfrak{B}, \varphi)}
\]
for all $p \in \mathbb{C} \langle t_1, \ldots, t_{n+m}\rangle$.
\end{thm2}
\begin{proof}
For each $k \in \mathbb{N}$ define the state $\psi_k : \mathfrak{A}_k \oplus \mathfrak{B}_k \to \mathbb{C}$ by $\psi_k(A \oplus B) = \frac{1}{2}(\tau_k(A) + \varphi_k(B))$ for all $A \in \mathfrak{A}_k$ and $B \in \mathfrak{B}_k$ and define the state $\psi : \mathfrak{A} \oplus \mathfrak{B} \to \mathbb{C}$ by $\psi(A \oplus B) = \frac{1}{2}(\tau(A) + \varphi(B))$ for all $A \in \mathfrak{A}$ and $B \in \mathfrak{B}$.  By the first assumption of the theorem, it is clear that
\[
\limsup_{k\to\infty}\left\|p\left(X_1^{(k)} \oplus 0, \ldots, X^{(k)}_n\oplus 0, 0 \oplus Y^{(k)}_1, \ldots, 0 \oplus Y^{(k)}_m\right)\right\|_{\mathfrak{A}_k \oplus \mathfrak{B}_k}
\]
is
\[
\left\|p(X_1\oplus 0, \ldots, X_n\oplus 0, 0 \oplus Y_1, \ldots, 0 \oplus Y_m)\right\|_{\mathfrak{A} \oplus \mathfrak{B}}
\]
for all $p \in \mathbb{C} \langle t_1, \ldots, t_{n+m}\rangle$.  Since
\[
\lim_{k\to\infty}\psi_k\left(p\left(X_1^{(k)} \oplus 0, \ldots, X^{(k)}_n\oplus 0, 0 \oplus Y^{(k)}_1, \ldots, 0 \oplus Y^{(k)}_m\right)\right)
\]
is
\[
\psi(p(X_1\oplus 0, \ldots, X_n\oplus 0, 0 \oplus Y_1, \ldots, 0 \oplus Y_m))
\]
for all $p \in \mathbb{C} \langle t_1, \ldots, t_{n+m}\rangle$, the proof now follows the remainder of the proof of Theorem \ref{ultraproduct}.
\end{proof}
We note that Lemma \ref{oneinequal} and Remarks \ref{3.3} generalize trivially to the above setting.

\section{Nuclearly Embeddings with States}
\label{sect:MR}

\subsection{Extending Completely Positive Maps with States}

In this section we will develop the ability to extend unital, completely positive maps on C$^*$-algebras in a state-preserving manner.  This will enable us to extend the classical result that the free product of unital, completely positive maps can be taken when the GNS representations of the C$^*$-algebras are faithful.  This will allow us to show $(\mathfrak{A}_1, \pi_1, \xi_1) \ast (\mathfrak{A}_2, \pi_2,\xi_2) \simeq (\mathfrak{A}_1, \pi'_1, \xi'_1) \ast (\mathfrak{A}_2, \pi'_2, \xi'_2)$ provided that $\pi_i$ and $\pi'_i$ are faithful, unital representations and the vector state defined by $\xi_i$ equals the vector state defined by $\xi'_i$ for all $i \in \{1,2\}$.  The ability to extend unital, completely positive maps in a state-preserving manner will be of use in the next section where the modification of the second equivalence of Theorem \ref{exactnessTheorem} to reduced free products is obtained.
\par
 First we state the commonly known result for unital, completely positive maps between the reduced free products of C$^*$-algebra with faithful GNS representations (where the reduced free product of more than two C$^*$-algebras can be taken recursively as the process is associative).
\begin{thm}[Blanchard, Dykema]
\label{oldfpocpm}
For $i\in \{1,\ldots, n\}$ let $\mathcal{H}_i$ and $\mathcal{H}'_i$ be Hilbert spaces and let $\xi_i \in \mathcal{H}_i$ and $\xi'_i \in \mathcal{H}'_i$ be unit vectors.  Suppose that there exists unital, completely positive $\psi_i : \mathcal{B}(\mathcal{H}'_i) \to \mathcal{B}(\mathcal{H}_i)$ such that 
\[
\langle \psi_i(T)\xi_i, \xi_i\rangle_{\mathcal{H}_i} = \langle T\xi'_i, \xi'_i\rangle_{\mathcal{H}'_i}
\]
for all $T \in \mathcal{B}(\mathcal{H}'_i)$. Then there exists a unital, completely positive map
\[
\Psi: \ast_{i=1}^n (\mathcal{B}(\mathcal{H}'_i), Id, \xi'_i) \to \ast_{i=1}^n (\mathcal{B}(\mathcal{H}_i), Id, \xi_i)
\]
such that $\Psi(T'_i) = \psi_i(T'_i)$ for all $T'_i \in \mathcal{B}(\mathcal{H}'_i)$ and
\[
\Psi(T'_{1}T'_{2} \cdots T'_{m}) = \Psi(T'_{1})\Psi(T'_{2})\cdots\Psi(T'_{m})
\]
whenever $m \geq 1$, $T'_{j} \in \mathcal{B}(\mathcal{H}'_{i_j})$, $\langle T'_j \xi'_{i_j}, \xi'_{i_j}\rangle_{\mathcal{H}'_{i_j}}=0$ for all $j \in \{1,\ldots, m\}$, and $i_j \neq i_{j+1}$ for all $j \in \{1,\ldots, m-1\}$.  We denote $\Psi$ by $\ast^n_{i=1} \psi_i$.
\end{thm}
\begin{proof}  
See Theorem 2.2 of \cite{BD}.
\end{proof}
\par
Next we endeavour to extend the above result to the reduced free products of C$^*$-algebras where we do not require that the states have faithful GNS constructions.  To do this, we will apply Theorem \ref{ecpmws} to the above result.  We begin the proof of Theorem \ref{ecpmws} in the case that $\varphi$ is a $^*$-homomorphism and then appeal to Stinespring's Theorem.
\begin{lem}
\label{extendlemma}
Let $\mathfrak{A} \subseteq \mathcal{B}(\mathcal{H})$ be a C$^*$-algebra, let $\mathcal{K}$ be a Hilbert space, let $\xi \in \mathcal{H}$ and $\eta \in \mathcal{K}$ be unit vectors, and let $\pi : \mathfrak{A} \to \mathcal{B}(\mathcal{K})$ be a $^*$-homomorphism.  Suppose further that $\xi \in \overline{\mathfrak{A}\mathcal{H}}$ and $\langle A\xi,\xi\rangle_\mathcal{H} = \langle \pi(A)\eta,\eta\rangle_\mathcal{K}$ for all $A \in \mathfrak{A}$.  Then there exists a contractive, completely positive map $\psi : \mathcal{B}(\mathcal{H}) \to \mathcal{B}(\mathcal{K})$ that extends $\pi$ such that $\langle T\xi,\xi\rangle_\mathcal{H} = \langle \psi(T)\eta, \eta\rangle_\mathcal{K}$ for all $T \in \mathcal{B}(\mathcal{H})$.  More specifically $\psi : \mathcal{B}(\mathcal{H}) \to \mathcal{B}(\mathcal{K}_0) \oplus \mathcal{B}(\mathcal{K}_0^\bot) \subseteq \mathcal{B}(\mathcal{K})$ where $\mathcal{K}_0 = \overline{\pi(\mathfrak{A})\eta}$.
\end{lem}
\begin{proof}
Let $\mathcal{H}_0 := \overline{\mathfrak{A}\xi}$ and let $\mathcal{K}_0 := \overline{\pi(\mathfrak{A})\eta}$.  Define $V_0 : \mathfrak{A}\xi \to \mathcal{K}_0$ by $V_0(A\xi) = \pi(A)\eta$ for all $A \in \mathfrak{A}$.  To see that $V_0$ is well-defined, notice
\[
\left\|\pi(A)\eta\right\|^2_{\mathcal{K}} = \langle \pi(A^*A)\eta,\eta\rangle_\mathcal{K} = \langle A^*A\xi,\xi\rangle_\mathcal{H} = \left\|A\xi\right\|^2_\mathcal{H}.
\]
Hence $V_0$ is well-defined and extends to a unitary map $V : \mathcal{H}_0 \to \mathcal{K}_0$.
\par
We claim $V\xi = \eta$.  To see this, recall if $(E_\lambda)_\Lambda$ is a C$^*$-bounded approximate identity for $\mathfrak{A}$ then $E_\lambda \xi$ converges to the projection of $\xi$ onto $\overline{\mathfrak{A}\xi}$ and similarly $\pi(E_\lambda)\eta$ converges to the projection of $\eta$ onto $\overline{\pi(\mathfrak{A})\eta}$.  Since $\xi \in \overline{\mathfrak{A}\xi}$, $\lim_\Lambda E_\lambda \xi = \xi$.  Therefore $V\xi = \lim_{\Lambda} V(E_\lambda\xi) = \lim_\Lambda \pi(E_\lambda)\eta$.  However, since $V$ is isometric, $\xi$ and $\eta$ are both unit vectors, and $\pi(E_\lambda)\eta$ converges to the projection of $\eta$ onto $\overline{\pi(\mathfrak{A})\eta}$, $\eta \in \overline{\pi(\mathfrak{A})\eta}$ and $V\xi = \eta$.
\par
Since $\mathcal{K}_0$ is a reducing subspace for $\pi(\mathfrak{A})$, there exists $^*$-homomorphisms $\pi_0 : \mathfrak{A}\to\mathcal{B}(\mathcal{K}_0)$ and $\pi^\bot_0 : \mathfrak{A} \to \mathcal{B}(\mathcal{K}_0^\bot)$ such that, with respect to the Hilbert space decomposition $\mathcal{K} = \mathcal{K}_0 \oplus \mathcal{K}_0^\bot$,
\[
\pi(A) = \left[  \begin{array}{cc} \pi_0(A) & 0\\ 0 & \pi^\bot_0(A)  \end{array} \right]
\]
for all $A \in \mathfrak{A}$.
\par
Let $P \in \mathcal{B}(\mathcal{H}_0, \mathcal{H})$ be the canonical inclusion.  By Arveson's Extension Theorem (see Theorem 7.5 of \cite{Pa}) $\pi_0^\bot$ extends to a contractive, completely positive map $\psi^\bot_0 : \mathcal{B}(\mathcal{H}) \to \mathcal{B}(\mathcal{K}_0^\bot)$.  Define $\psi : \mathcal{B}(\mathcal{H}) \to \mathcal{B}(\mathcal{K})$ by
\[
\psi(T) = \left[  \begin{array}{cc} VP^*TPV^* & 0\\ 0 & \psi^\bot_0(A)  \end{array} \right]
\]
(with respect to the Hilbert space decomposition $\mathcal{K} = \mathcal{K}_0 \oplus \mathcal{K}_0^\bot$) for all $T \in \mathcal{B}(\mathcal{H})$.  Clearly $\psi$ is a contractive, completely positive being the direct sum of contractive, completely positive maps.  To see $\psi(A) = \pi(A)$ for all $A \in \mathfrak{A}$ notice $\psi^\bot_0(A) = \pi^\bot_0(A)$ for all $A \in \mathfrak{A}$ by construction.  Moreover
\[
VP^*APV^*(\pi(A_0)\eta) = VP^*AP(A_0\xi) = VP^*((AA_0)\xi) = \pi(AA_0)\eta = \pi_0(A)(\pi(A_0)\eta)
\]
for all $A_0 \in \mathfrak{A}$.  Therefore, as $\pi(\mathfrak{A})\eta$ is dense in $\mathcal{K}_0$, $VP^*APV^* = \pi_0(A)$ for all $A \in \mathfrak{A}$ so $\psi$ extends $\pi$.
\par
Lastly
\[
\langle \psi(T)\eta, \eta\rangle_\mathcal{K} = \langle TPV^*\eta, PV^*\eta\rangle_\mathcal{H} = \langle TP\xi, P\xi\rangle_\mathcal{H} = \langle T\xi,\xi\rangle_\mathcal{H}
\]
for all $T \in \mathcal{B}(\mathcal{H})$ as desired.  
\end{proof}
\begin{thm}
\label{ecpmws}
Let $\mathfrak{A} \subseteq \mathcal{B}(\mathcal{H})$ be a unital C$^*$-algebra (not necessarily with the same unit as $\mathcal{B}(\mathcal{H})$), let $\mathcal{K}$ be a Hilbert space, let $\xi \in \mathcal{H}$ and $\eta \in \mathcal{K}$ be unit vectors, and let $\varphi : \mathfrak{A} \to \mathcal{B}(\mathcal{K})$ be a unital, completely positive map.  Suppose further that $\xi \in \overline{\mathfrak{A}\mathcal{H}}$ and $\langle A\xi,\xi\rangle_\mathcal{H} = \langle \varphi(A)\eta,\eta\rangle_\mathcal{K}$ for all $A \in \mathfrak{A}$.  Then there exists a unital, completely positive map $\psi : \mathcal{B}(\mathcal{H}) \to \mathcal{B}(\mathcal{K})$ that extends $\varphi$ such that $\langle T\xi,\xi\rangle_\mathcal{H} = \langle \psi(T)\eta, \eta\rangle_\mathcal{K}$ for all $T \in \mathcal{B}(\mathcal{H})$.
\end{thm}
\begin{proof}  By Stinespring's Representation Theorem (see Theorem 4.1 of \cite{Pa}) there exists a Hilbert space $\mathcal{K}'$, a unital $^*$-homomorphism $\pi : \mathfrak{A} \to \mathcal{B}(\mathcal{K}')$, and an isometry $V : \mathcal{K} \to \mathcal{K}'$ such that $\varphi(A) = V^*\pi(A)V$ for all $A \in \mathfrak{A}$.  Let $\eta' = V\eta \in \mathcal{K}'$.  Since $V$ is an isometry $\eta'\in\mathcal{K}'$ is a unit vector.  Moreover for all $A \in \mathfrak{A}$
\[
\langle \pi(A)\eta',\eta'\rangle_{\mathcal{K}'} = \langle V^*\pi(A)V\eta,\eta\rangle_\mathcal{K} = \langle\varphi(A)\eta,\eta\rangle_\mathcal{K} = \langle A\xi,\xi\rangle_\mathcal{H}.
\]
Therefore, by Lemma \ref{extendlemma}, there exists a contractive, completely positive map $\psi' : \mathcal{B}(\mathcal{H})\to\mathcal{B}(\mathcal{K}')$ extending $\pi$ such that $\langle \psi'(T)\eta',\eta'\rangle_{\mathcal{K}'} = \langle T\xi,\xi\rangle_\mathcal{H}$ for all $T \in \mathcal{B}(\mathcal{H})$. 
\par
Define $\psi : \mathcal{B}(\mathcal{H}) \to \mathcal{B}(\mathcal{K})$ by $\psi(T) = V^*\psi'(T)V$ for all $T \in \mathcal{B}(\mathcal{H})$. Clearly $\psi$ is a contractive, completely positive map being the composition of contractive, completely positive maps. Moreover 
\[
\psi(A) = V^*\psi'(A)V = V^*\pi(A)V = \varphi(A)
\]
for all $A \in \mathfrak{A}$ so $\psi$ extends $\varphi$.  Hence
\[
I_\mathcal{K} = \varphi(I_\mathfrak{A}) = \psi(I_\mathfrak{A}) \leq \psi(I_\mathcal{H}) \leq \left\|\psi\right\| I_\mathcal{K} \leq I_\mathcal{K}
\]
so $\psi(I_\mathcal{H}) = I_\mathcal{K}$.  Lastly
\[
\langle \psi(T)\eta,\eta\rangle_\mathcal{K} = \langle \psi'(T)V\eta, V\eta\rangle_{\mathcal{K}'} = \langle \psi'(T)\eta',\eta'\rangle_{\mathcal{K}'} =  \langle T\xi,\xi\rangle_\mathcal{H}
\]
for all $T \in \mathcal{B}(\mathcal{H})$.  
\end{proof}
By using Theorem \ref{ecpmws} we are able to extend Theorem \ref{oldfpocpm}.
\begin{thm}
\label{fpocpm}
For $i \in \{1,\ldots, n\}$ let $\mathfrak{A}'_i$ and $\mathfrak{A}_i$ be unital C$^*$-algebras, let $\pi'_i :\mathfrak{A}'_i \to \mathcal{B}(\mathcal{H}'_i)$ and $\pi_i :\mathfrak{A}_i \to \mathcal{B}(\mathcal{H}_i)$ be faithful, unital representations, and let $\xi'_i \in \mathcal{H}'_i$ and $\xi_i \in \mathcal{H}_i$ be unit vectors.  Suppose that there exists unital, completely positive maps $\varphi_i :\mathfrak{A}'_i \to \mathfrak{A}_i$ such that 
\[
\langle\pi_i(\varphi_i(A))\xi_i, \xi_i\rangle_{\mathcal{H}_i} = \langle \pi'_i(A)\xi'_i, \xi'_i\rangle_{\mathcal{H}'_i}
\]
for all $A \in \mathfrak{A}'_i$. Then there exists a unital, completely positive map
\[
\Phi: \ast_{i=1}^n (\mathfrak{A'}_i, \pi'_i, \xi'_i) \to \ast_{i=1}^n (\mathfrak{A}_i, \pi_i, \xi_i)
\]
such that $\Phi(A'_i) = \varphi_i(A'_i)$ for all $A'_i \in \mathfrak{A}'_i$ and
\[
\Phi(A'_1A'_2 \cdots A'_m) = \Phi(A'_1)\Phi(A'_2)\cdots\Phi(A'_m)
\]
whenever $m \geq 1$, $A'_j \in \mathfrak{A}'_{i_j}$, $\langle \pi'_{i_j}(A_j)\xi'_{i_j}, \xi'_{i_j}\rangle_{\mathcal{H}'_{i_j}}=0$ for all $j \in \{1,\ldots, m\}$, and $i_j \neq i_{j+1}$ for all $j \in \{1,\ldots, m-1\}$.  We denote $\Phi$ by $\ast^n_{i=1} \varphi_i$.
\end{thm}
\begin{proof}
By Theorem \ref{ecpmws} each $\varphi_i$ extends to a unital, completely positive map $\psi_i : \mathcal{B}(\mathcal{H}'_i) \to \mathcal{B}(\mathcal{H}_i)$ such that $\langle \psi_i(T)\xi_i, \xi_i\rangle_{\mathcal{H}_i} = \langle T\xi'_i, \xi'_i\rangle_{\mathcal{H}'_i}$ for all $T \in \mathcal{B}(\mathcal{H}'_i)$.  By Theorem \ref{oldfpocpm} there exists a unital, completely positive map
\[
\Psi: \ast_{i=1}^n (\mathcal{B}(\mathcal{H}'_i), Id, \xi'_i) \to \ast_{i=1}^n (\mathcal{B}(\mathcal{H}_i), Id, \xi_i)
\]
such that $\Psi(T'_i) = \psi_i(T'_i)$ for all $T'_i \in \mathcal{B}(\mathcal{H}'_i)$ and
\[
\Psi(T'_{1}T'_{2} \cdots T'_{m}) = \Psi(T'_{1})\Psi(T'_{2})\cdots\Psi(T'_{m})
\]
whenever $m \geq 1$, $T'_{j} \in \mathcal{B}(\mathcal{H}'_{i_j})$, $\langle T'_j \xi'_{i_j}, \xi'_{i_j}\rangle_{\mathcal{H}'_{i_j}}=0$ for all $j \in \{1,\ldots, m\}$, and $i_j \neq i_{j+1}$ for all $j \in \{1,\ldots, m-1\}$.  Since $\ast_{i=1}^n (\mathfrak{A'}_i, \pi'_i, \xi'_i) \subseteq \ast_{i=1}^n (\mathcal{B}(\mathcal{H}'_i), Id, \xi'_i)$, let $\Phi$ be the restriction of $\Psi$ to $\ast_{i=1}^n (\mathfrak{A'}_i, \pi'_i, \xi'_i)$.  Clearly $\Phi(A'_i) = \varphi_i(A'_i)$ for all $A'_i \in \mathfrak{A}'_i$ and
\[
\Phi(A'_1A'_2 \cdots A'_n) = \Phi(A'_1)\Phi(A'_2)\cdots\Phi(A'_m)
\]
whenever $m \geq 1$, $A'_j \in \mathfrak{A}'_{i_j}$, $\langle \pi'_{i_j}(A_j)\xi'_{i_j}, \xi'_{i_j}\rangle_{\mathcal{H}'_{i_j}}=0$ for all $j \in \{1,\ldots, m\}$, and $i_j \neq i_{j+1}$ for all $j \in \{1,\ldots, m-1\}$.
\par
To show that the image of $\Phi$ lies in $\ast_{i=1}^n (\mathfrak{A}_i, \pi_i, \xi_i)$, it suffices to show that $\Phi$ maps a dense subset of $\ast_{i=1}^n (\mathfrak{A}'_i, \pi'_i, \xi'_i)$ into $\ast_{i=1}^n (\mathfrak{A}_i, \pi_i, \xi_i)$.  Recall that $^*$-$Alg\left(\bigcup^n_{i=1} \mathfrak{A}'_i\right)$ is dense in $\ast_{i=1}^n (\mathfrak{A}'_i, \pi'_i, \xi'_i)$.  Moreover for all $j \in \{1,\ldots, n\}$ $\mathfrak{A}'_j = \mathbb{C} I_{\mathfrak{A}'_j} + (\mathfrak{A}'_j)^0$ where
\[
(\mathfrak{A}'_j)^0 := \left\{ A \in \mathfrak{A}'_j \, \mid \, \langle \pi'_{j}(A)\xi'_{j}, \xi'_{j}\rangle_{\mathcal{H}'_{j}}=0 \right\}.
\]
Therefore, since $I_{\mathfrak{A}'_j}$ is the unit of $\ast_{i=1}^n (\mathfrak{A}'_i, \pi'_i, \xi'_i)$ for all $j \in \{1,\ldots, n\}$, it is easy to see that the span of $I_{\ast_{i=1}^n (\mathfrak{A}'_i, \pi'_i, \xi'_i)}$ with
\[
\left\{ A'_{1} A'_{2} \cdots A'_{m} \, \mid \, m \geq 1, A'_j \in (\mathfrak{A}'_{i_j})^0, i_j \neq i_{j+1}  \right\}
\]
is dense in $\ast_{i=1}^n (\mathfrak{A}'_i, \pi'_i, \xi'_i)$.  Since $\Psi$ is unital and each $\pi'_i$ and $\pi_i$ is unital, 
\[
\Phi\left(I_{\ast_{i=1}^n (\mathfrak{A}'_i, \pi'_i, \xi'_i)}\right) = \Psi\left(I_{\ast_{i=1}^n (\mathcal{B}(\mathcal{H}'_i), Id, \xi'_i)}\right) = I_{\ast_{i=1}^n (\mathcal{B}(\mathcal{H}_i), Id, \xi_i)} = I_{\ast_{i=1}^n (\mathfrak{A}_i, \pi_i, \xi_i)} \in \ast_{i=1}^n (\mathfrak{A}_i, \pi_i, \xi_i).
\]
If $m \geq 1$, $A'_j \in (\mathfrak{A}'_{i_j})^0$, and $i_j \neq i_{j+1}$ for all $j \in \{1,\ldots, m-1\}$ then
\[
\Phi(A'_1A'_2 \cdots A'_m) = \Phi(A'_1)\Phi(A'_2)\cdots\Phi(A'_m) = \varphi_{i_1}(A'_1) \varphi_{i_2}(A'_2) \cdots \varphi_{i_m}(A'_m) \in \ast_{i=1}^n (\mathfrak{A}_i, \pi_i, \xi_i).
\]
Hence $\Phi$ maps a set whose span is dense in $\ast_{i=1}^n (\mathfrak{A}'_i, \pi'_i, \xi'_i)$ to a subset of $\ast_{i=1}^n (\mathfrak{A}_i, \pi_i, \xi_i)$.  Thus the result follows by the linearity and continuity of $\Phi$.  
\end{proof}
In \cite{BD} an example was given of unital C$^*$-algebras $\mathfrak{A}_i$, states $\varphi_i : \mathfrak{A}_i \to \mathbb{C}$ with faithful GNS representations, a unital C$^*$-algebra $\mathfrak{D}$ with a state $\psi$ that was not faithful but had a faithful GNS representation, and unital $^*$-homomorphism $\pi_i : \mathfrak{A}_i \to \mathfrak{D}$ such that the family $\{\pi_i(\mathfrak{A}_i)\}_{i \in \{1,2\}}$ was free with respect to $\psi$, $\psi \circ \pi_i = \varphi_i$ for all $i \in\{1,2\}$, and yet no $^*$-homomorphism $\pi : (\mathfrak{A}_1,\varphi_1) \ast (\mathfrak{A}_2,\varphi_2) \to \mathfrak{D}$ existed with the property that $\pi(A) = \pi_i(A)$ for all $A \in \mathfrak{A}_i$.  However, when $\mathfrak{D}$ is a C$^*$-algebra as described in Construction \ref{freeprod}, Theorem \ref{fpocpm} will allow us to construct such a $^*$-homomorphism.
\begin{thm}
If the unital, completely positive maps $\varphi_i : \mathfrak{A}'_i \to \mathfrak{A}_i$ in Theorem \ref{fpocpm} are $^*$-homomorphism, the resulting map $\Psi$ is a $^*$-homomorphism.
\end{thm}
\begin{proof}
Since $\Phi$ is a unital, completely positive map, it suffices to show that $\Phi$ is multiplicative on a set with dense span in $\ast_{i=1}^n (\mathfrak{A}'_i, \pi'_i, \xi'_i)$.  Recall the span of $I_{\ast_{i=1}^n (\mathfrak{A}'_i, \pi'_i, \xi'_i)}$ with
\[
\left\{ A_{1} A_{2} \cdots A_{m} \, \mid \, m \geq 1, A'_j \in (\mathfrak{A}'_{i_j})^0, i_j \neq i_{j+1}  \right\}
\]
is dense in $\ast_{i=1}^n (\mathfrak{A}'_i, \pi'_i, \xi'_i)$.  Hence, to show that $\Phi$ is multiplicative, it suffices to show that
\[
\Phi(A_{1} A_{2} \cdots A_{m}B_{1} B_{2} \cdots B_{k}  ) = \Phi(A_{1} A_{2} \cdots A_{m}) \Phi(B_{1} B_{2} \cdots B_{k})
\]
whenever $m,k \geq 1$, $A_j \in (\mathfrak{A}'_{i_j})^0$,  $B_j \in (\mathfrak{A}'_{i'_j})^0$, $i_j \neq i_{j+1}$ for all $j \in \{1,\ldots, m-1\}$, and $i'_j \neq i'_{j+1}$ for all $j \in \{1,\ldots, k-1\}$.  However if $i_m \neq i'_1$ then
\[
\Phi(A_{1} A_{2} \cdots A_{m}B_{1} B_{2} \cdots B_{k}  ) = \Phi(A_{1}) \Phi(A_{2}) \cdots \Phi(A_{m})\Phi(B_{1}) \Phi(B_{2}) \cdots \Phi(B_{k}  ) = \Phi(A_{1} A_{2} \cdots A_{m}) \Phi(B_{1} B_{2} \cdots B_{k})
\]
by the properties of $\Phi$.  Hence we may assume that $i_m = i'_1$.  To complete the proof, we will proceed by induction on $k$.
\par
If $k = 1$ then we can write $A_mB_1 = \lambda I_{\mathfrak{A}'_{i_m}} + C$ where $\lambda \in \mathbb{C}$ and $C \in (\mathfrak{A}'_{i_m})^0$.  Thus
\begin{eqnarray} 
\Phi(A_{1} A_{2} \cdots A_{m-1}A_{m}B_{1}  )\!\!\!\!&=&\!\!\!\! \lambda \Phi(A_{1} A_{2} \cdots A_{m-1}) + \Phi(A_{1} A_{2} \cdots A_{m-1}C)\nonumber\\
&=&\!\!\!\! \lambda\Phi(A_{1}) \Phi(A_{2}) \cdots \Phi(A_{m-1}) + \Phi(A_{1}) \Phi(A_{2}) \cdots \Phi(A_{m-1})\Phi(C) \nonumber\\
&=&\!\!\!\!  \Phi(A_{1}) \Phi(A_{2}) \cdots \Phi(A_{m-1})(\lambda I + \Phi(C)) \nonumber\\
&=&\!\!\!\!  \Phi(A_{1}) \Phi(A_{2}) \cdots \Phi(A_{m-1})\varphi_{i_m}(A_mB_1) \nonumber\\
&=&\!\!\!\!  \Phi(A_{1}) \Phi(A_{2}) \cdots \Phi(A_{m-1})\varphi_{i_m}(A_m)\varphi_{i_m}(B_1) \nonumber\\
&=&\!\!\!\!  \Phi(A_{1} A_{2} \cdots A_{m}) \Phi(B_{1})\nonumber
\end{eqnarray}
by the properties of $\Phi$ and the fact that $\varphi_{i_m}$ is a $^*$-homomorphism on $\mathfrak{A}'_{i_m}$.  Thus the base case is complete.
\par
Suppose the result holds for some $k \geq 1$.  Write $A_mB_1 = \lambda I_{\mathfrak{A}'_{i_m}} + C$ where $\lambda \in \mathbb{C}$ and $C \in (\mathfrak{A}'_{i_m})^0$.  Thus
\begin{eqnarray} 
\!\!\!\!&&\!\!\!\! \Phi(A_{1} A_{2} \cdots A_{m-1}A_{m}B_{1} B_2 \cdots B_{k+1}  )\nonumber\\
&=&\!\!\!\! \lambda \Phi(A_{1} A_{2} \cdots A_{m-1}B_2 \cdots B_{k+1}  ) + \Phi(A_{1} A_{2} \cdots A_{m-1} C B_2 \cdots B_{k+1}  )\nonumber\\
&=&\!\!\!\! \lambda \Phi(A_{1} A_{2} \cdots A_{m-1})\Phi(B_2 \cdots B_{k+1}  ) + \Phi(A_{1}) \Phi(A_{2}) \cdots \Phi(A_{m-1}) \Phi(C) \Phi(B_2) \cdots \Phi(B_{k+1}) \nonumber\\
&=&\!\!\!\! \lambda \Phi(A_{1}) \Phi(A_{2}) \cdots \Phi(A_{m-1})\Phi(B_2) \cdots \Phi(B_{k+1}) + \Phi(A_{1}) \Phi(A_{2}) \cdots \Phi(A_{m-1}) \Phi(C) \Phi(B_2) \cdots \Phi(B_{k+1}) \nonumber\\
&=&\!\!\!\! \Phi(A_{1}) \Phi(A_{2}) \cdots \Phi(A_{m-1})(\lambda I + \Phi(C))\Phi(B_2) \cdots \Phi(B_{k+1})\nonumber\\
&=&\!\!\!\! \Phi(A_{1}) \Phi(A_{2}) \cdots \Phi(A_{m-1})(\varphi_{i_m}(A_mB_1))\Phi(B_2) \cdots \Phi(B_{k+1})\nonumber\\
&=&\!\!\!\! \Phi(A_{1}) \Phi(A_{2}) \cdots \Phi(A_{m-1})\varphi_{i_m}(A_m)\varphi_{i_m}(B_1)\Phi(B_2) \cdots \Phi(B_{k+1})\nonumber\\
&=&\!\!\!\! \Phi(A_{1} A_{2} \cdots A_{m}) \Phi(B_{1} B_{2} \cdots B_{k+1})\nonumber
\end{eqnarray}
by the properties of $\Phi$ and the fact that $\varphi_{i_m}$ is a $^*$-homomorphism on $\mathfrak{A}'_{i_m}$.  Hence, by the Principle of Mathematical Induction, the proof is complete.
\end{proof}
\begin{cor}
For $i \in \{1,\ldots, n\}$ let $\mathfrak{A}_i$ be a unital C$^*$-algebra, let $\pi_i : \mathfrak{A}_i \to \mathcal{B}(\mathcal{H}_i)$ and let $\pi'_i : \mathfrak{A}_i \to \mathcal{B}(\mathcal{H}'_i)$ be faithful, unital representations, and let $\xi_i \in \mathcal{H}_i$ and $\xi'_i \in \mathcal{H}'_i$ be unit vectors.  Suppose 
\[
\langle \pi_i(A)\xi_i,\xi_i\rangle_{\mathcal{H}_i} = \langle \pi'_i(A)\xi'_i, \xi'_i\rangle_{\mathcal{H}'_i}
\]
for all $A \in \mathfrak{A}_i$ and all $i \in \{1,\ldots, n\}$.  Then there exists a $^*$-isomorphism
\[
\Phi : \ast_{i=1}^n (\mathfrak{A}_i, \pi'_i, \xi'_i) \to \ast_{i=1}^n (\mathfrak{A}_i, \pi_i, \xi_i)
\]
such that $\Phi(A) = A$ for all $A \in \mathfrak{A}'_i$.
\end{cor}

\subsection{Nuclear Embeddings with States}

Before we obtained Theorem \ref{main} and Theorem \ref{ultraproduct}, one natural guess could have been that these theorems would hold only if there existed a state-preserving nuclear embedding into $\mathcal{B}(\mathcal{H})$.  Thus it is natural to ask if such an embedding always exists.  A similar but stricter property has already been studied in free probability.
\begin{defn}[Eckhardt]
\label{2.7.1}
Let $\mathfrak{A}$ be a unital, nuclear C$^*$-algebra and let $\varphi$ a state on $\mathfrak{A}$.  We say that $\varphi$ is CP-approximable if there exists a net $(T_\lambda)_\Lambda$ of finite-rank, unital, completely positive maps on $\mathfrak{A}$ such that $T_\lambda$ converges to the identity pointwise and $\varphi \circ T_\lambda = \varphi$ for all $\lambda \in \Lambda$.  
\end{defn}
Some notatble results relating to this definition are:
\begin{prop}[Proposition 4.5 of \cite{RX}]
\label{2.7.2}
Let $\mathfrak{A}$ be a nuclear, unital C$^*$-algebra and $\varphi$ a faithful state on $\mathfrak{A}^{**}$.  Then there is a net of finite rank, unital, completely positive maps $(T_\lambda : \mathfrak{A} \to \mathfrak{A})_\Lambda$ converging pointwise to the identity in norm such that $\varphi \circ T_\lambda = \varphi$ for all $\lambda \in \Lambda$.
\end{prop}
\begin{prop}[Corollary 4.3 of \cite{Ec}]
\label{2.7.3}
Let $\mathfrak{A}$ be the CAR algebra.  Then there is a state $\psi$ on $\mathfrak{A}$ that is not CP-approximable.
\end{prop}
However by relaxing the conditions of CP-approximability, we obtain the following result.
\begin{thm}
\label{2.7.11}
Let $\mathfrak{A}$ be a unital C$^*$-algebra.  Then the following are equivalent:
\begin{enumerate}
	\item $\mathfrak{A}$ is exact.
		\item If $\varphi$ is a state on $\mathfrak{A}$ then for every Hilbert space $\mathcal{K}$, every faithful, unital representation $\sigma : \mathfrak{A} \to \mathcal{B}(\mathcal{K})$, and every unit vector $\xi \in \mathcal{K}$ such that $\varphi(A) = \langle \sigma(A) \xi, \xi\rangle_\mathcal{K}$ for all $A \in \mathfrak{A}$ there exists a net of matrix algebras $(\mathcal{M}_{n_\lambda}(\mathbb{C}))_\Lambda$, nets of unital, completely positive maps $(\phi_\lambda : \mathfrak{A} \to \mathcal{M}_{n_\lambda}(\mathbb{C}))_\Lambda$ and $(\psi_\lambda : \mathcal{M}_{n_\lambda}(\mathbb{C}) \to \mathcal{B}(\mathcal{K}))_\Lambda$, and unit vectors $(\xi_\lambda \in \mathbb{C}^{n_\lambda})_\Lambda$ such that  $\langle\phi_\lambda(A) \xi_\lambda, \xi_\lambda \rangle_{\mathbb{C}^{n_\lambda}} = \varphi(A)$ for all $A \in \mathfrak{A}$, $\langle \psi_\lambda(T)\xi, \xi\rangle_\mathcal{K} = \langle T \xi_\lambda, \xi_\lambda \rangle_{\mathbb{C}^{n_\lambda}}$ for all $T \in \mathcal{M}_{n_\lambda}(\mathbb{C})$, and $\lim_\Lambda \left\|\sigma(A) - \psi_\lambda(\phi_\lambda(A))\right\| = 0$ for all $A \in \mathfrak{A}$.
\end{enumerate}
\end{thm}
The condition that $\mathfrak{A}$ is unital is trivial to remove in the above theorem (with the replacement of unital, completely positive maps with contractive, completely positive maps).  Moreover it can be easily shown using Theorem \ref{ecpmws} that it suffices to prove the above for one fixed representation $\sigma : \mathfrak{A} \to \mathcal{B}(\mathcal{K})$ and unit vector $\xi \in \mathcal{K}$ such that $\varphi(A) = \langle \sigma(A) \xi, \xi\rangle_\mathcal{K}$ for all $A \in \mathfrak{A}$.  Similarly it can be show that if the matrix algebras with vector states are replaced with finite dimensional C$^*$-algebras with arbitrary states in the above definition, then the two statements are equivalent. Although the above result is essentially implied in Lemma 2.4 of \cite{Oz}, we include a proof for completeness.  We begin with a simple lemma.
\begin{lem}
\label{2.7.9}
Let $\mathfrak{A}$ be a unital, nuclear C$^*$-algebra and let $\varphi$ be a pure state on $\mathfrak{A}$.  Then there exists nets of unital, completely positive maps $(\phi_\lambda :\mathfrak{A} \to \mathcal{M}_{n_\lambda}(\mathbb{C}))_\lambda$ and $(\psi_\lambda : \mathcal{M}_{n_\lambda}(\mathbb{C}) \to \mathfrak{A})_\Lambda$ and unit vectors $(\xi_\lambda \in \mathbb{C}^{n_\lambda})_\Lambda$ such that the net $(\psi_\lambda \circ \phi_\lambda)_\Lambda$ converges to the identity on $\mathfrak{A}$ in the point-norm topology, $\langle \phi_\lambda(A) \xi_\lambda, \xi_\lambda\rangle_{\mathbb{C}^{n_\lambda}} = \varphi(A)$ for all $A \in \mathfrak{A}$ and all $\lambda \in \Lambda$, and $\varphi(\psi_n(T)) = \langle T\xi_\lambda, \xi_\lambda\rangle_{\mathbb{C}^{n_\lambda}}$ for all $T \in \mathcal{M}_{n_\lambda}(\mathbb{C})$ and all $\lambda \in \Lambda$.  Thus $\varphi$ is CP-approximable.
\end{lem}
\begin{proof}
Let $(\pi, \mathcal{H}, \xi)$ be the GNS representation of $\mathfrak{A}$ given by $\varphi$ and let $\mathcal{K} = \mathbb{C} \xi$.  Since $\varphi$ is a pure state, $\pi$ is an irreducible representation.  By Lemma 3.4 of \cite{KS} (or see Lemma 4.8.6 in \cite{BO}) there exists nets of unital, completely positive maps $(\phi_\lambda :\mathfrak{A} \to \mathcal{M}_{n_\lambda}(\mathbb{C}))_\Lambda$ and $(\psi_\lambda : \mathcal{M}_{n_\lambda}(\mathbb{C}) \to \mathfrak{A})_\Lambda$ and isometries $(V_\lambda : \mathcal{K} \to \mathbb{C}^{n_\lambda})_\Lambda$ such that the net $(\psi_\lambda \circ \phi_\lambda)_\Lambda$ converges to the identity on $\mathfrak{A}$ in the point-norm topology, $V_\lambda^*\phi_\lambda(A)V_\lambda = P_\mathcal{K} \pi(A) P_\mathcal{K}$ for all $A \in \mathfrak{A}$ and all $\lambda \in \Lambda$, and $V_\lambda \pi(\psi_\lambda(T))V_\lambda^* = V_\lambda V_\lambda^* TV_\lambda V_\lambda^*$ for all $T \in \mathcal{M}_{n_\lambda}(\mathbb{C})$ and all $\lambda \in \Lambda$.  For each $\lambda \in \Lambda$ let $\xi_\lambda = V_\lambda \xi \in \mathbb{C}^{n_\lambda}$.  Then each $\xi_\lambda$ is a unit vector and $V_\lambda V_\lambda^*\xi_\lambda = \xi_\lambda$.  Moreover
\[
\langle \phi_\lambda(A) \xi_\lambda, \xi_\lambda\rangle_{\mathbb{C}^{n_\lambda}} = \langle V_\lambda^*\phi_\lambda(A)V_\lambda \xi, \xi\rangle_{\mathcal{H}} = \langle P_\mathcal{K} \pi(A) P_\mathcal{K} \xi, \xi\rangle_{\mathcal{H}} = \langle \pi(A)\xi, \xi\rangle_{\mathcal{H}} = \varphi(A)
\]
for all $A \in \mathfrak{A}$ and 
\[
\varphi(\psi_\lambda(T)) = \langle \pi(\psi_\lambda(T)) \xi,\xi\rangle_{\mathcal{K}} = \langle V_\lambda \pi(\psi_\lambda(T))V_\lambda^*\xi_\lambda,\xi_\lambda\rangle_{\mathbb{C}^{n_\lambda}}= \langle V_\lambda V_\lambda^* TV_\lambda V_\lambda^* \xi_\lambda,\xi_\lambda\rangle_{\mathbb{C}^{n_\lambda}}= \langle  T \xi_\lambda,\xi_\lambda\rangle_{\mathbb{C}^{n_\lambda}}
\]
for all $T \in \mathcal{M}_{n_\lambda}(\mathbb{C})$.  
\end{proof}
Note that Lemma \ref{2.7.9} is related to Proposition \ref{2.7.2}.  Moreover Proposition \ref{2.7.3} says we cannot drop the assumption of $\varphi$ being a pure state. 
\begin{proof}[Proof of Theorem \ref{2.7.11}]
As every unital, exact C$^*$-algebra embeds into a unital, nuclear C$^*$-algebra (specifically the Cuntz algebra $\mathcal{O}_2$ by Theorem 2.8 of \cite{KP} in the separable case), it is trivial to verify that $\mathfrak{A}$ may be assumed to be nuclear.
\par
Let $\pi : \mathfrak{A} \to \mathcal{B}(\mathcal{H}_u)$ be the universal representation of $\mathfrak{A}$.  Therefore there exists a unit vector $\xi \in \mathcal{H}_u$ such that $\varphi(A) = \langle \pi(A)\xi, \xi\rangle$ for all $A \in \mathfrak{A}$.  Let $\mathfrak{K}$ be the compact operators in $\mathcal{B}(\mathcal{H}_u)$ and let $\mathfrak{B} = \pi(\mathfrak{A}) + \mathfrak{K}$ which is a C$^*$-algebra containing $\mathfrak{A}$ and $\mathfrak{K}$.
\par
We claim that $\mathfrak{B}$ is nuclear.  To see this notice
\[
0 \rightarrow \mathfrak{K} \rightarrow \mathfrak{B} \rightarrow \mathfrak{B}/\mathfrak{K} \rightarrow0
\]
is an exact sequence of C$^*$-algebra.  Since $\mathfrak{K}$ is nuclear, $\mathfrak{B}$ will be nuclear provided that $\mathfrak{B}/\mathfrak{K}$ is nuclear (see Proposition 10.1.3 of \cite{BO}).  However, by the Second Isomorphism Theorem of C$^*$-Algebras,
\[
\mathfrak{B}/\mathfrak{K} = (\pi(\mathfrak{A}) + \mathfrak{K})/\mathfrak{K} \simeq \pi(\mathfrak{A})/(\pi(\mathfrak{A}) \cap \mathfrak{K}).
\]
Since $\mathfrak{A}$ is nuclear and $\pi(\mathfrak{A}) \cap \mathfrak{K}$ is an ideal in $\pi(\mathfrak{A})$, $\pi(\mathfrak{A})/(\pi(\mathfrak{A}) \cap \mathfrak{K})$ is nuclear (see Corollary 9.4.4 of \cite{BO}).  Hence $\mathfrak{B}$ is nuclear.
\par
To show that $\varphi$ is weakly CP-approximable it suffices to show that the vector state $\psi : \mathfrak{B} \to\mathbb{C}$ defined by $\psi(T) = \langle T\xi,\xi\rangle$ for all $T \in \mathfrak{B}$ is weakly CP-approximable.   Let $\mathcal{K} := \overline{\mathfrak{B}\xi}$.  Then $\mathcal{K}$ is an invariant subspace of $\mathfrak{B}$ with cyclic vector $\xi$.  By the uniqueness of the GNS construction, the restriction of $\mathfrak{B}$ to $\mathcal{K}$ is the GNS representation of $\psi$.  However, since $\mathfrak{B}$ contains the compact operators, $\mathcal{K}$ has no non-trivial $\mathfrak{B}$-invariant subspaces.  Hence the GNS representation of $\psi$ is irreducible and thus $\psi$ is a pure state on $\mathfrak{B}$ (see Theorem I.9.8 of \cite{Da}).  Thus $\psi$ is weakly CP-approximable by Lemma \ref{2.7.9} and hence $\varphi$ is weakly CP-approximable.  
\end{proof}

\section{Open Questions}
\label{sect:Open Questions}

In this section we will brief discuss three questions pertaining to the material presented in this paper along with their difficulties.  Our first question is whether or not Theorem \ref{ecpmws} can be extended to operator systems.
\newline
\newline
{\bf Question 1)}  \textit{Let $\mathfrak{A}$ be a unital C$^*$-algebra, let $\mathcal{S}$ be an operator system, let $\varphi : \mathfrak{A} \to \mathbb{C}$ be a state, let $\phi :\mathcal{S} \to \mathcal{B}(\mathcal{H})$ be a unital, completely positive map, and let $\xi \in \mathcal{H}$ be a unit vector such that $\varphi(A) = \langle \phi(A)\xi,\xi\rangle_{\mathcal{H}}$ for all $A \in \mathcal{S}$.  Does there exists a unital, completely positive map $\psi : \mathfrak{A} \to \mathcal{B}(\mathcal{H})$ extending $\phi$ such that $\varphi(A) = \langle \psi(A)\xi,\xi\rangle_{\mathcal{H}}$ for all $A \in \mathfrak{A}$?}
\newline
\par
Clearly the proof of Theorem \ref{ecpmws} cannot be modified to solve the above question.
\newline
\newline
{\bf Question 2)}  \textit{Clearly Theorem \ref{main} and Theorem \ref{ultraproduct} are equivalent statements.  Can this be seen directly as in the tensor product case?}
\newline
\par
In order to use Theorem \ref{main} to prove Theorem \ref{ultraproduct}, it would suffice to prove Corollary \ref{freeImbed} directly.  However, due to the differences in the structures of the norms of the objects in Corollary \ref{freeImbed}, it appears difficult to directly prove such a map exists.  
\newline
\newline
{\bf Question 3)}  \textit{The concepts of Theorem \ref{main} and Theorem \ref{ultraproduct} can be generalized to free products with amalgamation.  Do these theorems still hold in this more general setting?}
\newline
\par
The only issue in the proof of Theorem \ref{main} given above when applied to reduced free products with amalgamation appears to be in the inductive step of Lemma \ref{complemma} where the Gram-Schmidt Orthogonalization process was used to approximate the norm of an operator by the norm of a matrix of operators.

\section{Acknowledgments and References}
\label{sect:AcknowledgementsReferences}

I would like to thank Professor Dimitri Shlyakhtenko for informing me of this problem and for his various ideas and advice pertaining to this problem.

\textsc{Department of Mathematics, UCLA, Los Angeles, California, 90095-1555, USA}
\par
\textit{E-mail address}: pskoufra@math.ucla.edu

\end{document}